\documentclass[12pt]{amsart}
\usepackage{amssymb}
\usepackage{amsfonts}
\usepackage{amscd}
\usepackage[T1]{fontenc}

\usepackage{xy}
\input xy
\xyoption{all}
\usepackage{pdflscape}
\usepackage{hyperref}

\usepackage[english]{babel}

\def\A{\mathbb{A}}

\def\B{\mathbb{B}}

\def\I{\mathbb{I}}

\def\R{\mathbb{R}}
\def\C{\mathbb{C}}
\def\F{\mathbb{F}}

\def\P{\mathbb{P}}
\def\S{\mathcal{S}}

\let\myacute=\'

\def\emptyset{\varnothing}

\def\<{\langle}
\def\>{\rangle}
\def\N{\mathbb{N}}
\def\Z{\mathbb{Z}}

\def \Qp {\mathbb{Q}_p}

\def\cL{\mathcal{L}}

\def \begindm {\begin{displaymath}}

\def \enddm {\end{displaymath}}

\def\C{\mathbb{C}}
\def\R{\mathbb{R}}
\def\Q{\mathbb{Q}}
\def\F{\mathbb{F}}

\def\cL{\mathcal{L}}
\def\cP{\mathcal{P}}
\def\cM{\mathcal M}
\def\cO{\mathcal{O}}

\newtheorem{thm}{Theorem}[section]

\newtheorem{cor}{Corollary}[section]
\newtheorem{prop}{Proposition}[section]

\newtheorem{Conjecture}{Conjecture}[section]
\newtheorem{Def}{Definition}[section]
\newtheorem{ex}{Example}[section]
\newtheorem{remark}{Remark}[section]
\newtheorem{remarks}{Remarks}[section]

\numberwithin{equation}{section}
\newtheorem{prob}{Problem}[section]
\newtheorem{note}{Note}[section]

\long\def\symbolfootnote[#1]#2{\begingroup\def\thefootnote{\fnsymbol{footnote}}\footnote[#1]{#2}\endgroup}

\title{Model Theory of Adeles and Number Theory}
\author[J. Derakhshan]{Jamshid Derakhshan}
\address{St Hilda's College, University of Oxford, Cowley Place, Oxford OX4 1DY, UK}
\email{derakhsh@maths.ox.ac.uk}

\begin{document}



\begin{abstract} 
This paper is a survey on model theory of adeles and applications to model theory, algebra, and 
number theory. Sections 1-12 concern model theory of adeles and the results are joint works with Angus Macintyre. The topics covered include quantifier elimination in enriched Boolean algebras, quantifier elimination in restricted products and in adeles and adele spaces of algebraic varieties in natural languages, definable subsets of adeles and their measures, solution to a problem of Ax from 1968 on decidability of the rings $\Z/m\Z$ for all $m>1$, definable sets of minimal idempotents (or  "primes of the number field" ) in the adeles, stability-theoretic notions of stable embedding and tree property of the second kind, elementary equivalence and isomorphism for adele rings, axioms for rings elementarily equivalent to restricted products and for the adeles, converse to Feferman-Vaught theorems, a language for adeles relevant for Hilbert symbols in number theory, imaginaries in adeles, and the space adele classes. 

Sections 13-18 are concerned with connections to number theory around zeta integrals and 
$L$-functions. Inspired by our model theory of adeles, I propose a model-theoretic approach to automorphic forms on $GL_1$ (Tate's thesis) and 
$GL_2$ (work of Jacquet-Langlands), and formulate several notions, problems and questions. The main idea is to formulate notions of constructible adelic integrals  and observe that the integrals of Tate and Jacquet-Langlands are constructible. These constructible integrals are related to the $p$-adic and motivic integrals in model theory. 

This is a first step in a model-theoretic approach to the Langlands conjectures. I also formulate questions on 
Artin's reciprocity for ideles, $p$-adic fields and adeles with additive and multiplicative characters in continuous logic, adelic Poisson summation and transfer principles, identities of adelic integrals, and questions related to model-theoretic aspects of Archimedean integrals (which relate to O-minimality).

\end{abstract}

\maketitle

\tableofcontents

\section{\bf Introduction}\label{sec-introduction}

\medskip

\subsection{\bf History of adeles and role in number theory} 

\

\medskip

Given an algebraic number field $K$ (i.e. a finite extension of $\Q$) and valuation $v$ of $K$, 
let $K_v$ denote the completion of $K$ with respect to the absolute value corresponding to $v$. 
The adeles of $K$, denoted $\A_K$ is the 
subset of the direct product of $K_v$, over all $v$, consisting of elements $(a(v))\in \prod_v K_v$ such that $a(v)$ lies in the ring of integers $\mathcal{O}_{v}$ of $K_v$ (i.e. has non-negative valuation) for all but finitely many $v$. $\A_K$ is a locally compact 
ring (with componentwise addition and multiplication) and the image of $K$ diagonally embedded is a discrete subspace with compact quotient $\A_K/K$. Being locally compact, $\A_K$ admits an invariant measure. 

The ring of adeles can be defined for any global field, i.e. a number field of function field in one variable over a finite field, see \cite{CF}. In this paper we only consider the case of number fields. It is expected that similar results hold also for global fields of positive characteristic. 

The ring of adeles was first introduced by Weil in a letter to Hasse on 3 November 1937. It was for the case of function fields and as "a method to establish (a new proof of) the theorem of Riemann-Roch". Around this time, Chevalley defined the notion of ideles, also in a letter to Hasse on 20 June 1935, and initially called them "\'{e}lements id\'{e}aux" (Ideles are invertible adeles, and form a group denoted $\I_K$). For more details on the early history of adeles see Peter Roquette's book \cite{roquette-book}, page 191 and Subsection 11.3.1.

The adeles were later independently introduced by Artin and Whaples in 1945 based on the ideles. The name adele was invented by Weil as "additive adele" and already used by him in a 1958 Bourbaki talk on "adeles and algebraic groups " 
on results that later appeared in \cite{weil-adeles-gps} systematically studying Tamagawa measures of adelic algebraic groups 
and proving for many classical groups what is now called Weil conjecture on Tamagawa numbers, a highly influential work and conjecture. 

Weil's main motivation for defining adeles was a proof of the Riemann hypothesis for function fields of curves over finite fields. It was clear from the beginning that adeles have rich connections with $L$-functions and zeta functions of number fields involving characters. Interestingly most of the developments and applications of adeles have been in continuation of this path. I shall give a short account of of some key developments below, and later a model-theoretic approach to these problems. 

Dedekind had defined a generalization of the Riemann zeta function for any number field and Hecke proved analytic continuation and functional equation for these zeta functions. 

At the suggestion of Artin, in 1950, in \cite{tate-thesis} Tate used harmonic analysis on the adeles to give a new proof of Hecke's results. Tate's proof gave more information. In 1952 independently Iwasawa obtained the same results. 

Tate's work has greatly influenced several works in number theory. This includes the Langlands program, starting with work of  Jacquet and Langlands \cite{jac-lang} on $GL_2$ and then Godement and Jacquet \cite{jaq-good} on $GL_n$, and later work by Langlands and others, see \cite{lang-prob}. These have given applications of adeles to automorphic forms. The classical theory of modular forms and work of Hecke was generalized adelically. These have given rise to a large number of results within the Langlands program on automorphic representations of adelic groups $G(\A_K)$, where $G$ is an algebraic group over $K$. The case of $GL_1$ being Tate's thesis, $GL_2$ the Jacquet-Langlands work which extends results on classical modular forms and Maass forms, and the case of $GL_n$ being work of Godement and Jacquet \cite{jaq-good}. 
The case of a general reductive group involves the concept of Langlands dual group. For later developments, especially relating to $l$-adic representations of elliptic curves and algebraic varieties see \cite{manin-book}. 

Another important development was 
Weil's adelic interpretation of work of Siegel on quadratic forms (called Siegel Mass formula) as the adelic volume of $SO_{n}(\A_{\Q})/SO_{n}(\Q)$ with respect to the Tamagawa measure being $2$ for all $n$.  
Weil's conjecture on Tamagawa numbers states that the volume of $G(\A_K)/G(K)$ is $1$ for any semi-simple simply connected linear algebraic group $G$ over a number field $K$. This was proved by Weil for many classical groups, Langlands for split semi-simple groups, and Kottwitz in general, (see Kneser's article in the volume \cite{CF} and the book \cite{manin-book}). 
The conjecture of Birch and Swinnerton-Dyer emerges from an analogue of Weil's conjecture for elliptic curves. The results in Subsection 4.4 and Section 10 relate to these issues. 

An insight of Paul Cohen (unpublished notes with Peter Sarnak) was that properties of $\A_K/K^*$, defined as the quotient of the space of adele classes $\A_K$ by the action of 
$K^*$ by multiplication, are relevant for the the solution of the Riemann Hypothesis. Much work in this direction has been carried out by Alain Connes and co-authors. See \cite{connes-selecta} where results from Tate's thesis are generalized to adele classes to obtain results on zeros of zeta functions. The results in Section 12 and Sections 18-20 are related to these ideas and concepts.

\medskip

\subsection{\bf Model theory} 

\

\medskip

The first work on the model theory of adeles was by Weispfenning \cite{weisp-hab}. He proved decidability of the ring of adeles of number fields via his work on model theory for lattice products. In 2005 Angus Macintyre told me of his ideas and insights on a model theory for adeles. This was the start of our long term collaboration. I shall give a survey of many of our results in Sections 2-12. 

Like Weispfenning, we use the celebrated work of Feferman and Vaught \cite{FV} on products of structures but our approach is  different. For us the Boolean algebra of subsets of the index set of the product is replaced by the Boolean algebra of idempotents in $\A_K$. We also prove analogues of results of Feferman-Vaught on products for the case of restricted products. The interpretation of the Boolean algebra in the ring of adeles enables studying model theory of $\A_K$ as an $\cL_{rings}$-structure which has been needed to understand issues related to the topology, measure, and analytic structure on the adeles. It enables one to internalize the Feferman-Vaught theorems and obtain quantifier elimination results in the language of rings and prove that the definable subsets of $\A_K^m$, for any $m\geq 1$, are measurable (with respect to a Haar measure on $\A_K$) and study their measures. We also present the Feferman-Vaught theorems for restricted products in a more general case of many-sorted languages with relation and function symbols. We can also combine the ring-theoretic approach to $\A_K$ with the Feferman-Vaught type structures involving the Boolean algebra of subsets of the index set of normalized valuations 
and consider expansions of the Boolean language and expansions of the language of rings. Any choice for these expansions yields a language for the adeles which is an example of what we call a language for restricted products.

The approach to $\A_K$ as a ring via its idempotents has enabled much further results. These include solution to a problem on Ax on decidability of $\Z/m\Z$ for all $m$, study of elementary equivalence of adele rings, and converse of Feferman-Vaught theorems which amount to Feferman-Vaught theorems for rings. In this connection we have given axioms for adeles and 
general restricted products. 

\medskip

{\it Outline of the Paper}

\medskip

The material in Sections 2-12 are {\it all joint works with Angus Macintyre}.

In Section 2, I state the basic results and notation on adeles and idempotents. In Section 3, I state results on enriched Boolean algebras and enriched valued fields that are needed. 

Section 4 contains the model-theoretic formalism for restricted products and their languages and the 
quantifier elimination results for restricted products and adeles. 

Section 5 contains results on definable sets and their measurability, and connections to values of zeta functions at integers, and a description of the definable subsets of the set of minimal idempotents in the adeles (analogue of primes) in terms of Ax's Boolean algebra of primes in \cite{ax}.

Section 6 contains a sketch the solution to a problem posed in 1968 by Ax \cite{ax} on decidability of the rings $\Z/m\Z$ for all $m>1$, via decidability of $\A_{\Q}$, of which a short proof is also sketched. 

Section 7 is concerned with results on elementary equivalence of adele rings and their isomorphism, and its relation to 
arithmetical equivalence of number fields and their zeta functions.

Section 8 is concerned with axioms for restricted products, and a converse to Feferman-Vaught. 
The main result states that any commutative unital ring satisfying the axioms is elementarily equivalent to a restricted product of connected rings. This is proved via a Feferman-Vaught theorem for general commutative rings.

Section 9 is concerned with the stability-theoretic notions of stable embedding and tree property of the second kind for the adeles.

Section 10 contains results on adele spaces of algebraic varieties introduced by Weil. The results state that these spaces can be put into the formalism of Section 4 
and are model-theoretic restricted products and admit quantifier elimination in a certain geometric language. Some results are obtained on Tamagawa volumes and uniform definability of convergence factors for adelic measures on adele spaces for certain varieties.

In Section 11 an expansion of the language of rings is considered which is suitable for studying 
reciprocity laws in number theory and definability results on Hilbert symbols using results on certain enrichments of Boolean algebras introduced in \cite{DM-bool}.

Section 12 concerns imaginaries in adeles, the space of adele classes $\A_{\Q}/\Q^*$, and the quotient $\hat{\Z}^*\setminus\A_{\Q}/\Q^*$, where $\hat{\Z}^*$ the maximal compact subgroup of the idele class group $\I_K/K^*$ which 
appears in various works by Alain Connes and Connes-Consani on the distribution of zeros of the Riemann zeta function and on the arithmetic site topos, see \cite{connes-selecta},\cite{CC2},\cite{connes-c-site}. 
The result stated proves that the double quotient is interpretable in $\A_{\Q}$ answering a question of Zilber.

In Sections 12-18, I introduce a model theoretic approach to certain topics in number theory and formulate several questions and problems.

Section 13 concerns Artin reciprocity from class field theory, the idele class group, and a number of problems.

In Section 14 I state a general result of myself \cite{zeta1} on analytic properties of Euler products, over all $p$, of $p$-adic integrals (suitably normalized) introduced and studied by Denef, and later Denef, Loeser and others in motivic integration. This result shows that it is possible to analytically continue the Euler product beyond its abscissa of convergence and give some information on its poles which is enough to give an asymptotic formula for the partial sums of the coefficients of the Dirichlet series the represents the Euler product. 

This result naturally applies to several global zeta functions that have the form of an Euler product of $p$-adic integrals. An application is stated in \cite{zeta1} and \cite{zeta-surv} to a question of Uri Onn on counting 
counting conjugacy classes in congruence quotients of algebraic groups over number fields. 

I state a special case of this result for $SL_n$, $n\geq 2$. Let $c_m$ denote the number of conjugacy classes in the congruence quotient $SL_n(\Z/m\Z)$. 
Then the global conjugacy class zeta function $\sum_{i\geq 1} c_m m^{-s}$ admits meromorphic continuation beyond its abscissa of convergence, and consequently 
$$c_1+\dots+c_m \sim cN^{\alpha}$$ for some $c\in \R_{>0}$ as $N\rightarrow \infty$. This is proved by writing the global conjugacy zeta function as an Euler product of local conjugacy class zeta functions which count conjugacy classes in congruence quotients of the group of $p$-adic points and by my joint 
work with Mark Berman, Uri Onn, and Pirita Paajanen \cite{BDOP} the local conjugacy class zeta functions can be written as
$p$-adic integrals of Denef type over definable sets.

I also state a connection of the result on meromorphic continuation of Euler products 
to counting rational points in orbits of group actions and a problem of Gorodnik and Oh in \cite{GO} that is studied in \cite{zeta2}.

Section 15 is concerned with finite fields with additive character in continuous logic and a question of Hrushovski is stated.

Section 16 is concerned with $p$-adic fields with additive character in continuous logic and consequences for the model theory of adeles with additive characters in usual first-order and in continuous logic.

Section 17 concerns ideles with multiplicative character and associated $L$-functions. 

In Section 18 I discuss Tate's thesis on analytic continuation and functional equation for a large class of zeta integrals generalizing zeta functions of number fields, and formulate model-theoretic questions.

Section 19 is concerned with automorphic representations and the Langlands conjectures. Having considered the basic case of $GL_1$ in Section 18 on Tate's thesis, I focus on automorphic representations of $GL_2$. I propose a framework to model-theoretically study these issues and formulate a number of notions, problems, and questions in relation to model-theoretic aspects of the adelic integrals involved. The ideas on model theory of adeles presented in Sections 2-12 naturally connect with these problems, and it is hoped that the notions and questions can be a first step towards a systematic investigation on Langlands correspondence. 

In this regard, I define notions of $\cL$-constructible adelic integrals and $\C$-valued $\cL$-constructible functions on the adeles, where $\cL$ is a suitable language. These integrals are mostly Euler products over primes, and their local $p$-adic factors have some similarities to the integrals of motivic constructible functions of Cluckers-Denef-Loeser in \cite{CL2}. 

A proof of Jacquet and Langlands shows that the 
Jacquet-Langlands global zeta integrals of cuspidal automorphic representations of $GL_2$ are $\cL$-constructible of a certain kind (I call it of Whittaker type). This is stated and several related problems are formulated.

These are hoped to be a starting point for systematic model-theoretic investigations on the Langlands program beginning with the works of Jacquet-Langlands and Godement-Jacquet on $GL_n$. 
A model theory for modular forms and automorphic forms could follow. 

The Archimedean completions of number fields need to be considered as well in relation to the $p$-adic completions. Related to this, I have stated questions involving the expansion of $(\R,+,.,0,1)$ by 
restricted real analytic functions with exponentiation introduced by 
van den Dries-Macintyre-Marker \cite{VMM}, and O-minimality in connection to Hodge theory a la Bakker-Klingler-Tsimerman \cite{BKT}. 

In Section 20, I  formulate questions around Tate's adelic Poisson summation formula, adelic transfer principles guided by the Ax-Kochen transfer principle for truth of sentences and Cluckers-Denef-Loeser-Macintyre transfer principles 
for identities of local integrals across 
families of local fields, and a completeness problem for identities between adelic integrals that relates to axioms for adeles.

\medskip

\subsection{\bf Acknowledgments}

\

\medskip

I am very grateful and indebted to Angus Macintyre for introducing me to the subject, and for generously and patiently teaching me and sharing with me his ideas and insights that have led to the results in this paper, and for our collaborations and his continuous inspiration and support.

Many thanks to Ehud Hrushovski and Boris Zilber for many discussions and advice, from which I have learned much. Udi introduced me to some of his questions and results that are stated in Section 16, which influenced Sections 15-16.

I am grateful to Peter Sarnak and Nicolas Templier for discussions and asking some questions that I have stated in Section 13, and for their advice.

I am grateful to Ramin Takloo-Bighash for helpful discussions on topics in number theory, and on work of Jacquet-Langlands that were helpful in my understanding of the topics.

Many thanks to Laurent Lafforgue for many discussions on aspects of the Langlands conjectures, and for teaching me some of his insights and recent works on functoriality, especially during a visit to IHES. These have influenced Sections 18-20.

\section{\bf Adeles, idempotents, and Boolean values}\label{sec-basic}

\medskip

\subsection{\bf Global fields}\label{}

\

\medskip

A global field $K$ is either a finite extension of $\Q$ (called a number field) or a finite separable extension of $\F_q(t)$ where $t$ is transcendental over $\F_q$ (called a function field). 

In this paper, we shall only be concerned with number fields. See Cassels' article \cite{Cassels} for basic results and notions around absolute values and valuations on a field. Here we only state that two absolute values $|.|_1$ and 
$|.|_2$ are called equivalent if $|.|_1=|.|_2^c$ for some $c\in \R>0$, and this holds exactly when they define the same topology. We consider absolute values on a number field $K$ up to this equivalence. These absolute values are then of only the following kinds:

(1) Discrete non-Archimedean with residue field finite of cardinality $q$,

(2) Completion of $|.|$ is $\R$,

(3) Completion of $|.|$ is $\C$.

Among these absolute values we chose a distinguished one in each kind, that is called {\it normalized}.  In Case (1) $|.|$ is normalized if $|\pi|=1/q$ where $\pi$ is a uniformizing element (i.e. $v(\pi)$ is a minimal positive element) of the value group of $K$ (see below as well). In Case (2) (resp. Case (3)), $|.|$ is normalized if it is the usual absolute value (resp. square of the usual absolute value).

We denote by $K_v$ the completion of $K$ with respect to $v$. For details on global fields see Cassels' paper \cite{Cassels}. 

In the non-Archimedean cases, the absolute values are given by valuations $v(x)$ which are maps from $K^*$ into an 
ordered abelian group $\Gamma$ called the value group. One extends $v(x)$ to $K$ by putting $v(0)=\infty$. 

The non-Archimedean absolute values $|.|_p$ on $K=\Q$ arise from the standard $p$-adic valuations $v_p$ as 
$|a/b|_p=p^{-v_p(a/b)}$, where $v_p(z)=k$ if $k$ is the largest power of $p$ that divides $z$ when $z\in \Z$, and 
$v_p(a/b)=v_p(a)-v_p(b)$, for $a,b\in \Z$. Thus the absolute values of $\Q$ correspond to primes of $\Z$ and 
the absolute value given by the embedding of $\Q$ in $\R$.

For a general number field $K$, one replaces the prime numbers $p$ by prime ideals $\frak{p}$ of the ring of integers $\cO_K$ of $K$, and similarly the absolute values correspond to these and the real as well as complex embeddings of $K$. Similarly one has a valuation $v_{\frak{p}}$ of $K$ (which we again normalize) which restricts to $(p)$, and any $(p)$ extends to finitely many $\frak{p}$. Furthermore, the real absolute value extends to finitely many real and complex absolute values.

We denote by $V_K^{fin}$ the set of all normalized valuations on $K$ which give a non-Archimedean absolute value 
up to equivalence. We denote by $\mathrm{Arch}(K)$ the finite set of Archimedean normalized absolute values, and will call them {\it Archimedean valuations}. We put  
$V_K=V_K^{fin} \cup \mathrm{Arch(K)}$, and call it the set of {\it normalized valuations} of $K$.

$\Gamma$ (or $\Gamma_K$) shall denote the value group, $\cO_K:=\{x: |x|\leq 1\}$  
the valuation ring, and $\cM_K:=\{x: |x|< 1\}$ the maximal ideal of a valued field $K$.

For a non-Archimedean completion $K_v$ of $K$, we also denote 
the absolute value of $K_v$ by $|.|_v$, the valuation ring by $\mathcal{O}_v$, the maximal ideal by $\cM_v$, and 
the residue field by $k_v$.

\subsection{\bf Restricted direct products and measures}\label{ssec-rest-prod}
Let $\Lambda$ be an index set and $\Lambda_{\infty}$ a fixed finite 
subset. Suppose we are given, for each $\lambda\in \Lambda$, a topological space  
$G_{\lambda}$, and for all $\lambda\notin \Lambda_{\infty}$, 
a fixed open subset $H_{\lambda}$ of $G_{\lambda}$. 
Let $\prod_{\lambda \in \Lambda} G_{\lambda}$ denote the Cartesian product of the $G_{\lambda}$. We denote an element of this set by $x=(x(\lambda))_{\lambda}$, where 
$x(\lambda)$ denotes the $\lambda$-component of $x$ which is an element of $G_{\lambda}$.

The 
restricted direct product $G$ of $G_{\lambda}$ with respect to $H_{\lambda}$ is defined to be the set of all elements 
$(x(\lambda))_{\lambda}\in \prod_{\lambda \in \Lambda} G_{\lambda}$ such that $x(\lambda)\in H_{\lambda}$ for 
all but finitely many $\lambda$, and denoted $\prod'_{\lambda\in \Lambda} G_{\lambda}$. 
It carries the restricted product topology with a basis of open sets consisting of the products $\prod_{\lambda \in \Lambda} \Gamma_{\lambda}$,
where $\Gamma_{\lambda}\subseteq G_{\lambda}$ is open for all $\lambda\in \Lambda$, and 
$\Gamma_{\lambda}=H_{\lambda}$ for all but finitely many $\lambda\in \Lambda$.

Suppose for each $\lambda\in \Lambda$, $G_{\lambda}$ carries a measure $\mu_{\lambda}$ such that $\mu_{\lambda}(H_{\lambda})=1$ for all $\lambda\notin \Lambda_{\infty}$. 
The measure $\mu$ on 
$G$ induced by the $\mu_{\lambda}$ is the measure with a basis of measurable sets consisting of the sets 
$\prod_{\lambda\in \Lambda} M_{\lambda}$, where $M_{\lambda}\subseteq G_{\lambda}$ is $\mu_{\lambda}$-measurable and 
$M_{\lambda}=H_{\lambda}$ for all but finitely many $\lambda\in \Lambda$, and where $$\mu(\prod_{\lambda\in \Lambda} M_{\lambda})=\prod_{\lambda\in \Lambda} 
\mu_{\lambda}(M_{\lambda}).$$
It is denoted by $\prod_{\lambda} \mu_{\lambda}$.

Suppose $G_{\lambda}$ is a locally compact group for all $\lambda \in \Lambda$, and $H_{\lambda}$ a compact open subgroup 
for all $\lambda \notin \Lambda_{\infty}$. Then $G_{\lambda}$ carries a Haar measure $\mu_{\lambda}$ such that 
$\mu_{\lambda}(H_{\lambda})=1$ (for all $\lambda\notin \Lambda_{\infty}$). For any finite subset $S\subseteq \Lambda$ containing $\Lambda_{\infty}$, let 
$$G_S:=\prod_{\lambda\in S} G_{\lambda}\times \prod_{\lambda\notin S} H_{\lambda}.$$ 
This set is locally compact and open in $G$, and 
$G$ is the union of the $G_S$ over all finite subsets $S$ of $\Lambda$ containing $\Lambda_{\infty}$. In particular, $G$ is locally compact.

\medskip

\subsection{\bf Ring of adeles}\label{ssec-adeles} 

\

\medskip

The ring of adeles of a number field $K$ 
is the restricted direct product $\Bbb A_K=\prod'_{v\in V_K} K_v$
of the additive groups of $K_v$ with respect to the subgroups $\mathcal{O}_v$, with 
addition and multiplication defined componentwise. We write an adele $a$ as $(a(v))_{v}$. 

The ring of finite adeles of $K$ is the restricted direct product $\A_{K}^{fin}=\prod'_{v\in V_K^{fin}} K_v$
of all non-Archimedean $K_v$ with respect to $\mathcal{O}_v$. One has 
$$\Bbb A_K=(\prod_{v\in \mathrm{Arch(K)}}K_v) \times \Bbb A_{K}^{fin},$$
an algebraic and topological isomorphism.

Let $S$ be a finite subset of $V_K$ containing all the Archimedean valuations. We put 
$$\A_{K,S}=\prod_{v\in S} K_v \times \prod_{v\notin S} \mathcal{O}_v.$$
Then $\A_K=\bigcup_{S} \A_{K,S}$, over all finite subsets $S\subseteq V_K$ containing $V_K^{Arch}$.

There is an embedding of $K$ into $\Bbb A_K$ sending $a\in K$ to $(a,a,\cdots)$. The image is the ring of principal 
adeles, which we identify with $K$. It is a discrete subspace of $\Bbb A_K$ with compact quotient $\A_K/K$. 
If $K\subseteq L$ are number fields, then 
$\A_L$ is isomorphic to $\A_K \otimes_{K} L$ algebraically and topologically. See \cite{CF}, page 64.

The group of ideles $\Bbb I_K$ is the group of units of $\Bbb A_K$ and coincides with the restricted direct product of the multiplicative groups $K_v^*$ with respect to 
the unit groups $\cO_v^*$ of $\mathcal{O}_v$. It is given the restricted direct product topology. For any $a\in K^*$, $|a|_v=1$ for all except finitely many $v$. Thus 
$K^*$ is (diagonally) embedded in $\I_K$. One also has the product formula $\prod_{v\in V_K} |a|_v=1$. See \cite[pp.60]{CF}.

Each $K_v$ is a locally compact field, and carries an additive Haar measure $dx_v$. We make the following choice for these Haar measures:

(1) The measure on $K_v$ such that $\mathcal{O}_v$ has volume $1$ if $K_v$ is non-Archimedean

(2) The usual Lebesgue measure if $K_v$ is $\R$,

(3) The measure $2dxdy$ if $K_v$ is $\C$.

These give measures $\prod_{v\in V_K} dx_v$ and $\prod_{v\in V_K^{fin}} dx_v$ 
on $\A_K$ and $\A_K^{fin}$ respectively. The multiplicative groups  
$K_v^*$ carry a Haar 
measure $d^*x_v=dx_v/|x_v|_v$ invariant under multiplication, which 
give the measure $\prod_{v\in V_K} d^*x_v$ on 
$\Bbb I_K$.

As in Weil \cite{weil-adeles-gps}, we call 
the induced measure $\prod_{v\in V_K} dx_v$ on $\A_K$ the canonical measure and denote it by $\omega_{\A_K}$. 

\medskip

\subsection{\bf Idempotents and Boolean values}\label{ssec-idem} 

\

\medskip

Let $\cL_{\rm{rings}}=\{+,-,.,0,1\}$ denote the language of rings and $\cL_{Boolean}=\{\wedge,\vee,0,1,\neg\}$ the language of Boolean algebras. The set 
$$\B_K=\{a\in \A_K: a^2=a\}$$
of idempotents in $\A_K$ is a Boolean algebra 
with the Boolean operations
$$e\wedge f=ef,$$ 
$$e\vee f=1-(1-e)(1-f)=e+f-ef,$$ 
$$\neg e=1-e.$$ 
$\B_K$ is $\cL_{rings}$-definable in $\Bbb A_K$. It carries an order defined by $e\leq f$ if and only if $e=ef$, which is $\cL_{rings}$-definable.
An idempotent $e$ is minimal if it is non-zero and minimal with respect to this order. 

There is a correspondence between subsets of $V_K$ and 
idempotents $e$ in $\Bbb A_K$ defined by 
$$X \longmapsto e_X,$$
where given $X$,  
$$e_X=\begin{cases}
1 & \text{if~$v\in X$}\\
0 & \text{if~$v\notin X$}
\end{cases}$$
It is clear that $e_X\in \Bbb A_K$. Conversely, 
if $e\in \Bbb A_K$ is idempotent, then we let 
$$X=\{v: e(v)=1\}.$$ We have $e=e_{X}.$

Under this correspondence a minimal idempotent $e$ corresponds to a normalized valuation $\{v\}$ which we denote by $v_e$. Conversely,
$v$ corresponds to $e_{\{v\}}$. 

Note that $e\A_K$ is an ideal in $\A_K$ but not a unital subring of $\A_K$. It is a unital ring with induced addition and multiplication, and $e$ as the unit element. 

The map
$$a+(1-e)\A_K \mapsto ea$$
gives an isomorphism 
$$\Bbb A_K/(1-e)\Bbb A_K \cong e\Bbb A_K.$$
This follows from the isomorphism $\A_K\cong e\A_K \oplus (1-e)\A_K$ since $e$ is idempotent. 

If $e$ is a minimal idempotent, and corresponds to $\{v_e\}$, we have an isomorphism 
$$e\Bbb A_K\cong K_{v_e}$$
given by $$ea \mapsto a(v_{e}).$$

Let $\Phi(x_1,\dots,x_n)$ be an $\cL_{rings}$-formula. It is easily seen (see \cite{DM-ad}) that there is an associated $\cL_{rings}$-formula $\Phi^{Glob}(y,x_1,\dots,x_n)$ 
constructed from $\Phi$ and independent of $K$ such that for all $a_1,\dots,a_n \in \A_K$ and idempotents $e$
$$e\A_K \models \Phi(ea_1,\dots,ea_n) \Leftrightarrow \A_K\models \Phi^{Glob}(e,a_1,\dots,a_n).$$

For $a_1,\dots,a_n\in \Bbb A_K$, we define $[[\Phi(a_1,\dots,a_n)]]$ to be the supremum of all the minimal idempotents $e$ in $\Bbb B_K$ such that 
$$e\Bbb A_K \models \Phi(ea_1,\dots,ea_n).$$
Note that the set of such minimal idempotents is empty precisely when   
$$[[\Phi(a_1,\dots,a_n)]]=0,$$
and the set of such idempotents coincides the set of all minimal idempotents precisely when 
$$[[\Phi(a_1,\dots,a_n)]]=1.$$
We remark that $[[...]]$ is an internal version of the Boolean values of Feferman-Vaught (see \cite{FV}) defined by 
$$[[\Phi(x_1,\dots,x_n)]]=\{v\in V_K: K_v\models \Phi(a_1(v),\dots,a_n(v))\}.$$
Note that $e\Bbb A_K$ is $\cL_{rings}$-definable with the parameter $e$, and the functions 
$$\Bbb A_K^n\rightarrow \Bbb A_K$$
defined by 
$$(a_1,\dots,a_n)\rightarrow [[\Phi(a_1,\dots,a_n)]]$$ 
are $\cL_{rings}$-definable independently of $K$, for any $\Phi$. 

The idempotent-support of $a\in \Bbb A_K$ is defined as 
the Boolean value $[[a\neq 0]]$, and we denote it by $\mathrm{supp}(a)$

\medskip

\subsection{\bf Real, complex, and non-Archimedean Boolean values}\label{ssec-arch}

\

\medskip

We define the following formulas and idempotents.

. $\Psi^{\mathrm{Arch}}$ denotes a sentence that holds in $\R$ and $\C$ but does not hold in any 
non-Archimedean local field, for example 
$$\forall x \exists y (x=y^2 \vee -x=y^2).$$ 

. $\Psi_{\R}$ denotes a sentence that holds precisely in the $K_v$ which are isomorphic to $\R$, 
for example 
$$\Psi^{\mathrm{Arch}}\wedge \neg \exists y (y^2=-1).$$ 

. $\Psi_{\C}$ denotes a sentence that holds precisely in the $K_v$ which are isomorphic to $\C$, 
for example $$\Psi^{\mathrm{Arch}}\wedge \exists y (y^2=-1).$$ 

. A minimal idempotent $e$ is Archimedean if $$e\A_K\models \Psi^{\mathrm{Arch}},$$
and non-Archimedean otherwise. 

. A minimal idempotent $e$ is real if $e\A_K\models \Psi_{\R}$, and complex if $e\A_K\models \Psi_{\C}$.

. $e_{\R} $(resp.\ $e_{\C}$) denotes the supremum of all the real (resp.~complex) minimal idempotents. It is supported precisely on the set of $v$ such that $K_v$ is real  (resp. complex).

. $e_{\infty}$ denotes the supremum of all the Archimedean minimal idempotents. It is supported precisely on the set $\mathrm{Arch}(K)$.

. $e_{na}$ denotes $1-e_{\infty}$. It is supported precisely on the set of non-Archimedean valuations. 

Let $\Phi(x_1,\dots,x_n)$ be an $\cL_{\rm{rings}}$-formula and $a_1,\dots,a_n\in \Bbb A_K$. Then 

. $[[\Phi(a_1,\dots,a_n)]]^{real}$ denotes the supremum of all the minimal idempotents $e$ 
such that 
$$e\Bbb A_K\models \Psi_{\R} \wedge \Phi(ea_1,\dots,ea_n).$$

. $[[\Phi(a_1,\dots,a_n)]]^{complex}$ denotes the supremum of all the minimal idempotents $e$ 
such that $$e\Bbb A_K\models \Psi_{\C} \wedge \Phi(ea_1,\dots,ea_n).$$

. $[[\Phi(a_1,\dots,a_n)]]^{na}$ denotes the supremum of all the minimal idempotents $e$ such that 
$$e\Bbb A_K\models \neg \Psi^{\mathrm{Arch}} \wedge \Phi(ea_1,\cdots,ea_n).$$ 

The functions $$(a_1,\dots,a_n)\rightarrow [[\Phi(a_1,\dots,a_n)]]^{real}$$ 
$$(a_1,\dots,a_n)\rightarrow [[\Phi(a_1,\dots,a_n)]]^{complex}$$ 
and $$(a_1,\dots,a_n)\rightarrow [[\Phi(a_1,\dots,a_n)]]^{na}$$ 
are $\cL_{\rm{rings}}$-definable from $\A_K^n$ to $\A_K$, uniformly in $K$.

\medskip

\subsection{\bf Finite adeles}\label{sec-fin} 

\

\medskip

We can identify the ring of finite adeles $\A_K^{fin}$ with the ideal in $\A_K$ consisting of all adeles $a$ such that $a(v)=0$ for all $v\in Arch(K)$ (this will be the case in the results 
on von Neumann regularity and definition of valuation in Subsection 
\ref{ssec-deffin}). However in the results of Section \ref{sec-rest} on quantifier elimination and definable sets it is preferable to work with $\A_K^{fin}$ as a restricted direct product and $\A_K$ as the product of $\A_K^{fin}$ with $\prod_{v\in Arch(K)} K_v$. 

We let $\B_K^{fin}$ denote the Boolean algebra of idempotents in $\A_K^{fin}$ (considered either as a restricted product or as an ideal in $\A_K$) 
with the same Boolean operations as $\B_K$. It is definable in $\B_K$, $\A_K^{fin}$, and $\A_K$.

Given $a_1,\cdots,a_n\in \A_{K}^{fin}$, and an $\cL_{\rm{rings}}$-formula $\Phi(x_1,\cdots,x_n)$, the idempotent 
$[[\Phi(a_1,\cdots,a_n)]]^{na}$ can be defined inside $\A_K^{fin}$ as the supremum of all the 
minimal idempotents $e$ from $\B_K^{fin}$ such that 
$$e\A_{K}^{fin} \models \Phi(ea_1,\dots,ea_n),$$ 
which is again definable in $\cL_{rings}$, uniformly in $K$.

\medskip

\subsection{\bf Defining finite support idempotents}\label{ssec-deffin} 

\

\medskip

An idempotent $e$ has finite support if $\{v\in V_K: e(v)\neq 0\}$ is a finite set. 

The following theorem is of great importance in our works on the model theory of adeles.

\begin{thm}[Derakhshan-Macintyre {\cite{DM-ad}}] 
\label{defid}The set of finite support idempotents in $\A_K$ is $\cL_{rings}$-definable, uniformly in $K$.\end{thm} 
We denote the set of finite support idempotents by $Fin_K$. It is an ideal in the Boolean algebra $\Bbb B_K$. In 
\cite{DM-ad} we give two proofs of this result. The first uses the concept of a von Neumann regular ring defined as follows. 

We call an element $a$ of a commutative ring $R$ von Neumann regular if $(a)$  the ideal generated by 
$a$ is generated by an idempotent. $R$ is von Neumann regular if every element is 
von Neumann regular. Examples are direct products of fields. In contrast, $\A_K$ is not von Neumann regular, however it is shown in \cite{DM-ad} that an element $a\in \Bbb A_K$ is von Neumann regular if and only if the set 
$$\{v\in V_K: v(a(v))>0 \wedge a(v)\neq 0\}$$ is finite. This implies that an idempotent $e\in \A_K$  
has finite support if and only if for all $a\in \Bbb A_K$, the element $ae$ is von Neumann regular.

So we let $\phi(x)$ denote the $\cL_{rings}$-formula 
$$\exists e \exists y \exists z (e^2=e \wedge x=ey \wedge e=xz).$$ 
Then the set $Fin_K$ of finite support idempotents in $\A_K$ is defined by the $\forall \exists$-formula
$$x=x^2 \wedge \forall y (\phi(yx)).$$
This definition is independent of $K$. It follows that the set of finite support idempotents in $\A_K^{fin}$ is also definable. 

The second proof uses uniform definition of valuation rings in the language of rings which used quite often in the works of Macintyre and myself on the model theory of adeles. 
Explicit definitions previous to our work (e.g. for $\Q_p$ for all $p$) used a constant in the language for a 
uniformizing element. Non-explicit definitions had existed as well, using Beth definability type arguments because of the uniqueness of the Henselian valuations of the $K_v$. These results were mostly 
applied in connection to decidability and undecidability results on Hilbert 10th problem (see \cite{Koenigsmann}). 

In joint work with Cluckers-Leenknegt-Macintyre \cite{CDLM} the following is proved.
\begin{thm}[Cluckers-Derakhshan-Leenknegt-Macintyre {\cite[Theorem 2]{CDLM}}]
\label{CDLM-th} 
There is an existential-universal formula in the language of rings that uniformly defines the valuation ring of all Henselian valued fields with finite or pseudo-finite residue field.
\end{thm}

This theorem is best possible in terms of quantifier complexity as we show in \cite{CDLM} that there is no existential or universal $\cL_{rings}$-formula that uniformly defines the valuation rings of $\Q_p$, or $\F_p((t))$ for almost all $p$, or the valuation rings of 
all the finite extensions of a given $\F_p((t))$ or $\Q_p$. 

We denote by $\Phi_{val}(x)$ the formula given by Theorem \ref{CDLM-th}. For our work on adeles we apply Theorem \ref{CDLM-th} to all the completions $K_v$, where $v$ is non-Archimedean.
\begin{note} Note that Theorem \ref{CDLM-th} applies not only to all finite extensions of $\Q_p$ for all $p$, but to all 
local fields of positive characteristic $\F_q((t))$ and their ultraproducts. In fact, the formula $\Phi_v(x)$ is defining the valuation ring of all these fields as well as the valuation rings of all Henselian valued fields with higher rank value groups and pseudofinite residue field. 
\end{note}

In \cite{DM-ad} that an idempotent $e$ has finite support if and only if 
$$\Bbb A_K\models\exists x (e_{na}e=[[\neg\Phi_{val}(x)]]^{na})$$ 
which is also independent of $K$.

\section{\bf Enriched theories of Boolean algebras and valued fields}\label{sec-bool}

Our analysis of the model theory of the adeles and restricted products depends crucially on the 
theory of an associated Boolean algebra and the theories of the local fields which are the factors or "stalks" of the adeles. 

In this analysis, various enrichments are of fundamental importance. The Boolean algebra should be enriched by at least adding a predicate $Fin(x)$ to the language 
with intended interpretation in a Boolean algebras of sets as "finite", and in the Boolean algebra of idempotents in the adeles or in rings 
as "finite unions of atoms". We should also enrich the language for the factors. We give some details below.

\medskip

\subsection{\bf Enrichments of Boolean algebras}\label{ssec-benrich} 

\

\medskip

As before $\mathcal{L}_{Boolean}$ denotes the language of Boolean algebras $\{\vee,\wedge,\neg,0,1\}$. Tarski proved the classical result that if we enrich $\mathcal{L}_{Boolean}$ by natural predicates $C_n(x)$ for all $n\geq 1$, with the interpretation that there are at least $n$ distinct atoms $\alpha$ with $\alpha\leq x$, then in the enriched language the 
theory of all infinite atomic Boolean algebras is complete, has quantifier elimination and is decidable. This theory, that we denote by $T^{Bool}$, 
is axiomatized by sentences stating that the models are infinite Boolean algebras and every nonzero element has an atom below it.
(See \cite[Theorem 16,pp.70]{KK}, a new proof is given by Macintyre and myself \cite{DM-bool} that is uniform for the other enrichments stated below). 

The main examples are Boolean algebras of subsets of an infinite set $I$, namely $\P(I)$ (which denotes the powerset of $I$) with the usual set-theoretic Boolean operations. 
These are clearly not the only models, since no countable model can be a powerset algebra. 

Feferman-Vaught \cite{FV} and Mostowski added a unary predicate $Fin(x)$ to the 
language of Tarski 
with the intended interpretation in a Boolean algebra of subsets $\P(I)$ as " $x$ is finite". Feferman-Vaught \cite{FV} proved that this theory is complete, decidable, and has quantifier elimination. A new proof was given Macintyre and myself 
in \cite{DM-bool}. Note that there are other models where $Fin(x)$ holds and $x$ is not finite.

Let $\mathcal{L}_{Boolean}^{fin}$ denote the enrichment of $\cL_{Boolean}$ 
got by adding all the unary predicates $C_j(x)$ and $Fin(x)$, for all $j\geq 1$. 
Let $\sharp(x)$ denote the number of atoms $\alpha$ such that $\alpha \leq x$. Note that the $C_j$ are definable in 
$\cL_{Boolean}$ but $Fin(x)$ is not.                                                                                                                                                                                                                                                                                                                                                                                                                                                                                                                                                                                                                                                                                                                                                                                                                                                                                                                                                                                                                                                                                                                                                                                                                                                                                                                                                                                                                                                                                                                                                                                                                                                                                                                                                                                                                                                                                                                                                                                                                                                                                                                                                                                                                                                                                                                                                                                                                                                                                                                                                                                                                                                                                                                                                                                                                                                                                                                                                                                                                                                                                                                                                                                                                                                                                                                                                                                                                                                                                                                                                                                                                                                                                                                                                                                                                                                                                                                                                                                                                         

Let $T^{fin}$ of infinite atomic Boolean algebras in the language $\mathcal{L}_{Boolean}^{fin}$.

\begin{thm}[Feferman-Vaught {\cite{FV}, new proof by Derakhshan-Macintyre \cite{DM-bool}}] 
\label{bool1}The theory $T^{fin}$ of infinite atomic Boolean algebras with the 
set of finite sets distinguished is complete, decidable and 
has quantifier elimination with respect to all the $C_n$, $(n\geq 1)$, and $Fin$ (i.e. in the language 
$\mathcal{L}_{Boolean}^{fin}$). The axioms required for completeness and quantifier elimination are the 
axioms of $T^{Bool}$ together with 
sentences expressing that $Fin$ is a proper (Boolean) ideal, the sentence
$$
\forall x (\neg Fin(x) \Rightarrow (\exists y)(y<x \wedge \neg Fin(y) \wedge \neg Fin(x\wedge \neg y))).
$$
and, for each $n<\omega$, the sentence $\forall x (\sharp(x)\leq n \Rightarrow Fin(x))$.
\end{thm}

This theorem plays a fundamental role in our works on the model theory of adeles and restricted products. 
It also applies to a commutative ring via the Boolean algebra of its idempotents (see Section \ref{sec-axioms}). In this way it has been used in the work in \cite{DM-axioms} of Macintyre and myself 
on axioms for rings elementarily equivalent to restricted products (see Section \ref{sec-axioms}).

Motivated by questions in number theory, around reciprocity laws, in \cite{DM-bool} we expanded the language 
$\mathcal{L}_{Boolean}^{fin}$ by unary predicates $Res(n,r)(x)$ for all $n,r\in \Z$, $n>0$, with the intended interpretation, in $\P(I)$, 
that $Fin(x)$ holds and the cardinal of $x$ is congruent to $r$ modulo $n$. 

Let $\mathcal{L}_{Boolean}^{fin,res}$ denote the enrichment of $\mathcal{L}_{Boolean}^{fin}$ by these predicates and $T^{fin,res}$ the theory of all infinite atomic Boolean algebras in the language $\mathcal{L}_{Boolean}^{fin,res}$. 

In \cite{DM-bool}, Macintyre and myself gave axioms for $T^{fin,res}$ and prove the following.

\begin{thm}[Derakhshan-Macintyre {\cite{DM-bool}\label{bool2}}] The theory $T^{fin,res}$ of infinite atomic 
Boolean algebras in the enriched language with all the $C_n$, $(n\geq 1)$, $Fin$, and all $Res(r,n)$, $(n,r\in \Z, n>0)$, 
is complete, decidable, and has 
quantifier elimination. The axioms needed to get the completeness and 
quantifier elimination are the axioms of $T^{fin}$ together with the Boolean-Presburger axioms 
as follows: 
$$\forall x (Res(n,r)(x)\Rightarrow Fin(x)),$$
$$\forall x (Fin(x) \wedge \sharp(x)=m \wedge m\equiv r (\mathrm{mod}~n) \Rightarrow Res(n,r)(x)),$$
$$\forall x(Res(n,r)(x) \wedge r\equiv s (\mathrm{mod}~n) \Rightarrow Res(n,s)(x)),$$
$$\forall x(Res(n,r)(x)\wedge r \not\equiv s (\mathrm{mod}~n) \Rightarrow \neg Res(n,s)(x),$$ 
$$\forall x(Res(m,r)(x)\wedge n|m \Rightarrow Res(n,r)(x)),$$
$$\forall x(Fin(x)\Rightarrow \bigvee_{0\leq r< n} Res(n,r)(x)),$$
for all $n,r,s,m$,
$$\forall x \forall y(x\wedge y=0 \wedge Res(n,r)(x) \wedge Res(n,s)(y) \Rightarrow Res(n,r+s)(x\cup y)),$$
for all $n,r,s$; and
$$\forall x \forall y(x\cap y=0 \wedge Res(n,r)(x\vee y) \Rightarrow \bigvee_{\substack{0\leq s<n\\
0\leq t<n\\s+t\equiv r (\mathrm{mod}~n)}} Res(n,s)(x) \wedge Res(n,t)(y)),$$
for all $n,r$.
\end{thm}

We can apply this theorem to the expansion of the Boolean algebra of idempotents of $\A_K$ or of more general commutative rings 
with adding the $Fin$ and $Res$-predicates and have quantifier elimination and decidability. 
In Section \ref{sec-recip} we give an application of this on Hilbert reciprocity in number theory from \cite{DM-ad}.

\medskip

\subsection{\bf Enrichments of valued fields}\label{ssec-venrich} 

\

\medskip

As before, for a valued field $K$, we denote the valuation by $v: K\rightarrow \Gamma \cup \{\infty\}$,  the value group by $\Gamma$ (with top element $\infty$), the valuation ring by $\cO_K$, 
 the maximal ideal by $\cM_K$, and the residue field by $k$ . 

The quantifier elimination for adeles or generally restricted products takes place in two steps. The first step is a general quantifier elimination that works for all restricted products, and would depend on a chosen enrichment of infinite atomic Boolean algebras. This is stated in Subsection 4.2. The second step is a finer result that requires a quantifier elimination in the factors or "stalks" of the restricted product. This depends on a chosen enrichment of Henselian valued fields. We shall state some convenient languages in this Subsection for this purpose. The resulting results are presented in Subsections 4.2 and 5.1.

Let $\cL$ be a language for the factors $K_v$. This means that $K_v$ are $\cL$-structures. If $v$ is non-Archimedean, then 
$K_v$ is non-trivially valued Henselian whose value group is $\Z$. One can have a symbol for the valuation in both $1$-sorted and many-sorted  situations. Even though the valuation of the non-Archimedean $K_v$ is uniformly definable by an $\cL_{rings}$-formula that is $\exists \forall$ (see Theorem 2.2), the sorting helps with some of the finer results and enables us to keep track of residue fields or residue rings where uniformities depend on, as $v$ or $K$ vary.

On the other hand for us it is important to keep the basic analysis within the language of rings and regard $\A_K$ as a ring, which we do in all cases.

In the interpretation of $K_v$ as an $\cL$-structure, we have to take into account the Archimedean $K_v$ as well. In this case, in the situations that we add a predicate for the valuation $v$ of the non-Archimedean $K_v$, we adopt the convention that $v$ is interpreted in the Archimedean $K_v$ as the trivial valuation. This means that in the Archimedean $K_v$, 
$v(a)=0$ if and only if $a$ is not zero. This means that $K_v$ is the valuation ring and $\{0\}$ the maximal ideal. For the real $K_v$ we take the language of ordered rings $\cL_{rings} \cup \{<\}$ and for the $K_v$ that are complex, we take the language of rings. Then the  Archimedean $K_v$ admit quantifier elimination by results of Tarski for real closed and algebraically closed fields (see \cite{KK}).

Note that while the valuation rings of all the non-Archimedean $K_v$ and their unit balls are $\cL_{rings}$-definable, the unit ball in $\R$ is also $\cL_{rings}$-definable, but the unit ball in $\C$ is not $\cL_{rings}$-definable.

The required quantifier elimination for the $K_v$ is the following.

\medskip

{\bf Q.E. for stalks:} {\it For an $\cL$-formula $\varphi(x)$, where $x$ is a tuple of variables, 
there exits an $\cL$-formula $\psi(x)$ which is quantifier-free in a distinguished sort of $\cL$, such that for all but finitely many $K_v$ we have
$$K_v\models \forall x (\varphi(x) \Leftrightarrow \psi(x)),$$
and in each of the exceptional $K_v$, $\phi$(x) is equivalent to a quantifier-free formula relative to a finite set of sorts.}

\medskip

Recall the following basic result.

\begin{thm}[Ax-Kochen-Ershov {\cite{cherlin}, see also \cite{DM-ad}}]\label{ake} For each sentence $\phi$ of the language of valued fields effectively there exist a positive integer $n$ and a sentence $\psi$ of the language of rings so that for any Henselian valued field $K$ of characteristic $0$ with residue field $k$ of characteristic not dividing $n$ and value group a $\Z$-group, i.e. elementarily equivalent to $(\Z,+,0,<)$ in the language of ordered abelian groups, we have
$$K\models \phi \Leftrightarrow k \models\psi $$
\end{thm}

It is convenient to use the following languages for the non-Archimedean $K_v$.

\

{\it 3.2.1. The Macintyre language \cite{Macintyre1}.} $\mathcal{L}_{Mac}=\{+,.,0,1,P_n(x)\}$ is the expansion of $\cL_{rings}$ by predicates $P_n(x)$ interpreted in a field as the set of non-zero $n$th powers, for all $n\geq 2$. Macintyre \cite{Macintyre1} proved that 
the $\cL_{Mac}$-theory of $\Q_p$ admits elimination of quantifiers. It follows that a definable subset of                                                                                                                                                                                                                                                                                                                                                                                                                                                                                                                                                                                                                                                                                                                                                                                                                                                                                                                                                                                                                                                                                                                                                      $\Q_p^m$, for any $m$, (in $\cL_{rings}$ or $\cL_Mac$) is a finite union of locally closed sets (i.e. an intersection of an open and a closed set) in $p$-adic topology, and is thus measurable. 

\medskip

{\it 3.2.2.} In \cite{PR-book}, Prestel and Roquette defined an extension $\cL_{PR}$ 
of $\cL_{Mac}$ for the theory of $p$-adically closed fields of 
rank $d$, which are defined by the condition that $\mathcal{O}_{K}/(p)$ has dimension $d$ over $\F_p$, by 
adding of constant symbols for an $\F_p$-basis of this quotient. They proved that the theory of $p$-adically closed fields of $p$-rank $d$ admits elimination of quantifiers in this language. 

Remark that if $K$ is a finite 
extension of $\Q_p$ of degree $d$, then $K$ is $p$-adically closed of $p$-rank $d$.  

It follows that an $\cL_{PR}$-definable subset of each non-Archimedean ${K_v}^m$, for any $m$, is a finite union of locally closed sets, and thus measurable, and an infinite definable subset of $K_v$ 
has non-empty interior. 

\begin{note} For the Archimedean $K_v$, it follows from Tarski's quantifier elimination theorems for $\R$ and $\C$ that a definable subset of $K_v^m$ for any $m$, is a finite union of locally closed sets and is measurable.\end{note} 

\medskip

{\it 3.3.3.} In \cite{Belair}, Belair defined an extension of the Macintyre language in which the theory of $\Q_p$, for all $p$, admits uniform elimination of quantifiers. He added constants for an element of least positive value and for coset representative for the group of non-zero $n$th powers and solvability predicate of Ax for the residue fields
$Sol_m(x_1,\dots,x_m)$, $m\geq 2$, interpreted in a valued field $K$ by
$$(\bigwedge_{1\leq i\leq m} v(x_i)\geq 0) \wedge \exists y (v(y)\geq 0 \wedge v(y^m+x_1y^{m-1}+\dots+x_m)>0).$$
By a result of Kiefe \cite{kiefe} the theory of pseudofinite fields admits quantifier elimination in the language of rings augmented by $Sol_k$, for all $k$, defined by
$$Sol_k(y_1,\dots,y_n) \leftrightarrow \exists z (z^m+y_1z+\dots+y_m=0).$$ 
This elimination of quantifiers holds uniformly for all $\F_q$, where $q$ is large enough 
(either fixed or unbounded characteristic).

\

{\it 3.2.4. The languages of Denef-Pas and Pas} (\cite{pas},\cite{pas2}). The Denef-Pas language 
\[
\cL_{Denef-Pas}=(\cL_{field},\cL_{residue},\cL_{group},v,\bar{ac})
\]
is a 3-sorted language with the language of rings for the field sort $\cL_{field}$ and for the residue field sort $\cL_{residue}$, and the language of ordered abelian groups $\{+,0,\leq,\infty\}$ with a top element $\infty$ for the value group sort $\cL_{group}$. 

There is a function symbol $v$ from the field sort to the value group sort interpreted in a valued field $K$ as the valuation, and a function symbol from the field sort to the residue field sort interpreted in $K$ as the angular component map modulo $\cM_K$. This map is defined by the following conditions 

1) $\bar{ac}(0)=0$, 

2) The restriction of $\bar{ac}$ to $K^*$ is a multiplicative map into $k^*$, 

3) The restriction of $\bar{ac}$ to the group of units of $\mathcal{O}_K$ coincides with the restriction of the residue map to the group of units. See \cite{pas}.

After choosing a uniformizing element $\pi_v$ in a valued field $K$, we can define $\bar{ac}(x)=Res(x\pi_v^{-v(x)})$, where $Res$ denotes the residue map $\cO\rightarrow k$ modulo $\cM_K$ (extended to $K$ by zero).

Each non-Archimedean $K_v$ has an $\bar{ac}$-map defined as above. We can also get an $ac$-map from a cross section which exists in an $\aleph_1$-saturated valued field (see \cite{cherlin-book}). 

By Pas' Theorem \cite[Theorem 4.1, pp.155]{pas}, the theory of Henselian valued fields of equicharacteristic zero admits quantifier elimination in the language $\cL_{Denef-Pas}$ for the field sort relative to the other sorts.  Since this quantifier elimination holds for all ultraproducts of $K_v$ of unbounded residue characteristic with respect to a non-principal ultrafilter, 
it follows that the quantifier elimination holds uniformly for all $K_v$ of residue characteristic $p>N$ for some $N$. This elimination is effective, as in Theorem \ref{ake}, i.e. $N$ and the quantifier free formulas can be effectively given.

In Pas \cite{pas2}, Pas defined an extension $\cL_{Pas}$ of the Denef-Pas language got by adding the higher ac-maps $\bar{ac}_n$, for all $n\geq 1$, and infinitely many sorts
$\cL_{Res(n)}$ equipped with the language of rings, and interpreted in a valued field $K$ as $\cO_K/\cM^n$, with maps $Res_n$ interpreted as the residue maps $\cO_K\rightarrow \cO_K/\cM_K^n$ extended to $K$ be zero 
(together with connecting maps between the residue sorts). An interpretation of $\bar{ac}_n$ is $Res_n(x \pi^{-v(x)})$. 

By \cite{pas2}, each non-Archimedean $K_v$ has quantifier elimination for the field sort relative to the other sorts in $\cL_{Pas}$.

Combining these results we can deduce.

\begin{cor}[Follows from Pas \cite{pas} and \cite{pas2}]\label{pas} Let $K$ be a number field with non-Archimedean  completions $K_v$, $v\in V_K^{fin}$. 
Given a formula $\phi(x_1,\dots,x_n)$ from the ring language or the language of valued fields, there is by an effective procedure
\begin{itemize}
\item an integer $N\geq 1$,
\item an $\cL_{Denef-Pas}$-formula $\psi(x_1,\dots,x_n)$ that has no quantifiers over the field sort and has its quantifiers over the value group $\Gamma$ and residue field $k$,
\item for each $K_v$ of residue characteristic $p<N$, an $\cL_{Pas}$-formula $\psi_p(x_1,\dots,x_n)$ and finitely many positive integers $p_1,\dots,p_l$ 
such that $\psi_p(x_1,\dots,x_n)$ is quantifier-free in the field sort and has its quantifiers ranging from the sorts $\cL_{Res(p_1)},\dots,\cL_{Res(p_l)}$ involving 
the maps $\bar{ac}_{p_1},\dots,\bar{ac}_{p_l}$,
\end{itemize}
such that
\begin{itemize}
\item if $K_v$ has residue characteristic greater than $N$, then $\phi$ and $\psi$ are equivalent in $K_v$,
\item if $K_v$ has residue characteristic $p<N$, then $\psi$ and $\psi_p$ are equivalent in $K_v$.
\end{itemize}
\end{cor}
 
\

{\it 3.2.5. The language of Basarab \cite{basarab}} $\cL_{Basarab}=(L_{field},L_r: r\geq 1)$ has infinitely many sorts 
where
$$L_r=(L_{rings,r},L_{group,r},L_{group},v,\theta_r,v_r)$$
with the language of valued fields for the field sort $L_{field}$, the language of rings for the sorts $L_{rings,r}$ for all $r$, the language of groups for the sorts $L_{group,r}$, and the language of ordered abelian groups with top element $\infty$ for the sort $L_{group}$. 

The sort $L_{field}$ is for the field, the sorts $L_{rings,r}$ are for the residue rings $\mathcal{O}_{K,r}:=\mathcal{O}_K/\mathcal{M}_{K,r}$, where 
$$\mathcal{M}_{K,r}=\{a\in \mathcal{O}_K: v(a)>rv(p)\},$$ the sorts $L_{group,r}$ are for the quotients $K^{\times}/1+\mathcal{M}_{K,r}$, and the sort $L_{group}$ is for the value group $\Gamma$. The symbol $v$ is interpreted as the valuation and $\theta_r$ is interpreted as the map 
$$\theta_r(a+\mathcal{M}_{K,2r})=a(1+\mathcal{M}_{K,r})$$
defined on the subset 
$\mathcal{O}_{K,2r}\setminus (\mathcal{M}_{K,r}/\mathcal{M}_{K,2r})$ of $\mathcal{O}_{K,2r}$ with values in $K^{\times}/1+\mathcal{M}_{K,r}$. $v_r$ is interpreted as the map induced from the valuation on the disjoint union 
$$\mathcal{O}_{K}/\mathcal{M}_{K,2r} \cup K^{\times}/1+\mathcal{M}_{k}$$ into $\Gamma\cup \{\infty\}$. The structure 
$$\mathcal{K}_r=(\mathcal{O}_{K,2r},K^{\times}/1+\mathcal{M}_{K,r},\Gamma,\theta_r,v_r)$$
is called the mixed $r$-structure assigned to $K$. Note that $\mathcal{M}_{K,0}=\mathcal{M}_K$ is the maximal ideal of 
$\mathcal{O}_{K}$. 

If $K$ has residue characteristic zero, then $\mathcal{O}_{K}/\mathcal{M}_{K,r}$ is the residue field $k$ of $K$ and $$K^{\times}/1+\mathcal{M}_{K,r}=K^{\times}/1+\mathcal{M}_{K}$$ for all $r$. So all the mixed $r$-structures assigned to $K$ become the triple 
$$(k,K^{\times}/1+\mathcal{M}_K,v)$$ with the exact sequence 
$$1\rightarrow {k_K}^{\times}\rightarrow K^{\times}/1+\mathcal{M}_{K} \Gamma \rightarrow 0.$$

By \cite{basarab} (Theorem B, page 57), the theory of Henselian valued fields of characteristic zero with large residue field of fixed characteristic $p$ admits quantifier elimination in $\cL_{Basarab}$ for the field sort relative to the sorts $L_r$, $r\geq 1$.

Basarab's Theorem B in \cite{basarab} also applies to residue characteristic zero Henselian fields, hence to ultraproducts of Henselian valued fields of unbounded residue characteristic. It follows that given a formula $\phi(x_1,\dots,x_n)$ there is a formula $\psi(x_1,\dots,x_n)$ which is quantifier-free in the field sort and has its quantifiers from the sorted language $(k,K^{\times}/1+\mathcal{M}_K,v)$ such that $\phi$ and $\psi$ are equivalent in any Henselian valued field $K$ of residue characteristic greater than some $N$ depending on $\phi$ only, and $N$ can be found effectively.

Combining the residue characteristic zero and and fixed residue characteristic $p>0$ results of Basarab, we deduce the following.

\begin{cor}[Follows from Basarab {\cite[Theorem B]{basarab}}]\label{basarab} Let $K$ be a number field with non-Archimedean completions $K_v$, $v\in V_K^{fin}$. Given an $\cL_{Basarab}$-formula (or in particular a formula from the ring language or the language of valued fields) $\phi(x_1,\dots,x_n)$ there is (by an effective procedure)
\begin{itemize}
\item an integer $N\geq 1$,
\item an $\cL_{Basarab}$-formula $\psi(x_1,\dots,x_n)$ that has no quantifiers over the field sort and has its quantifiers from the sorts $(k,K^{\times}/1+\mathcal{M}_K,v)$,
\item for each $K_v$ of residue characteristic $p<N$, an $\cL_{Basarab}$-formula $\psi_p(x_1,\dots,x_n)$ and an integer $r_p\geq 1$ such that $\psi_p(x_1,\dots,x_n)$ is quantifier free in the field sort and has its quantifiers from the sorted language 
$L_{r_p}$,\end{itemize}
such that 
\begin{itemize}
\item if $K_v$ has residue characteristic greater than $N$, then $\phi$ and $\psi$ are equivalent in $K_v$,
\item if $K_v$ has residue characteristic $p<N$, then $\psi$ and $\psi_p$ are equivalent in $K_v$.
\end{itemize}
\end{cor}

{\bf Note:}  Other many-sorted languages for quantifier elimination in Henselian valued fields have been introduced by Weispfenning  \cite{Weisp2} and Kuhlmann \cite{kuhlmann}. These are closely related to $\cL_{Basarab}$. 

\

{\it 3.2.6. Adeles with product valuation \cite{DM-supp}.} This language has three sorts $$(\cL_{adeles,ring},\cL_{adeles,value},\cL_{adeles,residue},v^*,\bar{ac}^*),$$
with the language of rings for the sort $\cL_{adeles,ring}$, the language of ordered abelian groups together with a top element $\infty$ for the sort $\cL_{adeles,value}$, and the language of rings for the sort $\cL_{adeles,residue}$. 

$\cL_{adeles,ring}$ is for the ring of adeles, 
$\cL_{adeles,value}$ for the restricted direct product of the value groups $\prod'_{v\in V_{K}^f} (\Gamma_v \cup \infty)$, where $\Gamma_v$ is the value group of $K_v$ and $\infty$ a top element (a restricted product is with respect to the formula $x\geq 0$, cf. Subsection \ref{ssec-lang}); and $\cL_{adeles,residue}$ is for the direct product of the residue fields of $K_v$ over all $v\in V_{K}^f$. 

The function symbol 
$v^*$ is interpreted as the product valuation
$$\A_K^{fin} \rightarrow \prod'_{v\in V_{K}^f} (\Gamma_v\cup \infty)$$
onto the lattice-ordered group $\prod'_{v\in V_{K}^f} (\Gamma_v\cup \infty)$ defined by $v^*(a)=(v(a(v))_v$, for $a\in \A_K$. 

The function symbol $\bar{ac}^*$ is interpreted as the map from $\A_K\rightarrow \prod_{v\in V_K} k_v$ onto the product defined by $ac^*(a)=(ac(a(v)))_v$. 

For more details and model-theoretic results on the product valuation, 
see \cite{DM-supp}.

\section{\bf Generalized products and restricted products}\label{sec-rest}

The model theoretic notions of generalized product of $\cL$-structures, for a language $\cL$, were introduced and studied in the works of Feferman-Vaught and Mostowski (see \cite{FV}). 
What we call restricted product of $\cL$-structures appears in \cite{FV} under the name of weak product. It is a substructure of the generalized product. Most of the analysis of Feferman-Vaught is for the generalized product, however the results can also be proved for restricted products as well. 

In \cite{DM-ad} and \cite{DM-supp}, Macintyre and myself do this in a more general case of having a many-sorted language with function symbols and relation symbols. 
We give an outline in this Section, especially aimed at results on quantifier elimination for restricted products.

We remark that in \cite{DM-axioms} Macintyre and myself proved an analogue of the main theorem of Feferman-Vaught for rings. Interestingly,  
this gives both a converse to Feferman-Vaught and at the same time axioms for rings elementarily equivalent to restricted products and adeles. The proof is a modification and a ring-theoretic analogue of \cite{FV}. This shall be discussed in Section \ref{sec-axioms} of this paper.

\medskip

\subsection{\bf Language for restricted products}\label{ssec-lang} 

\

\medskip

Let $L$ denote a many-sorted first-order language with a set of sorts $Sort$ and signature $\Sigma$ with relation and function symbols and equality in each sort. See \cite[Section 4.3]{enderton} for the basic definitions and results on many-sorted languages on well-formed formulas, substructures. We give a few definitions.

An $L$-embedding $F:N\rightarrow M$ is a collection of maps
$$F_{\sigma}: N_{\sigma} \rightarrow M_{\sigma}$$
indexed by the sorts $\sigma$, such that for any relation symbol $R$ of sort $(\sigma_1,\dots,\sigma_k)$, 
$$N \models R(f_1,\dots,f_k)\Leftrightarrow 
M \models R(F_{\sigma_1}(f_1),\dots,F_{\sigma_k}(f_k)),$$
and for any function symbol $G$ of sort $(\sigma_1,\dots,\sigma_k,\sigma_{k+1})$, 
$$G(F_{\sigma_1}(f_1),\dots,F_{\sigma_k}(f_k))=F_{\sigma_{k+1}}(G(f_1,\dots,f_k)),$$
where $f_1,\dots,f_k,f_{k+1}$ range over elements of sorts $\sigma_1,\dots,\sigma_k,\sigma_{k+1}$ respectively. 

Note that each $F_{\sigma}$ is injective since we have equality as a binary relation on each sort. 

$N$ is said to be an $L$-substructure of $M$ if $N_{\sigma} \subseteq M_{\sigma}$ for all $\sigma$ and the identity maps are $L$-embeddings. 

Now suppose that $(M_i)_{i\in I}$ is a family of $L$-structures. Let $\Pi:=\prod_{i\in I} M_i$. We give $\Pi$ an $L$-structure and make it sorted by the set 
$Sort$, and give an interpretation of the signature $\Sigma$ as follows.

If $\sigma \in Sort$, then the $\sigma$-sort of 
$\prod_{i\in I} M_i$ is the product $\prod_{i\in I} (M_i)_{\sigma}$, where $(M_i)_{\sigma}$ is the $\sigma$-sort of $M_i$. 

The interpretation in $\Pi$ of a relation symbol $R$ of sort $(\sigma_1,\dots,\sigma_r)$ is $\prod_{i\in I} R^{M_i}$, where 
$R^{M_i}$ is the interpretation of $R$ in $M_i$ (a subset of $(M_i)_{\sigma_1}\times \dots \times (M_i)_{\sigma_r}$). 

The interpretation in $\Pi$ of a function symbol of sort $(\sigma_1,\dots,\sigma_r,\sigma_{r+1})$ is given by 
$$\tau^{(\Pi)}(f_{1},\dots,f_{r})(i)=
\tau^{(M_i)}(f_{1}(i),\dots,f_{r}(i))$$
for all $i\in I$, where $f_1,\dots,f_r$ range over elements of sorts $\sigma_1,\dots,\sigma_r$ respectively.

Let $\Phi(x_{1},\dots,x_{r})$ be an $L$-formula. Define 
$$[[\Phi(f_{1},\dots,f_{r})]]:=
\{i: M_i \models \Phi(f_{1}(i),\dots,f_{r}(i))\},$$
where $f_1,\dots,f_r$ range over elements of sorts $\sigma_1,\dots,\sigma_r$ respectively. This is 
a many-sorted generalization of Feferman-Vaught's Boolean values. 


Let $\cL_{Boolean}^{+}$ denote a given enrichment of the language of Boolean algebras $\cL_{Boolean}$ that contains the predicate $Fin(x)$ (e.g. $\cL_{Boolean}^{fin}$ and $\cL_{Boolean}^{fin,res}$ from Subsection \ref{ssec-benrich}). 

Let $\P(I)$ denote the Boolean algebra of subsets of $I$ and $\P(I)^+$ its expansion to an $\cL_{Boolean}^+$-structure. 
\begin{Def} For any $\mathcal{L}_{Boolean}^{+}$-formula $\Psi(z_1,\dots,z_m)$ and $L$-formulas $\Phi_1,\dots,\Phi_m$ in the free 
variables $x_{1},\dots,x_{r}$ of sorts $\sigma_1,\dots,\sigma_r$ respectively, 
let 
$$\Psi \circ < \Phi_1,\dots,\Phi_m>$$ denote the relation defined by 
$$\Pi \models \Psi \circ<\Phi_1,\dots, \Phi_m>(f_{1},\dots,f_{r}) \Leftrightarrow$$ 
$$\P(I)^+\models \Psi([[\Phi_1(f_{1},\dots,f_{r})]],\dots,
[[\Phi_m(a_{1},\dots,a_{r})]]),$$
where $a_1,\dots,a_r$ range over elements of sorts $\sigma_1,\dots,\sigma_r$ respectively.\end{Def}

Expand $L$ by adding a new relation symbol for each of these relations. Let $\cL_{Boolean}^{+}(L)$ denote the resulting language. This gives 
$\Pi$ an $\cL_{Boolean}^{+}(L)$-structure, generalizing the 1-sorted case in Feferman-Vaught \cite{FV}. See also \cite{DM-supp}.

We now define a many-sorted generalization of the Feferman-Vaught notion of a generalized product. 

Suppose for each sort $\sigma$ we have a formula $\Phi_{\sigma}(x)$ in a 
single free variable $x$ of sort $\sigma$. Suppose that for all $\sigma$ and $i$ the set 
$$S_{\sigma,i}=\{a\in Sort_{\sigma}(M_i): M_i\models \Phi_{\sigma}(a)\}$$
is an $L$-substructure of $M_i$.  In particular, for any function symbol $F$ of sort $(\sigma,\tau)$, if 
$a\in S_{\sigma,i}$, then $F(a)\in S_{\tau,i}$
for all $i$.

\begin{Def} With the above assumptions and notation, define the restricted product of $M_i$ with respect to the formulas $\Phi_{\sigma}(x)$, denoted by $\prod_{i\in I}^{(\Phi_{\sigma})} M_i$, 
to be the structure sorted by $Sort$, such that for $\sigma\in Sort$, its $\sigma$-sort 
is the set of all $a\in \prod_{i\in I} (M_i)_{\sigma}$ such that $[[\neg \Phi_{\sigma}(a)]]$ is finite.\end{Def}

$\prod_{i\in I}^{(\Phi_{\sigma})} M_i$ is an $\cL_{Boolean}^{+}(L)$-substructure of $\Pi$. Indeed, if $F$ is a function symbol of sort $(\sigma,\tau)$,
and $a$ is in the $\sigma$-sort of $\prod_{i\in I}^{(\Phi_{\sigma})} M_i$, then since the sets $S_{\sigma,i}$ are 
$L$-substructures of $M_i$ for all 
$i$, we have $Fin([[\neg \Phi_{\tau}(F(a)]])$, so $F(a)$ is in $\tau$-sort of $\prod_{i\in I}^{(\Phi_{\sigma})} M_i$. Clearly $\prod_{i\in I}^{(\Phi_{\sigma})} M_i$ is $\cL_{Boolean}^{+}(L)$-definable. 

\medskip

\subsection{\bf Quantifier elimination in restricted products}\label{ssec-qe} 

\

\medskip

The following theorem, originally proved by Feferman-Vaught and extended to many-sorted case by Macintyre and myself 
in \cite{DM-supp} gives quantifier elimination for $\prod^{(\Phi)}_{i\in I} M_i$ in the language $\cL_{Boolean}^{+}(L)$.
\begin{thm}[Feferman-Vaught \cite{FV}, Derakhshan-Macintyre\cite{DM-supp}]\label{restricted-qe} 
For any $\cL_{Boolean}^{+}(L)$-formula $\Psi(x_1,\dots,x_n)$, where $x_1,\dots,x_n$ are free 
variables of sorts $\sigma_1,\dots,\sigma_n$ respectively, one can effectively construct 
$L$-formulas ree 
$$\Psi_1(x_1,\dots,x_n),\dots,\Psi_m(x_1,\dots,x_n)$$
with the same free variables, 
and an $\cL_{Boolean}^+$-formula $\Theta(X_1,\dots,X_m)$ 
such that for any indexed family $(M_i: i\in I)$ of $L$-structures and any $a_1,\dots,a_n\in \prod^{(\Phi)}_{i\in I} M_i$, 
$$\prod^{(\Phi)}_{i\in I} M_i\models \Psi(a_1,\dots,a_n)$$ if and only if 
$$\P(I)^+\models \Theta([[\Psi_1(a_1,\dots,a_n)]],\dots,[[\Psi_m(a_1,\dots,a_n)]]).$$
\end{thm}

Theorem \ref{restricted-qe} applies to many restricted products. To apply it to 
the ring of adeles $\A_K$ and the ring of finite adeles $\A_K^{fin}$ represent 
$\A_K$ (resp. $\A_K^{fin}$) as the restricted product of the $K_v$, where $v\in V_K$ (resp. $v\in V_K^{fin}$) with respect to the formula $\Phi_{val}(x)$ from Theorem \ref{CDLM-th} that uniformly defines the valuation rings of $K_v$ for all $v$. This is uniform for all number fields $K$.

\begin{note} In Section \ref{ssec-adsp} we show that a variant of this theorem holds for the adele spaces of varieties 
$V(\A_K)=\prod^{'}_{v\in V_K} V(K_v)$, where $V$ is an algebraic variety 
and the restricted product is with respect to $V(\mathcal{O}_v)$. The space $V(\A_K)$ coincides with the set of solutions of the defining equations of $V$ in the adeles $\A_K$.\end{note}

In applying Theorem \ref{restricted-qe}, we choose a language $L$ extending $\cL_{rings}$ such that 
all the completions $K_v$ are $L$-structures. If $L$ is a definitional extension of $\cL_{rings}$, then 
$\A_K$ is an $\cL_{rings}$-structure. If $L$ is not a definitional extension of $\cL_{rings}$, then we  
consider $\A_K$ as an $\cL_{Boolean}^+(L)$-structure. Taking $L$ to be $\cL_{Mac}$, $\cL_{Belair}$, $\cL_{Denef-Pas}$, or $\cL_{Basarab}$ and applying Theorem \ref{restricted-qe} we get quantifier eliminations for $\A_K$ in $\cL_{Boolean}^+(L)$. 

Taking $L$ to be $\cL_{rings}$ and applying Theorems \ref{bool1} and \ref{restricted-qe} we get.
\begin{cor}[Derakhshan-Macintyre \cite{DM-ad}]\label{qe-fin}Let $\varphi(x_1,\dots,x_n)$ be an $\cL_{rings}$-formula. Then there are $\cL_{rings}$-formulas 
$$\psi_1(x_1,\dots,x_n),\dots,\psi_l(x_1,\dots,x_n),$$
where $l\geq 1$, and a Boolean combination $\Psi(x_1,\dots,x_n)$ 
of $Fin([[\psi_k(x_1,\dots,x_n)]])$ and 
$C_j(\psi_s([[x_1,\dots,x_n]]))$, where $k,s \in \{1,\dots,l\}$, such that 
$$\A_K \models \forall x_1 \dots \forall x_n (\varphi(x_1,\dots,x_n) \Leftrightarrow \Psi(x_1,\dots,x_n)).$$
\end{cor}

\begin{cor}[Derakhshan-Macintyre \cite{DM-ad}]\label{def-ring} Let $n\geq 1$. A definable subset of $\A_K^n$ in the language of rings is a Boolean combination of sets defined by the $\cL_{rings}$-formulas
\begin{enumerate}
\item $Fin([[\psi(x_1,\dots,x_n)]])$,
\item $C_j([[\phi(x_1,\dots,x_n)]])$, 
\end{enumerate}
where $j\geq 1$, and $\psi$ and $\phi$ are $\cL_{rings}$-formulas.
\end{cor}

To see that $Fin([[\psi(x_1,\dots,x_n)]])$ and $C_j(\phi([[x_1,\dots,x_n]]))$ are $\cL_{rings}$-formulas, we use Theorem \ref{defid}. For example, $Fin([[\psi(x_1,\dots,x_n)]])$ can be expressed as "there exists an idempotent $e$ such that $Fin(e)$ holds and $e$ is the supremum of all the minimal idempotents $e$ such that $e\A_K \models \psi(x_1(e),\dots,x_n(e))$. 

\begin{note} By Theorem \ref{restricted-qe}, the formulas $\Psi_j$ do not depend on the choice of the family of structures $M_i$, hence the quantifier eliminations in Theorem \ref{restricted-qe} and Corollaries 
\ref{qe-fin} and \ref{def-ring} are independent of the number field $K$.\end{note}

\medskip

\subsection{\bf The case of finite index set}\label{ssec-finindex}

\

\medskip

 If the index set $I$ is finite, then 
Theorem \ref{restricted-qe} becomes the following statement, which is of independent interest and extends results going back to Mostowski for the 1-sorted case.

\begin{thm}\cite{DM-ad}\label{fv-fin-prod} Consider a finite  index set $I=\{1,\dots,s\}$. 
Let $\psi(x_1,\dots,x_n)$ be an $\cL$-formula. Then there are finitely many 
$t$-tuples of formulas 
$$(\psi_1(x_1,\dots,x_n),\dots,\psi_t(x_1,\dots,x_n))$$
for some $t\in \N$, 
and elements $\S_1,\dots,\S_k$, for some $k\in \N$, where each $\S_j$
is in $\P(I)^{t}$ (where $\P(I)$ denotes the powerset of $I$)
such that for arbitrary $\cL$-structures $M_1,\dots,M_s$, and any 
$a_1,\dots,a_n$ in $M_1\times \dots \times M_s$
$$M_1\times \dots \times M_s\models \psi(a_1,\dots,a_n)$$
if and only if for some $j$ the sequence 
$$[[\psi_1(a_1,\dots,a_n)]],\dots,[[\psi_t(a_1,\dots,a_n)]]$$ is equal to $\S_j$.

\end{thm}
\begin{proof} Follows immediately by Theorem \ref{restricted-qe}.\end{proof}

\begin{cor}\cite{DM-ad}\label{mostowski} Let $A\subset M_1 \times \dots \times M_s$ be an $\cL$-definable set.
Then $A$ is a finite union of rectangles
 $ B_1 \times \dots \times B_s$ , where $B_i$ is a definable subset of $M_i$.
\end{cor}
\begin{proof} Follows immediately by Theorem \ref{fv-fin-prod}.
\end{proof}

\begin{remark} For any finite subset $S=\{v_1,\dots,v_l\}$ of $V_K$ containing the all Archimedean valuations, we can view $\A_K$ as the finite direct product 
$$K_{v_1} \times \dots \times K_{v_l} \times \A_K^S,$$ where 
$\A_K^S$ is the restricted direct product of $K_v$ over all $v\notin S$ with respect to the rings $\cO_v$. 
\end{remark}

\begin{remark} For any subset $T=\{i_1,\dots,i_l\}$ of the index set $I$, we can view $\prod_{i\in I}^{(\Phi)} M_i$ as the finite direct product 
$$M_{i_1} \times \dots \times M_{i_l} \times \prod_{i\notin T}^{(\Phi)} M_i,$$ 
where $\prod_{i\notin T}^{(\Phi)} M_i$ is the restricted direct product of the $M_i$ with respect to the formula $\Phi(x)$. In this way Theorem \ref{fv-fin-prod} and Corollary \ref{mostowski} can be applied to $\A_K$ and $\prod_{i\in I}^{(\Phi)} M_i$. 

By Corollary 
\ref{mostowski}, the definable subsets of $$\prod_{v\in S} K_v \times \A_K^S$$ 
(resp. $$\prod_{t\in T} M_t \times \prod_{t\notin T}^{(\Phi)} M_t),$$
are finite unions of sets of the form 
$$X_1 \times \dots \times X_l \times Y$$
where $X_j$ is a definable subset of $K_{v}$ (resp. $M_{i_j}$) for $v\in S$ (resp. for $j=1,\dots,l$), and $Y$ is a definable subset of 
$\A_K^S$ (resp. $\prod_{i\notin S}^{(\Phi)} M_i$). 

As $Y$ is a restricted product, Theorem \ref{restricted-qe} applies to it (and for example gives results on its definable subsets).
\end{remark}

This direct product decomposition can be specially useful in calculating measures of definable sets in $\A_K$ and 
$\prod_{i\in I}^{(\Phi)} M_i$. 

Taking $S$ to be the set of all
Archimedean valuations, this way one can compare the measures got from the Archimedean factors (a finite product) with those got from the non-Archimedean factors (an infinite restricted direct product). 

\medskip

\subsection{\bf An example from algebraic groups: Weil's conjecture on Tamagawa numbers}

\

\medskip

For the definition of the adele space of a variety $V(\A_K)$ and of Tamagawa number see Subsection \ref{ssec-adsp}. These are naturally definable subsets of $\A_K^m$ for some $m$. They also have the structure of a model-theoretic
restricted product (see  Subsection \ref{ssec-adsp}), and so the results of Subsection \ref{ssec-qe} are applicable to them.

Let $G$ be an algebraic group over a number field $K$. Tamagawa proved that the volume of 
$SO_n(f)(\A_K)/SO_n(f)(K)$ with respect to the Tamagawa measure is equal to $2$, where $SO_n(f)$ denotes the special orthogonal group of a non-degenerate 
quadratic form $f$ in $n$-variables with rational coefficients, and proved that this is equivalent to Siegel's famous formula for the Mass of a quadratic form (called Siegel's Mass formula), 
thereby giving also a new volume-theoretic proof of Siegel's formula. 

Weil gave a more general conjecture that the Tamagawa volume of $G(\A_K)/G(K)$ is equal to $1$ for all simply connected semi-simple groups and proved it for many classical groups. Langlands proved it for all Chevalley groups. Kottwitz proved the general case. The story of Weil's conjecture and related results has been quite interesting. Eskin-Rudnick-Sarnak gave a new proof of Siegel's Maas formula using ergodic theory. For details and references see \cite[Chapter 5]{Platonov-R-book} (and Kneser's article in \cite{CF} for the early results). See also Subsection \ref{ssec-tam}.

In each of these, the volume is calculated after first computing the volume of the points over the finite adeles $\A_K^{fin}$, then computing the product of the volumes of the set of points over the Archimedean factors, and then finally comparing the two quantities. Mysteriously in all these cases, the product is an integer. 

As stated above, this method of calculating volumes can be carried out for definable sets by 
Theorems \ref{restricted-qe} and \ref{fv-fin-prod} and Corollaries \ref{def-ring}, \ref{qe-fin}, and \ref{mostowski}. It is a powerful method for calculating adelic volumes.

\section{\bf Definability in adeles}\label{sec-def}

\medskip

\subsection{\bf Definable subsets of $\A_K^m$}\label{ssec-defset} 

\

\medskip

Let $L$ be a language for the $K_v$. 
Given a subset $I$ of $V_K^{fin}$, a formula $\phi(x_1,\dots,x_n)$ from $L$, and $a_1,\dots,a_n \in \A_K$,
we denote 
$$[[\psi(a_1,\dots,a_n)]]^I=\{v\in I: K_v \models \phi(a_1(v),\dots,a_n(v))\}.$$

\begin{thm}[Derakhshan-Macintyre {\cite{DM-ad}}]
\label{th-def-sets2} 
Let $K$ be a number field and $n\geq 1$. Let $X$ be a definable subset of 
$\A_K^n$ defined by a formula $\phi(x_1,\dots,x_n)$ that is any of the following
\begin{itemize} 
\item an $\cL_{rings}$-formula or a formula of the language of valued field,
\item an $\cL_{Boolean}^{fin}(\cL_{Basarab})$-formula (resp. an $\cL_{Boolean}^{fin}(\cL_{Denef-Pas})$-formula or an $\cL_{Boolean}^{fin}(\cL_{Pas})$-formula),
\end{itemize}
Then there is a finite set $S=\{v_1,\dots,v_t\}$ of non-Archimedean valuations effectively computable from $X$, an integer $N\geq 1$, and $\cL_{Basarab}$ (resp. $\cL_{Pas}$) formulas $\psi_{v_1},\dots,\psi_{v_t}$, such that 
$X$ is a Boolean combination of the following sets:
\begin{enumerate}

\item $\{(a_1,\dots,a_n)\in \A_K^n: \P(V_K)^+\models \Theta(a_1,\dots,a_n) \newline \wedge \bigwedge_{v_j\in T} K_{v_j}\models \psi_{v_j}(a_1(v_j),\dots,a_n(v_j))\}$, 
\item $\{(a_1,\dots,a_n)\in \Bbb A_K^n: \P(V_K)^+\models Fin([[\psi(a_1,\dots,a_n)]]\}$,

\end{enumerate}
where $T\subseteq S$, and $\Theta(a_1,\dots,a_n)$ is a conjunction from the following set of conditions
\begin{itemize}
\item $C_j([[\varphi_1(a_1,\dots,a_n)]]^{real})$
\item $C_k([[\varphi_2(a_1,\dots,a_n)]]^{complex})$
\item $C_s([[\varphi_3(a_1,\dots,a_n)]]^{V_K^{fin}\setminus S})$

\end{itemize}
such that  $j,k,s\geq 1$, and the following hold:
\begin{itemize}
\item $\varphi_1$ is quantifier-free in the language of ordered rings, 
\item $\varphi_2$ is quantifier-free in $\cL_{rings}$, 
\item $\varphi_3$ and $\psi$ are $\cL_{Basarab}$-formulas (resp. $\cL_{Denef-Pas}$-formulas) that are quantifier-free in the field sort and have their quantifiers from the residue field sort,   
\item for every $j$, $\psi_{v_j}$ is an $\cL_{Basarab}$-formula (resp. an $\cL_{Pas}$-formula) that is quantifier-free in the field sort and 
has its quantifiers from the sort $\cL_{r(v_j)}$ (resp. the sort $\cL_{Res_{r(v_j)}}$) for some $r(v_j)\geq 1$ depending on $v_j$.
\end{itemize}
A definable subset of the truncated restricted product $\prod_{v\in V_K^{fin}\setminus S} K_v$ defined by a formula $\phi$ as above 
is a Boolean combinations of sets of the types
\begin{enumerate}
\item $\{(a_1,\dots,a_n)\in \A_K^n: C_j([[\varphi(a_1,\dots,a_n)]])\}$,
\item $\{(a_1,\dots,a_n)\in \Bbb A_K^n: Fin([[\psi(a_1,\dots,a_n)]]\}$,
\end{enumerate}
where $\psi$ and $\varphi$ are $\cL_{Basarab}$ or $\cL_{Denef-Pas}$ formulas that are quantifier-free in the field sort and 
have all their quantifiers from the residue field sort.
\end{thm}

If $K=\Q$ since we have uniform quantifier elimination for $\Q_p$ for all
$p$ in $\cL_{Belair}$, we can get a simpler description.

\begin{thm}[Derakhshan-Macintyre {\cite{DM-ad}}]\label{th-def-sets} Let 
$X$ be a subset of 
$\Bbb A_{\Q}^n$ in the language of rings or the language of Belair, where $n\geq 1$. Then $X$ is a Boolean 
combination of sets of the following types:
\begin{enumerate}
\item $\{(a_1,\dots,a_n)\in \A_K^n: \B_K^+\models \Theta(a_1,\dots,a_n)\}$
\item $\{(a_1,\dots,a_n)\in \Bbb A_K^n: \B_K^+\models Fin([[\psi(a_1,\dots,a_n)]])\}$,
\end{enumerate}
where $\Theta(a_1,\dots,a_n)$ is a conjunction from the following statements
\begin{enumerate}
\item $C_j([[\varphi_1(a_1,\dots,a_n)]]^{real})$
\item $C_k([[\varphi_2(a_1,\dots,a_n)]]^{complex})$
\item $C_s([[\varphi_3(a_1,\dots,a_n)]]^{na})$,
\end{enumerate}
and $j,k,s\geq 1$, $\varphi_1$ is quantifier-free in the language or ordered rings, 
$\varphi_2$ is quantifier-free in $\cL_{rings}$, and
$\varphi_3$ and $\psi$ are quantifier-free in $\cL_{Belair}$.
\end{thm}

\begin{remarks}\label{typeI}\noindent\begin{enumerate}
\item In Theorems \ref{th-def-sets2} and \ref{th-def-sets} a special case of the sets in the clause (1) are sets of the form
$$\{(a_1,\dots,a_n)\in \A_K^n: \P(V_K)^+\models [[\Theta(a_1,\dots,a_n)]]=1\}.$$
In this case we call $X$ a definable set of Type I.
\item If $\phi$ is from $\cL_{rings}$ or the language of valued fields, then $\P(V_K)^+$ can be replaced by the Boolean algebra of idempotents $\B_K^+$, thus obtaining a quantifier-elimination 
that takes place within the ring of adeles.
\item The plus $+$ in $\P(V_K)^+$ and $\B_K^+$ indicate that the Boolean algebras $\P(K)$ and $\B_K$ are enriched with the predicates of the expanded language.
\end{enumerate}
\end{remarks}

\medskip

\subsection{\bf Measurability of definable subsets}\label{ssec-defmeas} 

\

\medskip

\begin{thm}[Derakhshan-Macintyre {\cite{DM-ad}}]\label{th-def-meas} A definable subset of $\Bbb A_K^n$, $n\geq 1$, in the language of rings is measurable.\end{thm}
Remark that measurability for a subset of $\A_K^n$, $n\geq 1$, is with respect to the product measure induced from a measure 
on $\A_K$ (cf. \ref{ssec-rest-prod}).

To give an idea of the proof suppose $X$ is defined by $Fin([[\Psi(x_1,\dots,x_n)]])$. 
Then
$$X=\{(a_1,\dots,a_n)\in \A_K^n: \B_K^+\models Fin([[\Psi(a_1,\dots,a_n)]])$$
$$=\bigcup_{\mathcal{F}} (\bigcap_{v\in \mathcal{F}} \{(a_1,\dots,a_n): K_v\models \Psi(a_1(v),\dots,a_n(v))\}$$
$$\cap \bigcap_{w\notin \mathcal{F}} \{(a_1,\dots,a_n)\in \A_K^n: K_w\models \neg \Psi(a_1(w),\dots,a_n(w))\}),$$
where $\mathcal{F}$ ranges over all the finite subsets of $V_K$.
Then one uses the fact that definable subsets in $K_v$ are finite unions of locally closed sets (cf. \ref{ssec-venrich}) hence measurable. 

\begin{note} We also have the following strengthening in \cite{DM-ad}. 
Let $L$ be any expansion of the language of rings with the property that the $L$-definable subsets of $K_v^n$, for any $n\geq 1$ and $v\in V_K$, are measurable. 
Then any $\cL^{fin,res}(L)$-definable subset of $\A_K^n$, where $n\geq 1$, is measurable.
\end{note}

\begin{note}The language $\cL^{fin,res}(L)$ has more expressive power than $\cL_{rings}$ for the adeles.\end{note}

\medskip

\subsection{\bf Countable unions and intersections of locally closed sets} \label{ssec-unions}

\

\medskip

 The proof of Theorem \ref{th-def-meas} shows the following.

\begin{cor}[Derakhshan-Macintyre {\cite{DM-ad}}]\label{cor-def-loc-cl} A definable subset of $\A_K^n$ in the language of rings or the language $\cL_{Boolean}^{fin}(\cL)$, where $\cL$ is $\cL_{Denef-Pas}$ or $\cL_{Belair}$, is a countable union or countable intersection of locally closed sets (in adelic topology).
\end{cor}
Indeed, let $\phi$ be formula from $\cL_{rings}$ or $\cL_{Boolean}^{fin}(\cL)$ as in Corollary\ref{cor-def-loc-cl} . Then it is easily seen that,
\begin{itemize}
\item sets of the form $\{\bar a: [[\phi(\bar a)]]=0\}$ and $\{\bar a: [[\phi(\bar a)]]=1\}$ are finite unions of locally closed sets,

\item sets of the form $\{\bar a: Fin([[\phi(\bar a)]])\}$ are countable unions of locally closed sets,

\item  set of the form $\{\bar a: \neg Fin([[\phi(\bar a)]])\}$ are countable intersections of locally closed sets,
\end{itemize}
Similarly for the $C_j(x)$.

This description of definable sets is optimal and can not be improved. Even though the definable subsets of 
$K_v^m$, for any $v$ (Archimedean or non-Archimedean) and $m\geq 1$, are finite unions of locally closed sets 
(by quantifier elimination, cf. \ref{ssec-venrich}), this does not hold for $\A_{\Q}$ as the following example shows.

\begin{ex}\cite{DM-ad} Let $X=\{a\in \A_K: Fin([[a\neq a^2]])\}$. Then $X$ 
is not a finite union of locally closed sets in adelic topology, equivalently, $X$ is not 
a Boolean combination of open sets (cf. \cite{DM-ad} for details).
\end{ex}

\medskip

\subsection{\bf Euler products and zeta values at integers}\label{ssec-intval} 

\

\medskip

Measures of definable sets in $\A_K^m$ are closely related to values of zeta functions at integers. The following is proved in \cite{DM-ad}.

\begin{thm}[Derakhshan-Macintyre {\cite{DM-ad}}]\label{zeta-measure} Let $n\geq 1$ be an integer.
\noindent
\begin{itemize}
\item $\zeta(n)^{-1}$, the Euler product $\prod_{p\equiv 1(mod~4)}(1-p^{-n})$, and Euler products of the form $\prod_{p\in S} (1-p^{-n})$, where $S$ is a set of primes of the form $\{p: \F_p \models \sigma\}$ 
and $\sigma$ is a sentence of the language of rings, are measures of $\cL_{rings}$-definable subsets of $\A_{\Q}$.
\item $\zeta(n)$ is the measure of an $\mathcal{L}_{Boolean}^{fin}(\cL_{Denef-Pas})$-definable subset of $\A_{\Q}$.
\end{itemize}
\end{thm}

In the proof we show that, for any $n\geq1$, the number $(1-p^{-n})^{-1}$ is the measure of a subset of $\Q_p$ that is $\cL_{Denef-Pas}$-definable independently of $p$ (with $n$ as its only parameter). 
The Euler product $\prod_{p\equiv 1(mod~4)}(1-p^{-n})$ relates to the zeta function of the quadratic field $\Q(i)$ (see \cite{HW}).
\begin{prob}{\cite{DM-ad}} Generalize Theorem \ref{zeta-measure} to number fields.\end{prob}
\begin{prob}{\cite{DM-ad}} What can one say about measures of definable subsets of $\A_K^n$, where $n\geq 1$?\end{prob}

\medskip

\subsection{\bf Definable subsets of the set of minimal idempotents}\label{ssec-defmin} 

\

\medskip

Recall the correspondence between minimal idempotents in $\A_K$ and valuations of $K$. 
The following question naturally arises.  What are the $\emptyset$-definable subsets 
of the set of minimal idempotents in $\A_K$? Note that the reason to have definability without parameters in the question 
is that if we allow parameters then every subset of the set of minimal idempotents is definable as is easily seen by taking sup and inf of idempotents. 

Let $g(x)$ be a polynomial over $\Z$ in a single variable $x$. Let $P(g)$ denote the set of primes $p$ such that the reduction of $g(x)$ modulo $p$ has a root in $\F_p$. 

In \cite{ax}, Ax proved that if $\sigma$ is an 
$\cL_{rings}$-sentence, then there are $g_1(x),\dots,g_n(x)\in \Z[x]$ such that $\{p: \F_p \models \sigma\}$ is a Boolean combination of the sets $P(g_1), \dots, P(g_n)$. 

The following gives an answer to the question above in terms of Ax's Boolean algebra.
\begin{thm}[Derakhshan-Macintyre {\cite{DM-ad}}]\label{th-def-pr} Let $K$ be a number field. 
Let $X$ be a parameter-free $\mathcal{L}_{rings}$-definable subset of the set of minimal idempotents in $\A_{K}$. Then the following hold.
\begin{itemize}
\item $X$ is a union of 
sets of the form $\{v: K_v \models \sigma\}$, where $\sigma$ is an $\cL_{rings}$-sentence, together with one of the following:

i) all Archimedean $v$,

ii) all real $v$,

iii) all complex $v$,

\item There is a finite subset $\mathcal{F}$ of $X$ containing all the minimal idempotents supported on the Archimedean valuations such that $X\setminus \mathcal{F}$ consists of minimal idempotents corresponding to 
valuations from a union of sets of the form $\{v: k_v\models \sigma\}$, where $\sigma$ is an $\cL_{rings}$-sentence and $k_v$ is the residue field of $K_v$. 

If $K=\Q$, then there is a finite Boolean combination $\frak B$ of $P(g_1).\dots,P(g_n)$ for some  $g_1(x),\dots,g_n(x)\in \Z[x]$ such that $X\setminus \mathcal{F}$ consists of the minimal idempotents corresponding to the primes from $\frak{B}$. 
\end{itemize}
\end{thm}

\section{\bf A question of Ax on decidability of all the rings $\Z/m\Z$}\label{sec-ax}

\medskip

\subsection{\bf Ax's question}\label{ssec-axq} 

\

\medskip

In his fundamental paper \cite{ax} on the model theory of finite and pseudofinite fields, Ax asked (Problem 5, page 270) if the elementary theory of all the rings $\Z/m\Z$, for all $m>1$, is decidable. In other terms, given a sentence $\phi$ of the language of rings, whether it is possible to decide that $\phi$ holds in $\Z/m\Z$, for all $m>1$. 

If we take the $m$ to range over the primes $p$, then decidability of a sentence in all the $\F_p$ is proved in 
\cite{ax} as a consequence of the axiomatization and decidability of the theory of pseudofinite fields. Such methods do not give decidability results for finite rings beyond fields which is necessary for solving Ax's problem. 

In \cite{DM-ad} Macintyre and myself gave a positive solution to Ax's problem by reduction to the $\cL_{rings}$-decidability of $\A_{\Q}$ using the definability of $Fin$ and of the Boolean algebra of idempotents. We give a sketch of this below. 
We also give a sketch of a proof of decidability of the rational adeles $\A_{\Q}$ which uses besides our formalism and machinery, only Ax's main result in \cite{ax}. 

This shows the usefulness of adelic methods (here on finite rings) where previous techniques would not suffice. We hope that adelic methods can be used in other decision problems too. 

Remark that 
similar decidability proofs via adeles, are given in works on D'Aquino and Macintyre related to a question of Zilber and on a model-theoretic analysis of quotients of non-standard models of Peano arithmetic, see \cite{PDAJM}.

\medskip

\subsection{\bf Reducing Ax's problem to adelic decidability}

\

\medskip

Let $\phi$ be an $\cL_{rings}$-sentence in prenex normal form
$$Q_1 x_1 \dots Q_m x_m \psi(x_1,\dots,x_n),$$
where  
$Q_i$ is either $\forall$ or $\exists$, and $\psi$ is a disjunction of conjunctions of the form 
$$f_1(x_1,\dots,x_n)=0\wedge \dots \wedge f_k(x_1,\dots,x_n)=0$$
$$\wedge ~ g_1(x_1,\dots,x_n)\neq 0 \wedge \dots \wedge g_r(x_1,\dots,x_n)\neq 0,$$
where $f_i$ and $g_j$ are polynomials over $\Z$.

Let $atom(x)$ denote the statement that $x$ is a non-zero minimal idempotent. 

Then $\phi$ holds in $\Z/m\Z$ for all $m>1$ if and only if the following holds:

for any $z\in \A_K$ if $[[z]]^{Arch}=0 \wedge Fin(supp(z))$ 
and 
$$\forall e (atom(e) \wedge supp(e) \subseteq supp(z))\Rightarrow \Phi_{val}(ez) \wedge \neg \Phi_{val}((ez)^{-1}),$$
then 
$$supp(z)\A_{\Q} \models Q_1 x_1 \dots Q_m x_m \exists y 
(f_1(x_1,\dots,x_n)=zy\wedge \dots \wedge f_k(x_1,\dots,x_n)=zy)$$
$$\wedge \neg \exists y (g_1(x_1,\dots,x_n)=zy \wedge \dots \wedge g_r(x_1,\dots,x_n)=zy).$$

Indeed, 
$supp(z)\A_{\Q}$ is the product of the $\Z_p$ where $p$ ranges over the finitely many primes $p_1,\dots,p_r$ 
corresponding the minimal idempotents in the support of $z$, and for every such $p$, $z(p)$ has positive $p$-adic valuation. So $supp(z)\A_{\Q}/z\A_{\Q}$ is isomorphic to the product 
$$\Z_{p_1}/p_1^{k_1}\Z_p\times \dots \times \Z_{p_r}/p_r^{k_r}\Z_p,$$
where 
$k_j$ is the $p_j$-adic valuation of $z(p_j)$ for each $1\leq j\leq r$. 
Thus  $\phi$ holds in all the rings $\Z/m\Z$ for all $m>1$ if and only if 
$\phi$ holds in $supp(z)\A_{\Q}/z\A_{\Q}$ for all such $z$, which is expressed by the above formula.

\medskip

\subsection{\bf Decidability of $\A_{\Q}$}\label{ssec-dec} 

\

\medskip

We give a sketch of the proof of decidability of the $\cL_{rings}$-theory of $\A_{\Q}$ 
due to Macintyre and myself in \cite{DM-ad} using only a theorem of Ax \cite{ax} (see also \cite{fs}, or Theorem 31.2.4 (a) in \cite{FJ}) and 
Corollary \ref{qe-fin}. The first proof of decidability of $\A_K$ (in the language of generalized products of Feferman-Vaught) is due to Weispfenning \cite{weisp-hab}. Our proof is simpler.

The following fundamental theorem about model theory of finite fields is what we need.

\begin{thm}[Ax \cite{ax}]\label{ax-almost-all} Let $\phi$ be an $\cL_{rings}$-sentence. It is possible to decide if $\phi$ is true in $\F_p$ for almost all $p$, 
and if so to list the exceptional primes.\end{thm} 

Now we give the adelic decision procedure. Let $\varphi$ be an $\cL_{rings}$-sentence. 
By Corollary \ref{qe-fin}, it suffices to decide the following statements:

\

$(I)$ ~ $Fin([[\psi]])$,

\

$(II)$ ~ $C_j([[\phi]])$,

\

\noindent where $\psi, \phi$ are $\cL_{rings}$-sentence.

\

{\it The decision procedure for $(I)$}. Since the number of Archimedean normalized valuations is finite, it suffices to decide $Fin([[\psi]]^{na})$ (which says $\psi$ holds in finitely many $\Q_p$). 
By Theorem \ref{ake} (Ax-Kochen-Ershov) there is an $\cL_{rings}$-sentence $\tau$ and an effectively computable $C>0$ such that for any prime $p\geq C$ 
$$\Q_p \models \psi \Leftrightarrow \F_p\models \tau.$$
Thus it suffices to decide whether $\tau$ holds in finitely many $\F_p$, which follows from \ref{ax-almost-all}.

\

{\it The decision procedure for $(II)$}. We want to decide if $\phi$ holds in at least $j$ many $K_v$, where $j\geq 1$ is given. By $(I)$ we can decide if $Fin([[\phi]])$ holds or not. If it does not hold, then $C_j([[\phi]])$ holds for all $j$. If $Fin([\phi]])$ holds, then consider 
$\psi:=\neg \phi$, which holds in almost all $\Q_p$, and the exceptional primes are exactly the primes $p$ 
where $\phi$ holds in $\Q_p$. By Theorem \ref{ax-almost-all}, we can list this finite set of primes and decide if this set has cardinality at least $j$ or not. 

This concludes the proof of decidability of $\A_{\Q}$.

\begin{remark} In \cite{DM-ad} we show that this proof can modified to prove decidability in the languages $\cL^+(L)$, where $\cL^+$ is $\cL_{Boolean}^{fin}$ or $\cL_{Boolean}^{fin,res}$ and $L$ is $\cL_{rings}, \cL_{Denef-Pas}$ or $\cL_{Basarab}$.
\end{remark}

\begin{note} We could not directly apply Theorem \ref{restricted-qe} or Feferman-Vaught \cite{FV} to the rings $\Z/m\Z$ to solve the problem of Ax. We needed to reduce to the decidability of adeles $\A_{\Q}$. For the decidability, to apply Ax, we needed the quantifier-elimination for generalized products given by \cite{FV} or Theorem \ref{restricted-qe}, and quantifier-elimination of the Boolean theory $T^{fin}$.\end{note}

\medskip

\subsection{\bf Decidability of $\A_K$}\label{ssec-decK}

In \cite{DM-ad2}, decidability is proved by Macintyre and myself for $\A_K$ for any number field $K$, and for 
$\A_K$ for all $K$ of bounded degree over $\Q$. 

The following question arises.

\begin{prob}\cite{DM-ad} Is the theory of $\A_K$ for all number fields $K$ (i.e. the set of sentences in some given language that hold in $\A_K$ for all $K$) decidable?\end{prob}

We can show the following.

\begin{thm}[Derakhshan-Macintyre {\cite{DM-ad}\cite{DM-ad2}}] The set of all $\cL_{rings}$-sentences that hold in $\A_K$ for all number fields $K$  is decidable if and only if for a given $p$, the theory of all finite extensions of $\Q_p$ is decidable.\end{thm}

This raises the question.

\begin{prob}\cite{DM-ad} \label{prob-allfin} Is there a suitable language $\cL$ such that given an $\cL$-sentence 
we can decide if it holds in all finite extensions of $\Q_p$?\end{prob}

Kochen \cite{kochen-local} proved decidability for the maximal unramified extension $\Q_p^{ur}$ of $\Q_p$. In \cite{DM-MC}, Macintyre and myself prove model-completeness in the language of rings for $\Q_p^{ur}$ and finitely ramified extensions of it, more generally for any Henselian valued field with finite ramification whose value group is a $\Z$-group, and we characterize model-complete perfect fields with procyclic Galois group. 

The essence of Problem \ref{prob-allfin} concerns model theory of infinitely ramified extensions of $\Q_p$. Even the abelian case this is out of reach, i.e. 
we do not know whether the maximal abelian extension $\Q_p^{ab}$ of $\Q_p$ is decidable, or any model theory for it.

\begin{prob} Use adelic methods (in the spirit of our solution to the Ax problem) combined with suitable Galois theory to approach model theory of infinitely ramified extensions of $\Q_p$.\end{prob}

\section{\bf Elementary equivalence and isomorphism for adele rings}\label{sec-elem}

In \cite{DM-ad2} Macintyre and myself consider the question of how the $\A_K$, as $K$ varies, are divided into elementary equivalence classes. 
The main tool used is Theorem \ref{CDLM-th} on uniform definition of valuation rings in the non-Archimedean completions of number fields. 

Given an adele ring $\A_K$, the completions $K_v$, are recoverable as the "stalks" $\A_K/(1-e)\A_K$ where $e$ is a minimal idempotent, and for any minimal idempotent, the quotient above is isomorphic to some $K_v$ (Archimedean or non-Archimedean). 

By Theorem \ref{CDLM-th} the valuation, maximal ideal, and residue field of the non-Archimedean stalks can be uniformly defined or interpreted using sentences from the language of rings. 
Furthermore, we have a uniform definition, independent of $K$ but depending on $p$, of the collection of stalks with residue characteristic $p$ for any given $p$. 

Let $p$ be a prime in $\Z$. Then $p$ lifts to finitely many primes $\cP_1,\dots,\cP_r$ in $\cO_K$, and we have the decomposition
$$p\cO_K=\cP_1^{e_1}\dots \cP_r^{e_r}.$$
$e_i=e_i(\cP_i/p)$ is called the ramification index of $\cP_i$ over $p$. 
$\cO_K/\cP_i$ is a finite extension of $\F_p$ of dimension $f_i=f(\cP_i/p)$ over $\F_p$ which is called the residue degree of $\cP_i$ over $p$. 

If $K_{\cP_i}$ is the completion of $K$ at $\cP_i$, 
then $\Q_p\subseteq K_{\cP_i}$, as valued fields, and $e_i$ and $f_i$ are respectively the ramification index and residue field degree of $K_{\cP_i}$. We have the fundamental 
inequality 
$$\sum_{i=1}^{r} e_if_i=[K:\Q].$$ 
Note that 
$$e_i,f_i\leq [K:\Q].$$
The prime $p$ is said to be unramified if $e_i=1$ for all $i$, and ramified otherwise. $p$ splits completely if in addition all $f_i=1$.

\medskip

\subsection{\bf The number field degree}
 
\

\medskip

\begin{thm}[Derakhshan-Macintyre {\cite{DM-ad2}}]\label{deg} $\A_{K_1}\equiv \A_{K_2}$ implies that $[K_1:\Q]=[K_2:\Q]$.\end{thm}

We give an idea of the proof. Let $K$ be a number field $K$. To detect the dimension of $K$ over $\Q$ inside $\A_K$ in a first-order way we first find a 
prime $p$ that splits completely in $K$. Then by the fundamental inequality we must have 
$$r=n=[K:\Q].$$
So we can define $[K:\Q]$ as the number of minimal idempotents $e$ such that $\A_K/(1-e)\A_K$ has 
residue field $\F_p$ and $v(p)=1$, i.e. $v(p)$ is the minimal positive element of the value group of $e\A_K \cong K_{v_e}$, where $v$ denotes the valuation of $K_{v_e}$. This 
can be expressed by an $\cL_{rings}$-sentence independently of $K$ (but depending on $p$) by Theorem \ref{CDLM-th}. 

To get a prime $p$ that splits completely in $K$, take the normal closure $L$ of $K$. By the  
Chebotarev density theorem (see \cite{manin-book},\cite{nkrch}) there are infinitely many primes $p$ that split 
completely in $L$. It follows that $p$ splits completely in $K$.  

\begin{ex} Note that the converse of Theorem \ref{deg} does not hold e.g. $K_1=\Q(\sqrt{2}),~ K_2=\Q(\sqrt{3})$. \end{ex}

\medskip

\subsection{\bf The case of normal extensions}

\

\medskip

\begin{thm}[Derakhshan-Macintyre {\cite{DM-ad2}}] Suppose that $K$ is normal over $\Q$. If $L$ is a number field such that $\A_L$ is elementarily equivalent to 
$\A_K$, then $L=K$.\end{thm}
The proof uses a corollary of the Chebotarev density theorem that states that 
if $K$ is a Galois extension of $\Q$, then $K$ is completely determined by the rational primes that split completely in $K$ (see \cite{nkrch}, Corollary 13.10 page 548). 

\subsection{\bf Splitting types and arithmetical equivalence}

\

\medskip

Let $p$ be a prime. We do not assume that $p$ is unramified in $K$. 
 The splitting type of 
$p$ in $K$ is a sequence 
$$\Sigma_{p,K}=(f_1,\dots,f_r),$$
where $f_1\leq \dots \leq f_r$ is such that  
$p\cO_K=\cP_1^{e_1}\dots \cP_r^{e_r}$ and $f_j$ is the residue degree of $\cP_j$. Note that there can be 
repetitions and that the ramification indices $e_i$ are not present. 

For a splitting type $A$, define 
$$P_K(A)=\{p: \Sigma_{p,K}=A\}.$$ 
Note that 
$P_K(A)$ is empty for all but finitely many $A$ (since $\sum_{j=1}^{r} f_j \leq [K:\Q]$).

Let $K$ be a number field. The (Dedekind) zeta function of $K$ is defined by $\zeta_K(s)=\sum_{\frak a\in Spec(\cO_K)} N(\frak a)^{-s}$, where $N(\frak a)=[\cO_K:\frak a]$. 

In \cite[Theorem 1]{perlis} Perlis proves 
that if $K_1$ and $K_2$ are number fields, then $\zeta_{K_1}(s)=\zeta_{K_2}(s)$ if and only 
$P_{K_1}(A)=P_{K_2}(A)$ for all $A$. In this case 
$K_1$ and $K_2$ are said to be arithmetically equivalent. 

By \cite[Theorem 1]{perlis}, if $K_1$ and $K_2$ are arithmetically equivalent, then they have the same discriminant, the same number of real (resp. complex) absolute values, the same normal closure and unit groups.

It follows from Theorem \ref{CDLM-th} that if 
$\A_K\equiv \A_L$, then for each $p$, 
$$\Sigma_{p,K_1}=\Sigma_{p,K_2}.$$
Applying Hermit's theorem that there are only finitely many number fields with discriminant bounded by any given positive integer (see \cite{nkrch}), one can deduce the following.

\begin{thm}[Derakhshan-Macintyre {\cite{DM-ad2}}] For any given number field $K$, there are only finitely many number fields $L$ such that are $\A_K$ and $\A_L$ are elementarily equivalent.\end{thm}

This raises the question.
\begin{prob} Given a number field $K$, classify the number fields $L$ such that $\A_L$ is elementarily equivalent to $\A_K$. What are the elementary invariants?\end{prob}

\medskip

\subsection{\bf Elementary equivalence of adele rings - a rigidity theorem}

\

\medskip

The question asking to what extent a number field is determined by its zeta function has a long history. 

A number field $K$ that is isomorphic to any number field $L$ such that $\zeta_K(s)=\zeta_L(s)$ is called arithmetically solitary. Examples are any normal extension of $\Q$. The first nonsolitary field was discovered by Gassman in 1925 who gave two fields of degree 180 over $\Q$ which are arithmetically equivalent but not isomorphic (cf. \cite{perlis}).

By a theorem of Uchida \cite{uchida}, two number fields $L$ and $K$ are isomorphic if and only if their absolute Galois groups $G_{K_1}$ and $G_{K_2}$ are isomorphic, a theorem in the realm of Grothendieck's anabelian conjectures.

Iwasawa \cite{iwasawa} proved that for number fields $K$ and $L$, if $\A_K$ is isomorphic to $\A_L$, then 
$\zeta_K(s)=\zeta_L(s)$. The converse to Iwasawa's theorem relates to interesting questions. The converse is not true in general, but is true if the extensions are Galois, see \cite{perlis}.

From Perlis' \cite[Theorem 1]{perlis} and Theorem \ref{CDLM-th} it follows that if $\A_K$ and $\A_L$ are elementarily equivalent, then $\zeta_K(s)=\zeta_L(s)$. 

In \cite{DM-ad2} Macintyre and myself prove that elementary equivalence does determine the adele rings
up to isomorphism, giving a converse to Iwasawa's theorem under a stronger hypothesis. This is a first-order "rigidity theorem " for adeles.
\begin{thm}[Derakhshan-Macintyre {\cite{DM-ad2}}]\label{isom} Let $K$ an $L$ be number fields. If $\A_K$ and $\A_L$ are elementarily equivalent (as rings), then they are isomorphic. 
\end{thm}

The proof uses a theorem of Iwasawa in \cite[pages 331-356]{iwasawa} that for number fields $K$ and $L$, the adele rings 
$\A_K$ and $\A_L$ are isomorphic if and only there there is a bijection 
$\phi: V_{K}^{fin} \rightarrow V_{L}^{fin}$ such that the completions $K_v$ and 
$L_{\phi(v)}$ are isomorphic for all $v\in V_{K}^{fin}$. This condition is also equivalent to the condition that the finite adeles $\A_K^{fin}$ and $\A_L^{fin}$ are isomorphic (cf. \cite{DM-ad2}).

\begin{prob} Find conditions under which adele rings are isomorphic.\end{prob}

We also pose.

\begin{prob}\label{prob-isom} Does Theorem \ref{isom} extend to algebraic groups $G$? Find algebraic groups $G$ over $\Q$ such that if $G(\A_K)$ and $G(\A_L)$ are elementarily equivalent in the language of groups, then 
they are isomorphic. Is this true when $G$ is a $\Q$-split semi-simple algebraic group over $\Q$?\end{prob}
We note that one believes that for a $\Q$-split semi-simple algebraic group $G$, the field $\Q_p$ is definable in the group $G(\Q_p)$. It would be interesting to investigate adelic versions of this and use it to approach Problem \ref{prob-isom}.

\section{\bf Axioms for rings elementarily equivalent to restricted direct products and converse to Feferman-Vaught}\label{sec-axioms}

\medskip

\subsection{\bf The question and connection to nonstandard models of Peano arithmetic}

\

\medskip

The question of finding axioms for the theory of $\A_K$ is part of the general question of finding 
axioms under which any commutative unital ring is elementarily equivalent to a restricted direct product of
connected rings (a ring is connected if $0,1$ are the only idempotents). 

This problem is solved in joint work with 
Macintyre in \cite{DM-axioms} and is based on the work of D'Aquino and Macintyre in \cite{elem-prod} solving the case of products, which was in turn used by D'Aquino and Macintyre to answer a question of Zilber's on models of Peano arithmetic. The problem asks for a non-standard model $\cM$ of PA and $k\in \cM$, whether $\cM/k\cM$ interprets arithmetic. The solution in \cite{PDAJM} is that it does not interpret arithmetic and much more is proved around its model-theoretic tameness. 

Another ingredient in the Macintyre-D'Aquino solution to Zilber's problem is the work D'Aquino-Macintyre and myself in \cite{dpm} on truncated ordered abelian groups. 

In this work we provide axioms for a class of linear orders with addition called truncated ordered abelian groups, and prove that any model of these axioms is an initial segment of an ordered abelian group, thus has a semi-group structure arising from a process of truncation. 
This work applies to quotients of valuation rings with truncated valuations. We remark that Zilber's question 
was inspired by model-theoretic insights into quantum mechanics.

\medskip

\subsection{\bf Axioms for the rings}

\

\medskip

We now discuss the axioms of Macintyre and myself from \cite{DM-axioms}. We shall then 
prove an analogue of the results of Feferman-Vaught \cite{FV} and the results in Section \ref{sec-rest} for commutative unital rings. So we develop the analogue of the required notions (e.g. Boolean values) in the case of rings. 

Let $R$ be a commutative unital ring. The set 
$$\B=\{x\in R: x=x^2\}$$ 
of idempotents 
is a Boolean algebra with operations  
$$e \wedge f=ef,$$
$$\neg e=1-e,$$
$$e\vee f=1-(1-e)(1-f)=e+f-ef.$$
$\B$ carries an ordering defined by 
$e\leq f\Leftrightarrow ef=e$, which is $\cL_{rings}$-definable. The {\it atoms} of $\B$ are by definition the minimal idempotents that are not equal to $0,1$. 

For any $e$ in $\B$, 
$$R/(1-e)R \cong eR \cong R_e,$$
where $R_e$ is the localization of $R$ at $\{e^n: n\geq 0\}$. The first isomorphism is straightforward and the second isomorphism is shown in Lemma 1 in \cite{elem-prod}. $R_e$ is the stalk of $R$ at $e$. Of special important are the $R_e$ for {\it atoms} $e$.

We define Boolean values in the case of rings as follows. 
\begin{Def} Let $\Theta(x_1,\dots,x_n)$ be a formula of the language of rings, and $f_1,\dots,f_n\in R$. Then 
$[[\Theta(f_1,\dots,f_n)]]$ is defined to be 
$$\bigvee_{e} \{e: e~\text{an~atom},~R_e\models \Theta((f_1)_e,\dots,(f_n)_e)\}$$
provided $\bigvee$ exists in $\B$, where $f_e$ is the image of $f$ in $R_e$.\end{Def}
Note that $f_e$ can be identified with $f+(1-e)R$ using the above isomorphism.

We augment the language of rings $\cL_{rings}$ by a unary predicate symbol 
$Fin(x)$ that is interpreted in $R$ as a finite support element, i.e. a finite union of atoms. Let $\mathcal{F}in$ denote the ideal of finite support elements in $R$. 

Let $\cL_{rings}^{fin}=\cL_{rings}\cup \{Fin(x)\}$. We fix an $\cL_{rings}$-formula $\varphi(x)$ in the single variable $x$. 

Let $\mathcal{A}_{\varphi}$ denote the following axioms expressed as $\cL_{rings}^{fin}$-sentences. As in Subsection \ref{ssec-benrich}, $T^{fin}$ denotes the theory of infinite atomic Boolean
algebras in the language $\cL_{Boolean}^{fin}$.

\

{\bf Axiom 1.} $\B$ is atomic.

\

{\bf Axiom 2.} $[[\Theta(f_1,\dots,f_n)]]$ exists (an an element of $\B$).

\

{\bf Axiom 3.} For any atomic formula $\Theta(x_1,\dots,x_n)$ of the language of rings,
$$R\models \Theta(f_1,\dots,f_n) \Leftrightarrow \B\models [[\Theta(f_1,\dots,f_n)]]=1.$$

\

{\bf Axiom 4.} $(\B,\mathcal{F}in)\models T^{fin}$, and for all $\cL_{rings}$-formulas 
$\Theta(x_1,\dots,x_n,w)$ and $f_1,\dots,f_n\in R$ there is a $g\in R$ such that if 
$$[[\exists w \Theta(f_1,\dots,f_n,w)]]  \cap \neg [[\exists w (\varphi(w) \wedge \Theta(f_1,\dots,f_n,w))]]\in \mathcal{F}in,$$
then
$$[[\exists w \Theta(f_1,\dots,f_n,w)]]\cap \neg [[\Theta(f_1,\dots,f_n,g)]]\in \mathcal{F}in.$$

\

\begin{note} A special case of Axiom 4 is the following. 

\

{\bf Axioms 4'.} For all $\Theta(x_1,\dots,x_n,w)$ and $f_1,\dots,f_n\in R$, there is a $g\in R$ such that if 
$$[[\exists w (\varphi(w) \wedge \Theta(f_1,\dots,f_n,w))]]$$ 
is cofinite in 
$$[[\exists w \Theta(f_1,\dots,f_n,w)]],$$ 
then $[[\exists w \Theta(f_1,\dots,f_n,w)]]$ is cofinite in 
$[[\Theta(f_1,\dots,f_n,g)]]$. Here "cofinite" really means cofinite.\end{note}

\

{\bf Axiom 5.} $\forall x (Fin([[\neg \varphi(x)]]))$.

\

Let $(\cM_i)_{i\in I}$ be a family of $\cL_{rings}$-structures. Axioms 1-5 hold $\prod_{i\in I}^{(\varphi)} \cM_i$, the restricted product of $\cM_i$ 
with respect to $\varphi(x)$ (for Axiom 4 use Axiom of Choice).

\medskip

\subsection{\bf The ring-theoretic Feferman-Vaught and converse to Feferman-Vaught}


\begin{thm}[Derakhshan-Macintyre {\cite{DM-axioms}}]\label{main-th} Let $\varphi(\bar x)$ be an $\cL_{rings}$-formula. Let $R$ a commutative unital ring satisfying the axioms $\mathcal{A}_{\varphi}$. Then 
for each $\cL_{rings}$-formula $\Theta(x_1,\dots,x_m)$ there is, by an effective procedure, $\cL_{rings}$-formulas
$$\Theta_1(x_1,\dots,x_m),\dots,\Theta_k(x_1,\dots,x_m)$$ and an 
$\cL_{Boolean}^{fin}$-formula 
$\psi(y_1,\dots,y_k)$ such that for all $f_1,\dots,f_m$ in $R$
$$R\models \Theta(f_1,\dots,f_m) \Leftrightarrow $$
$$(\B,\mathcal{F}in)\models \psi([[\Theta_1(f_1,\dots,f_m)]],\dots,[[\Theta_k(f_1,\dots,f_m)]]).$$\end{thm}

Since $R$ and the restricted product $\prod_{e\text{~atom~of~}\B}^{(\varphi)} R_e$ have the same idempotents, the same ideal $\mathcal{F}in$, and same localization $R_e$ for all atoms $e$, the same restricting formula $\varphi$, and satisfy the axioms $\mathcal{A}_{\varphi}$, Theorem \ref{main-th} implies the following.
\begin{cor}[Derakhshan-Macintyre {\cite{DM-axioms}}] Let $\varphi(\bar x)$ be an $\cL$-formula and $R$ a commutative unital ring satisfying the axioms $\mathcal{A}_{\varphi}$. Then 
$$R \equiv \prod_{e\text{~atom~of~}\B}^{(\varphi)} R_e,$$
the restricted direct product with respect to $\varphi$.\end{cor}

This result can be regarded as a converse to Theorem \ref{restricted-qe} and the theorems of Feferman-Vaught \cite{FV}).

These axioms for restricted products connect well with the issues on elementary invariants for adele rings discussed in Section \ref{sec-elem} and in \cite{DM-ad2}. 

\begin{remark} Consider the $\cL_{rings}$-formula $\Phi_{val}(x)$ that defines the valuation ring of all non-Archimedean $K_v$ from Subsection \ref{ssec-deffin}, and the associated axiom system $\mathcal{A}_{\Phi_{val}(x)}$. If we 
augment $\mathcal{A}_{\Phi_{val}(x)}$ by the axioms for $p$-adically closed fields (in \cite{AK2} or \cite{PR-book}) in all the stalks $e\A_{\Q}$ where $e$ is non-Archimedean, and the axioms for 
real closed fields in all the $e\A_{\Q}$ where $e$ is real, then we get a complete system of axioms for adeles $\A_{\Q}$. See \cite{DM-axioms}.\end{remark}

\section{\bf Some stability theory}\label{sec-stab}

\medskip

\subsection{\bf Stable embedding}\label{ssec-stemb} 

\

\medskip

It is known that for many Henselian valued fields, the value group and the residue field are stably embedded. See \cite{HHM}. In \cite{DM-supp}, we show that the local fields $K_v$ are stably embedded in the adeles $\A_K$ (via the identification of 
$K_v$ with $e_{\{v\}}\A_K$). 
\begin{thm}[Derakhshan-Macintyre {\cite{DM-supp}}] Let $X$ be a definable subset of $\A_K^n$ with parameters from $\A_K$, where $n\geq 1$. 
Let $e$ be a minimal idempotent. 
Then $X\cap (e\A_K)^n$ is definable with parameters from $e\A_K$.
\end{thm}

\begin{prob}\cite{DM-supp} Prove a general stable embedding theorem for the factors of a restricted product of structures with respect to a formula (defined in Section \ref{sec-rest}).\end{prob}

In Subsection \ref{ssec-venrich} we defined the product valuation $\prod v$ from the finite adeles $\A_K^{fin}$ into the 
restricted product $\Gamma$ of the lattice-ordered monoids $\Z\cup\{\infty\}$ indexed by the non-Archimedean valuations.
$\Gamma$ is interpretable in the ring $\A_K^{fin}$, cf. \cite{DM-supp}.

\begin{thm} [Derakhshan-Macintyre {\cite{DM-supp}}]The value monoid $\Gamma$ of $\A_{\Q}^{fin}$ is not stably interpreted via the product valuation map.
\end{thm}
In the proof we define the subset $X$ of $\A_{\Q}$ consisting of 
idempotents which are supported exactly on the primes 
$p$ that are congruent to $1$ modulo $4$. For this, let $\Psi$ be a sentence that holds in $\Q_p$ for exactly 
the primes $p$ that are congruent to $1$ modulo $4$, and let $\Psi'$ be a sentence that holds in all 
non-Archimedean local fields and fails in all the Archimedean local fields. Then 
$$X=\{x\in \A_{\Q}: supp(x)=[[\Psi \wedge \Psi']]\}.$$

The image of $X$ under the product valuation $\prod v$ 
is the set $Y$ of all $g$ in 
$$\prod_p (\Z\cup \{\infty\})$$ which are $0$ at $p$ 
and $\infty$ elsewhere. Applying the Feferman-Vaught Theorem or Theorem \ref{restricted-qe} to $\Gamma$ and using the Presburger quantifier elimination for the factors (cf. \cite{enderton-book}), it follows
that $X$ is not definable in the value monoid

\medskip

\subsection{\bf The tree property of second kind}\label{ssec-tree} 

\

\medskip

The property of not having the tree property of the second kind $NTP_2$ is a generalization of the properties of being 
simple and $NIP$ (the negation of the independence property). 

It is known 
that ultraproducts of $\Q_p$ and certain valued difference fields 
have $NTP_2$ (cf.\ \cite{CH}).  

The theory of $\A_K$, for $K$ a number field, has the independence property in two different ways, firstly via the 
residue fields by Duret \cite{Duret} and \cite{FJ}, and secondly because the definable Boolean algebra $\B_K$. 
\begin{thm}[Derakhshan-Macintyre {\cite{DM-supp}}]\label{thm-ntp} The theory of finite adeles $\A_K^{fin}$ and the theory of adeles 
$\A_K$ do not have the property $NTP_2$.
\end{thm}

\medskip

\subsection{\bf Stable formulas and definable groups}

\

\medskip

Local stability theory is the study of stability properties of a formula. There is much literature on this. Here we only mention that Hrushovski and Pillay in \cite{udi-pillay-groups} develop a 
unified approach to local stability for $\Q_p$, $\R$, and pseudo-finite fields. The central notion being that of {\it a geometric field}. 
\begin{prob} Develop local stability theory for adeles or adele spaces of algebraic varieties (cf. Section \ref{sec-adgeom}) using local stability for the fields $K_v$ using the notion of geometric fields in the sense of Hrushovski-Pillay. In particular, what can one say about 
restricted products of geometric fields? To what extent the model-theoretic properties of geometric fields are preserved under restricted products?\end{prob}
In \cite{udi-pillay-groups}, theorems are proven about groups definable in $\R$ and $\Q_p$ showing they are related to algebraic groups. 
\begin{prob} What can one say about a group that is definable in $\A_K$? Is it related (e.g. virtually isogenous) to $G(\A_K)$ for an algebraic group $G$?\end{prob}

\section{\bf Adele geometry}\label{sec-adgeom}

\medskip

\subsection{\bf Adele spaces of varieties}\label{ssec-adsp} 

\

\medskip

Adele spaes of algebraic varieties were defined by Weil (cf. \cite{weil-adeles-gps}) and are important in number theory and arithmetic geometry. 
In joint work with Macintyre \cite{DM-ad} we show that the results of Section \ref{sec-rest} apply to these spaces and they admit an internal quantifier elimination and Feferman-Vaught theorem in a natural geometric language.

For simplicity, let $V$ be an affine variety over a number field $K$. Noncanonically 
choose $m$ and polynomials $f_1,\dots,f_e \in \Q[x_1,\dots,x_n]$ whose zero set is $V$. For convenience we assume 
$K=\Q$. 

For any valuation $v$, consider the sets $V(K_v)$ and $V(\cO_v)$. The adele space $V(\A_K)$ is defined as 
the restricted direct product of the $V(K_v)$ with respect to the $V(\cO_v)$. This is the union $\bigcup_{S} V_S$, where
$$V_S=\prod_{v\in S} V(K_v) \times \prod_{v\notin S} V(\mathcal{O}_v),$$
as $S$ ranges over all finite subsets of $V_K$ containing all the Archimedean valuations. 

Remark that the adele space can be defined for any abstract variety $V$ and is defined independently of the choice of an affine covering, see Weil's \cite{weil-adeles-gps}, Chapter 1.2. 

Note that $V(\A_K)$ coincides with the set of solutions of the polynomials $f_1,\dots,f_e$ in $\A_K^n$, and so is naturally a definable set in $\A_K^n$ in the language of rings. Nevertheless, it is important to show that 
$V(\A_K)$ can be represented as a restricted direct product in the sense of Section \ref{sec-rest} as that will give an internal connection between definability and the measure theory of $V(\A_K)$, and also yields an internal 
quantifier-elimination and Feferman-Vaught theorem for $V(\A_K)$.

This is done in \cite{DM-ad} as follows. We consider $V(K_v)$ as a subvariety of $K_v^n$, uniformly in $v$. Define the relational 
language $L^V$ to consist of predicates $R_W$ corresponding to all $\Q$-subvarieties $W$ of $V^l$  for all $l=1,2,3,\dots$. 
Suppose $R_W$ has arity $l$. We interpret $R_W$ in $V(K_v)$ as $W(K_v)$. Note that this is a subset of $V(K_v)^l$.

It is easy to see that every subset of $V(K_v)^l$ defined by an $\cL_{rings}$-formula $\psi(x_,\dots,x_l)$ is $L_V$-definable, uniformly in $v$. 
By Theorem \ref{CDLM-th} on the 
uniform $\cL_{rings}$-definability of $\cO_v$ in $K_v$, the rings $\cO_v$ are uniformly $L_V$-definable. 
So there is an $L^V$-formula $\Phi_V(x)$ that uniformly defines $V(\cO_v)$ in $V(K_v)$, for all $v$. 

Thus $V(\A_K)$ can be represented as the restricted direct product $\prod_{v\in V_K}^{(\Phi_V)} V(K_v)$ relative to $\Phi_V(x)$. Theorem \ref{restricted-qe} then gives quantifier elimination and 
a Feferman-Vaught theorem. 


\begin{thm}[Derakhshan-Macintyre {\cite{DM-ad}}]\label{adele-space-th} Let $\cL_{Boolean}^{+}$ be an extension of $\cL_{Boolean}$ containing $Fin(x)$. 
For any $\cL_{Boolean}^{+}(L^V)$-formula $\Psi(x_1,\dots,x_n)$, one can effectively construct 
$L^V$-formulas 
$$\Psi_1(x_1,\dots,x_n),\dots,\Psi_m(x_1,\dots,x_n)$$ 
with the same free variables as $\Psi$, 
and an $\cL_{Boolean}^+$-formula $\Theta(X_1,\dots,X_m)$ 
such that for any $a_1,\dots,a_n\in V(\A_K)$, 
$$V(\A_K)\models \Psi(a_1,\dots,a_n)$$ if and only if 
$$\P(I)^+\models \Theta([[\Psi_1(a_1,\dots,a_n)]],\dots,[[\Psi_m(a_1,\dots,a_n)]]).$$
\end{thm}

\begin{cor}[Derakhshan-Macintyre {\cite{DM-ad}}]\label{adele-space-cor} 
An $\cL_{Boolean}^{fin}(L^V)$-definable subset of $V(\A_K)$ is a Boolean combination of sets defined by the formulas 
\begin{itemize}
\item $Fin([[\psi(x_1,\dots,x_n)]])$,
\item $C_j(\phi([[x_1,\dots,x_n]]))$,
\end{itemize} 
where $j\geq 1$, and $\psi$ and $\phi$ are $L^V$-formulas.
\end{cor}
\begin{proof} Immediate by Theorem \ref{adele-space-th} and Theorm \ref{bool1}.\qed\end{proof}

\begin{note} If $W$ is a subvariety of $V$, both defined over $\Q$. Then $W(\A_K)$ is an $\cL_{Boolean}^{fin}(L^V)$-definable subset of $V(\A_K)$.
using the Boolean condition $[[..]]=1$.\end{note}

\begin{note} We do not have idempotents in $V(\A_K)$ as we have for the case of adeles $\A_K$ as $V(\A_K)$ is merely a locally compact topological space and not a ring.\end{note}

Theorem \ref{adele-space-th} and Corollary \ref{adele-space-cor} are suitable for proving results on definable subsets of $V(\A_K)$ and their measures.


\



\medskip

\subsection{\bf Tamagawa measures on adele spaces}\label{ssec-tam} 

\

\medskip

Let $V$ be a smooth algebraic variety defined over $K$. 
Let $\omega$ an algebraic differential form on $V$ defined over $K$ of top degree $n=dim(V)$. 
For any valuation $v$ of $K$, the form $\omega$ induces a measure $\omega_v$ on the topological space $V(K_v)$. See Weil's book \cite{weil-adeles-gps}, Chapter 2.2.
 
The measure on $V(\A_K)$ may or may not converge. It converges when $V$ is a semisimpe algebraic group (see \cite{weil-adeles-gps}). In many other cases it diverges, and one has to use convergence factors. 
Let $$\mu_v(V)=\int_{V(\mathcal{O}_v)} \omega_v.$$
A set of convergence factors for $V$ is defined to be a collection $(\lambda_v)_v$ of strictly positive real numbers indexed by the valuations $v\in V_K$ such that the 
product $$\prod_{v\in V_K^{fin}} \lambda_v^{-1} \mu_v(V)$$ is absolutely convergent.

The Tamagawa measure $\tau_V$ on $V(\A_K)$ derived from the form $\omega$ by means of the convergence factors $(\lambda_v)_v$ is defined to be the measure on $V(\A_K)$ inducing in each product
$$\prod_{v\in S} V(K_v) \times \prod_{v\notin S} V(\A_K)$$
the product measure 
$$\mu_K^{-dim(V)} \prod_{v\in V_K} (\lambda_v^{-1}\omega_v),$$
for any finite subset $S$ of $V_K$ containing all the 
Archimedean valuations, where $\mu_k=|\Delta_K|^{1/2}$ and $\Delta_K$ is the discriminant of $K$. One usually puts $\lambda_v=1$ for all Archimedean $v$.

\

(i) {\bf $GL_n$ and $SL_n$.} The Tamagawa measure on the additive group $\Bbb G_a(\A_K)$ and on $SL_m(\A_K)$ are convergent as 
$\int_{\Bbb G_a(\mathcal{O}_v)} \omega_v=1$, and 
$$\int_{SL_m(\mathcal{O}_v)} \omega_v=(1-q_v^{-2})\dots (1-q_v^{-m}).$$
The Tamagawa measure on the multiplicative group $\Bbb G_m(\A_K)$ and on $GL_m(\A_K)$ are divergent since
$\int_{\Bbb G_m(\mathcal{O}_v)} \omega_v=1-q_v^{-1}.$ and 
$$\int_{GL_m(\mathcal{O}_v)} \omega_v=(1-q_v^{-1})\dots (1-q_v^{-m}).$$
In these divergent cases, we can use the convergence factors $\lambda_v=(1-q_v^{-1})$ for $v\in V_K^{fin}$, 
and $\lambda_v=1$ for Archimedean $v$, to get a convergent adelic measure. Here, as before, $q_v$ is the cardinality of the residue field of $K_v$.

\

(ii) {\bf Hypersurfaces.} Generalizing work of Weil and Tamagawa (see \cite{weil-adeles-gps})
around an adelic interpretation and proof of the Siegal mass formula on quadratic forms, 
Ono \cite{ono-int} studied adele spaces of hypersurfaces. 

Let $f(x_1,\dots,x_n)$ be a non-constant absolutely irreducible polynomial over $K$. 
Let $\Omega$ be a universal domain containing $K$. Let 
$$W_f=\{(x_1,\dots,x_n)\in \Omega^n: F(x_1,\dots,x_n)\neq 0\}.$$
Identify this $K$-open set with the non-singular hypersurface 
$$\{(x_1,\dots,x_n,y)\in \Omega^{n+1}: F(x_1,\dots,x_n)y=1\}.$$
Then 
$$W_f(\A_K)=\{(x_1,\dots,x_n) \in \A_K^n: f(x_1,\dots,x_n)\in \I_K\}.$$

The $K$-open subset $W_f$ of $\Omega^n$, carries a gauge form $\omega$ induced from the 
form $dx_1\dots dx_n$ on $\Omega^n$. Ono \cite{ono-int} proved that the numbers defined by $\lambda_v=1-q_v^{-1}$ when $v$ is non-Archimedean, and $\lambda_v=1$ if $v$ is Archimedean, form convergence factors for $W_f$. Let $\tau_{W_f}$ denote  the Tamagawa measure on $W_f(\A_K)$ derived from $\omega$ by means of $\lambda_v$. 

\

(iii) {\bf Zeta integrals.} In Tate's thesis \cite{tate-thesis} on the analytic continuation and functional equation for the zeta-function of a number field and more general zeta integrals associated to characters of the idele class group $\I_K/K^*$, the measure that is used on $\I_K$ has its local factors of the form $(1-q_v^{-1})^{-1}d^*x_v$ at each non-Archimedean $v$ where $q_v$ is the cardinality of the residue field of $K_v$ (with certain normalizations for the Archimedean factors). Here the convergence factors are $(1-q_v^{-1})$. 

See Subsection \ref{ssec-char} for a description of the zeta integrals in Tate's thesis and model-theoretic approach and questions.

In \cite{DM-ad}, we prove  that the above normalization factors are uniformly definable. 

\begin{thm}[Derakhshan-Macintyre {\cite{DM-ad}}]\label{conv-fac} Let $K$ be a number field. There is an $\cL_{Denef-Pas}$-definable set of convergence factors (defined independently of $p$ and without parameters) for the following measures on adelic points:
\begin{itemize}
\item The Tamagawa measures on $\Bbb G_m(\A_K)$ and $GL_m(\A_K)$,
\item The Tamagawa measure on $W_f(\A_K)$, where $f$ is an absolutely irreducible polynomial over $K$.
\item The measure used by Tate on the ideles $\I_K$ for analytic continuation and functional equation for zeta functions and zeta integrals of characters.\end{itemize}\end{thm}

Uniform definability of the convergence factors enables the use of model-theoretic tools to evaluate the local $p$-adic integrals with respect to measures induced from differential forms, following Denef, Loeser, Cluckers, and others on motivic integration, cf. \cite{Denefrationality}, \cite{DL}, \cite{CL2}. Then the results of \cite{zeta1} which are stated in Section \ref{ssec-euler} yield analytic properties of the Euler product of the local integrals as global integrals. See Sections \ref{ssec-char} and \ref{ssec-lang} 
for examples of Euler products in connection with Tate's thesis and the Langlands program.

In the model-theoretic approach to $p$-adic and motivic integration one integrates functions of the form $|f(x)|^s$ over definable sets, where $f$ is a definable function from $\Q_p^n\rightarrow \Q_p$ (or a finite extension of $\Q_p$). 
It is important to try to integrate other functions. Some ideas and guiding themes for this can be found in Section \ref{ssec-lang} in the context of automorphic forms.


\begin{prob}\cite{DM-ad}\label{adelicKZ-prob} What can be said about the numbers that are Tamagawa measures of adelic spaces of varieties?
\end{prob}

The work of Kontsevich and Zagier in \cite{KZ} concerns periods which are complex numbers whose real and imaginary parts are absolutely convergent integrals, over real semi-algebraic subsets of $\R^n$, of rational functions with rational coefficients.

Problem \ref{adelicKZ-prob} can be regarded as an adelic version of some questions of Kontsevich-Zagier on numbers that arise as periods. One expects the Tamagawa measures of definable sets to be related to $L$-functions, cf. \cite{manin-book}.


\section{\bf Boolean Presburger predicates and Hilbert symbol}\label{sec-recip}

As in \ref{ssec-benrich} $\cL_{Boolean}^{fin,res}$ denotes the extension of $\cL_{Boolean}^{fin}$ got by adding the Presburger predicates $Res(n,r)(x)$ 
for all $n,r$. We consider $\A_K$ as a structure for the language $\cL^{fin,res}_{Boolean}(\cL_{rings})$. 

\begin{thm} [Derakhshan-Macintyre {\cite{DM-ad}}]\label{dec-pres} The $\cL^{fin,res}_{Boolean}(\cL_{rings})$-theory of $\A_K$ is decidable and has quantifier elimination.
\end{thm}

This follows from the decidability of the theory of infinite atomic Boolean algebras $T^{fin,res}$ in the language $\cL_{Boolean}^{fin,res}$ proved by Macintyre and myself in \cite{DM-bool}, see also \cite{DM-ad} and Subsection \ref{ssec-benrich}. 

Even though the rational field $\Q$ is undecidable, it turns out that there is an $\cL^{fin,res}_{Boolean}(\cL_{rings})$-definable subset of $\A_{\Q}$ that contains $\Q^*$. 

\begin{prop} [Derakhshan-Macintyre {\cite{DM-ad}}]\label{hilbert1} Any $a\in \Q^*$ is a non-square in $\Q_p$ only for an even number of $p$'s.\end{prop}
This follows from basic properties of the Hilbert symbol.  
Let $p$ be a prime or $p=\infty$, where $\Q_{\infty}=\R$. 
The Hilbert symbol $(a,b)_p$, for $a,b \in \Q_p$ 
is defined as follows.
$$
(a,b)_p=\begin{cases}
 1 ~ ~ ~ \text{if $ax^2+by^2-z^2$ has a non-zero solution $(x,y,z)\neq (0,0,0)$}\\
 -1 ~ ~ \text{otherwise.}
\end{cases}
$$ 
If $p\neq 2$ and $|a|_p=|b|_p=1$ then $(a,b)_p=1$. If $a,b\in \Q$, then, $(a,b)_p=1$ for large $p$. 
The product formula for the Hilbert symbol states that for $a,b \in \Q^*$,
$$\prod_{p \in \{\mathrm{Primes}\}\cup \{\infty\}} (a,b)_p=1.$$
(see \cite{cassels-local}, pp. 46).

It further follows that 
$\prod_{p \in \{\mathrm{Primes}\}\cup \{\infty\}} (a,b)_p=1$ if and only if the number of $(a,b) \in (\Q^*)^2$ such that 
$(a,b)_p \neq 1$ is even. 

Now we can define the set of all adeles $a\in \A_{\Q}$ such that $a(p)$ is a square at an even number of $p$ by the formula
$$\B_{\Q}^+\models Res(2,0)([[\neg P_2(x)]]),$$
where $\B_{\Q}^+$ is the expansion of $\B_{\Q}$ to the language $\cL^{fin,res}_{Boolean}$ and 
$P_2(x)$ is the formula that $x$ is a square. 

Similarly, let $\theta(a,b)$ be the formula $$\exists x \exists y \exists z (ax^2+by^2-z=0).$$
Let 
$$\mathcal{K}=\{(a,b)\in \A_{\Q}: Res(2,0)([[\theta(a,b)]])\},$$
We call this set the adelic kernel of the Hilbert symbol. 

By the product formula for the Hilbert symbol
$$(\Q^*)^2 \subseteq \mathcal{K}$$
where the inclusion of $(\Q^*)^2 \subseteq (\A_{\Q})^2$ is induced from the diagonal inclusion of each factor. 

We have shown the following.
\begin{prop} [Derakhshan-Macintyre {\cite{DM-ad}}] \label{hilb}\noindent
\begin{enumerate}
\item The set of all adeles $a\in \A_{\Q}$ such that $a(p)$ is a non-square in $\Q_p$ for an even number of $p$ is an $\cL^{fin,res}_{Boolean}(\cL_{rings})$-definable subset of $\A_{\Q}$ containing $\Q^*$.
\item The adelic kernel of the Hilbert symbol 
$\mathcal{K}$ is an $\cL^{fin,res}_{Boolean}(\cL_{rings})$-definable subset of $\A_K^2$ containing $(\Q^*)^2$.\end{enumerate}\end{prop}

\begin{prob} Extend Proposition \ref{hilb} to general number fields.\end{prob}
This should be compared to the undefinability of $K$ in $\A_K$ (true by the undecidability of $K$ proved by Julia Robinson).

\section{\bf Imaginaries in the adeles and the quotient of the space of adele classes by maximal compact subgroup of idele class group}\label{ssec-adcl} 
Paul Cohen (unpublished notes) and Alain Connes (see for example \cite{connes-selecta}) defined the space of adele classes 
$\A_K/K^*$. It is the quotient of $\A_K$ by the action of $K^*$ by multiplication by right. Their motivation was the Riemann hypothesis, 

Connes and Consani (cf. \cite{connes-c-site}, \cite{CC2},\cite{CC}) studied the space $\hat{\Bbb Z}^*\setminus\A_{\Q} /\Q^{*}$ which is the quotient of $\A_{\Q}/{\Q}^*$ by the action of the maximal compact subgroup $\hat{\Bbb Z}^*$ of the idele class group $\Bbb I_{\Q}/\Q^*$ acting on $\A_{\Q}/{\Q}^*$ from the left. $\hat{\Bbb Z}^*\setminus\A_{\Q} /\Q^{*}$ is related to the arithmetic site topos in \cite{connes-c-site}.

Boris Zilber asked whether $\hat{\Bbb Z}^*\setminus\A_{\Q} /\Q^{*}$ is interpretable 
in $\A_{\Q}^{fin}$ (i.e. its elements are equivalence classes of a definable equivalence relation). This was proved in \cite{DM}.

\begin{thm} [Derakhshan-Macintyre {\cite{DM-ad}}]The set $\hat{\Bbb Z}^*\setminus \A_{\Q} / \Q^*$ is interpretable in $\A_{\Q}$.\end{thm}

This raises the question of describing the imaginaries in $\A_K$. 

\begin{prob}\cite{DM-ad} Describe the imaginaries in the theory of $\A_K^{fin}$.\end{prob}

Hrushovski-Martin-Rideau \cite{HMR} have proved that $\Q_p$ admits uniform elimination of imaginaries for all $p$ relative to the "geometric sorts". These sorts are the spaces of lattices $GL_n(\Q_p)/G_n(\Z_p)$, for all $n\geq 1$. 
\begin{prob} \cite{DM-ad} Prove an analogue of the Hrushovski-Martin-Rideau theorem for the finite adeles $\A_K^{fin}$.\end{prob}

\begin{prob} \cite{DM-ad} Study the model theory of the space of adele classes $\A_K/K^*$.\end{prob}
$\A_K/K^*$ is 
a hyperring in the sense of Krasner, i.e. a structure with multiplication and a multi-valued addition. See Connes \cite{CC},\cite{CC2} and their references for hyperrings.

In \cite{DM-supp} Macintyre and myself prove quantifier elimination and related model-theoretic results for the Krasner hyperrings $K_v/1+\mathcal{M}_v^n$ for given $n$, and their restricted products $\prod_{v\in V_K^{fin}}^{'} K_v/1+\cM_v^n$ in a language suitable for hyperrings. 
\begin{prob} Prove results analogous to those in \cite{DM-supp} on the restricted product of the Krasner hyperrings for the case of the adele class hyperring $\A_K/K^*$.\end{prob}

\section{\bf Artin reciprocity}\label{ssec-artin} 

We shall only state Artin reciprocity for a global field. There is also a version for local fields which does deserve model-theoretic analysis, but we will not deal with that here.

Let $K\subseteq F$ be an extension of number fields. For any valuation $v\in V_K$, consider $K_v$, and for a valuation $u$ of $F$ lying over $v$ consider $F_u$. We have the norm map $N_{F_u/K_v}:F_u\rightarrow K_v$. This defines the norm map 
on ideles $N_{F/K}: \I_F\rightarrow \I_K$ where $N_{F/K}(x(u))$ is the idele in $\I_K$ whose $v$th component is 
$\prod_{u|v}N_{K_u/F_v}(x(u))$. 

Let $C_K=\I_K/K^{\times}$ denote the idele class group of $K$.
\begin{thm}[Artin Reciprocity] Let $K$ be a global field.
\begin{itemize}
\item There exists a homomorphism called the Artin map $\theta_K: C_K \rightarrow Gal(\bar{K}/K)^{ab}$ such that 
for every finite abelian extension $F/K$, the composition $\theta_{K/F}$ of $\theta_K$ with the projection 
$Gal(\bar{K}/K)^{ab}\rightarrow Gal(F/K)$ is surjective with kernel equal to $N_{F/K}(C_F)$. Conversely, 
any open subgroup $N$ of $C_K$ of finite index has the form $Ker(\theta_{F/K})$ for some finite abelian extension 
$F/K$, and $C_K/N\cong Gal(F/K)$.
\item Let $F/K$ be a finite abelian extension. Let $\frak p$ be a prime in $K$ that is unramified in $F$. Let $x_{\frak p}$ denote the idele $(1,\dots,1,\pi_v,1,\dots,1)$, where $\pi_v$ is a uniformizing element of $K_v^{\times}$ and $v$ corresponds to 
$\frak p$. 
Then the map $\theta_{F/K}$ is induced (modulo $K^{\times}$) from a surjective group homomorphism $\theta_{F/K}:\I_K\rightarrow Gal(F/K)$ which sends $x_{\frak p}$ to $Frob_{\frak p}$. 
\end{itemize}
\end{thm}
\begin{proof} See \cite{ramak}.
\end{proof}
\begin{prob} \noindent \begin{enumerate}
\item Give a model-theoretic analysis and interpretation of Artin reciprocity formulated for a restricted direct product with respect to a suitable formula (see Section \ref{sec-rest}) 
of suitable structures which are definable or interpretable in the non-Archimedean completions 
$K_v$ in a language with predicates for Artin symbols in the residue fields of $K_v$ for $v$ corresponding to an unramified prime in $K$. 
\item Use the methods of Galois stratification (see \cite{FJ}) and study functoriality and uniformity in the number field $K$. 
\item What generalizations of Artin reciprocity can be obtained this way?
\end{enumerate}
\end{prob}
 
The term functoriality in the problem refers to the functoriality in $K$ in Artin reciprocity (see \cite{ramak}). 

The following question was asked by Nicolas Templier after a talk I gave in Princeton.
\begin{prob}[Templier]\label{prob-temp} Let $F/K$ be a finite extension of number fields. Is 
the image of the norm map $N_{F/K}(\I_F)$ definable in $\I_K$ or $\A_K$ in some language? How does this definition depend on the number field $K$?\end{prob}

It seems plausible that $N_{F/K}(\I_F)$ is definable in $\A_K$ in the language of rings (using results of 
Section \ref{sec-rest} on $\cL_{rings}$-definability of Boolean values and $Fin$) because of the following.

\begin{ex} If $F$ is the extension
$\Q_p(\sqrt{2})$, then for any $p>2$, one has 
$$N_{F/\Q_p}(F^*)=\{z\in \Q_p^*: \exists x \exists y \in \Q_p ~(z=x^2-2y^2)\}.$$\end{ex}

In Problem \ref{prob-temp} one has to investigate what language to have for the $K_v$ or $K_v^*$ so that the induced restricted product language for $\A_K$ or $\I_K$ is 
suitable for the required definability.

The languages of Macintyre and Belair are well-suited to study the idelic norm groups since if $L$ is an extension of $\Q_p$ of degree $n$, then $N_{L/\Qp}(L^*)$ is an open subgroup of $\Q_p^*$ containing the group of nonzero $n$th powers 
$(\Q_p^*)^n$ and in $\cL_{Belair}$ there are constants for coset representatives for the groups $P_n$ of non-zero $n$th powers in $\Q_p^*$ for all $n$ and all $p$ (note: there is a bound independently of $p$ on the index of $P_n$). 

\begin{prob}\label{prob-norm} Study definability properties of the images of norm maps $N_{F/K}(\I_F)$ for all number fields $F$ and $K$, 
uniformly in the number field, in the language for restricted products (cf. \ref{sec-rest}) induced from the languages of Macintyre and Belair for the factors (cf. \ref{ssec-venrich}). 
\end{prob}
Templier suggested that Problem \ref{prob-temp} and Sarnak suggested that definability in adele rings may be useful in some problems on families of $L$-functions in work of Sarnak-Shin-Templier \cite{peter-nicolas}. 

Generalizations of Artin Reciprocity for non-abelian extensions is one of the aspects of the Langlands Program where the approach is via representations of adelic groups $G(\A_K)$, where $G$ is a suitable algebraic group. For more on this see Section 
\ref{ssec-lang}.

It would be interesting to have a model-theoretic approach to non-abelian extensions and Artin reciprocity. 

\begin{prob}\label{prob-artin}\noindent\begin{enumerate}
\item Let $\cL$ be any of the languages in Subsection \ref{ssec-venrich}. For a finite extension $F/K$ of number fields, is there an $\cL_{Boolean}^{fin,res}(L)$-definable subset $\S$ of $\A_K^n$, for some $n$, or an $\cL_{Boolean}^{fin,res}(L)$-definable subset $\S$ of a restricted product $\prod_{v\in V_K}^{\varphi} \cM_v$ for some $\cL$-structures $\cM_v$ and $\cL$-formula $\varphi(x)$, and a definable map $\hat{\sigma}_{F/K}$ from $\S$ to $Gal(F/K)$ generalizing the Artin map?
\item Let $Gal(\bar{K}/K)^m$ be the maximal $m$-step 
solvable quotient of $K$. Is there a homomorphism from $\S$ to $Gal(\bar{K}/K)^m$ that gives 
$\sigma_{F/K}$ by composing with the natural projection map? 
\item How much this would be true beyond the solvable case?
\end{enumerate}
\end{prob}

Note that $S$ is allowed to be definable by means of predicates related to Hilbert symbols as we have allowed $\cL_{Boolean}^{fin,res}$-definability. 

\medskip

\subsection{\bf Remarks on the idele class group of $\Q$}\label{ssec-idcl}

\

\medskip

One can relate the idele class group to definability in adeles.
\begin{thm} The idele class group $C_{\Q}$ is a definable subgroup of the adeles $\A_{\Q}$.\end{thm}
\begin{proof} By Proposition 6-12 page 23 in \cite{ramak},
$$\I_{\Q}=\Q^{\times} \times \R_{+}^{\times} \times \prod_{p} \Z_p^{\times}.$$
(This is used in the adelic proof of the Kronecker-Weber theorem on the maximal abelian extension of $\Q$).
Thus $\I_{\Q}/\Q^{\times} \cong \R_{+}^{\times} \times \prod_{p} \Z_p^{\times}$. Clearly this is $\cL_{rings}$-definable in $\A_{\Q}$ using the $\cL_{rings}$-definability of Boolean values (cf. Section 2).\end{proof}

\begin{prob} Does a similar definability result hold for a general number field $K$? If so how does the definition depend on the number field?\end{prob}

\medskip 

\section{\bf Euler products of $p$-adic integrals}\label{ssec-euler}

\subsection{\bf Analytic properties of the Euler products}

\

\medskip

Let $K$ be a finite extension of $\Q_p$ with residue field of cardinality $q$. Let $dx$ be a normalized Haar measure on $K$ giving the valuation ring $\cO_K$ volume $1$. 
Let $\varphi(x_1,\dots,x_n)$ be a formula of the language of rings. Let $f:K^n\rightarrow K$ be an $\cL_{rings}$-definable function. 
Let 
$$X=\varphi(K)=\{(a_1,\dots,a_n)\in K^n: K\models \varphi(a_1,\dots,a_n)\}$$ be the set defined by $\varphi$ in $K$. 

In \cite{Denefrationality}, Denef initiated the study of $p$-adic integrals of the form 
$$Z(s,p)=\int_{X} |f(x_1,\dots,x_n)|^s dx$$ and proved they are rational functions of $q^{-s}$. This generalizes work of Igusa for the case when $f \in \Z[x_1,\dots,x_n]$ and $X$ is $\Z_p^n$ or a Zariski closed subset of it.

Denef's result gave a solution to a conjecture of Serre on rationality of $p$-adic Poincare series of the form $\sum_{n\geq 1} c_kT^k$ where $c_k$ is the number of roots of $f$ modulo $p^k$ that lift to a root in $\Z_p$. See \cite{Denefrationality}. 
We refer to $Z(s,p)$ as a definable $p$-adic integral.

 Pas \cite{pas} and Macintyre \cite{Macintyre} independently proved that there are uniformities in the shape of these rational functions as $p$ varies if $\phi$ and $f$ are over $\Q$. 
The subject of motivic integration extends this uniformity and gives it a geometric meaning. It has been developed by Denef-Loeser \cite{DL}, Cluckers-Loeser \cite{CL2}, and Hrushovski-Kazhdan \cite{HK}, and has had several applications to algebraic geometry, number theory and algebra.

In \cite{zeta1}, inspired by results and problems in group theory (on zeta functions counting subgroups of a group) and number theory (on height zeta functions counting rational points of algebraic varieties), I 
considered Euler products over all primes $p$ of such definable $p$-adic integrals. These Euler products are of a global nature and relate to arithmetical questions on number fields, while the $p$-adic integrals are of a local nature. But the uniformities that are true over all $p$ of the shape of the rational functions can be used together with some results on algebraic geometry and model theory of finite and $p$-adic fields, together with combinatorial arguments, to prove the following result. 

\begin{thm}[Derakhshan {\cite{zeta1}}]\label{Thm-zeta} Let $Z(s,p)$ be as above. 
Let $a_{p,0}$ be the constant coefficient of $Z(s,p)$ when expanded as a power series in $q^{-s}$. Then the Euler product over all primes $p$ 
$$\prod_p a_{p,0}^{-1} Z(s,p)$$
has rational abscissa of convergence $\alpha$ and meromorphic continuation to the half-plane $\{s: Re(s)>\alpha -\delta\}$ for some $\delta>0$.
The continued function is holomorphic on the line $Re(s)=\alpha$ except for a pole at $s=\alpha$.\end{thm}
Tauberian theorems of analytic number theory then yield.
\begin{cor}[Derakhshan {\cite{zeta1}}] 
Suppose that the Euler product $\prod_p a_{p,0}^{-1} Z(s,p)$ can be written as the Dirichlet series $\sum_{n\geq 1} a_n n^{-s}$, then for some real numbers $c,c'\in \R$, 
$$a_1+a_2+\dots+a_N \sim c N^{\alpha}(log N)^{w-1}$$ 
$$a_1+a_22^{-\alpha}+\dots + a_N N^{-\alpha} \sim c'(log N)^w$$
as $N \rightarrow \infty$, where $w$ is the order of the pole of $Z(s,p)$ at $\alpha$.\end{cor}

\begin{prob}\label{prob-zeta} Formulate an adelic version of Theorem \ref{Thm-zeta}. For this, add "Archimedean factors" to the Euler products, and write the Euler product as an adelic integral of a suitable function over a definable subset of 
$\A_K^m$ for some $m\geq 1$. Once this is done, would the "completed 
Euler product" have meromorphic continuation to the whole complex plane? Would it satisfy a functional equation?
\end{prob}

Problem \ref{prob-zeta} relates to Definitions \ref{con-fin-comp}, \ref{def-sp-ad-con}, \ref{def-Whitt}, Remark \ref{rem-real}, Example \ref{ex-gamma}, and Problems \ref{zeta-arch} and \ref{forms}.  
(The conjectural connections to O-minimal structures and Hodge theory is challenging).

\medskip

\subsection{\bf Conjugacy class zeta functions in algebraic groups}\label{ssec-conj}

\

\medskip

Theorem \ref{Thm-zeta} applies to zeta functions counting conjugacy classes in Chevalley groups over a number field. An example of such a result is the following result from \cite{zeta1}. Let $c_m$ denote the number of conjugacy classes in 
$SL_n(\Z/m\Z)$. Then the global conjugacy class 
zeta function $\sum_{m\geq 1} c_m m^{-s}$ has rational abscissa of convergence $\alpha$ and 
meromorphic continuation the half-plane $\{s: Re(s)>\alpha -\delta\}$ for some $\delta>0$. It follows that
$$c_1+\dots+c_N \sim c N^{\alpha} (log N)^{w-1}$$
for some $c\in \R_{>0}$ as $N\rightarrow \infty$ ($w$ as in Theorem \ref{Thm-zeta}). 

To be able to apply \ref{Thm-zeta} one must show that the above zeta function is an Euler product of $p$-adic integrals of Denef-type over definable sets, uniformly in $p$. This is done in joint work with Mark Berman, Uri Onn, and Pirita Paajanen in 
\cite{BDOP}. There it is proved that the local factors of the Euler product, which are of the form $\sum_{m\geq 0} c_m q^{-ms}$ where 
$c_m$ denotes the number of conjugacy classes in the congruence quotient $SL_n(\Z_p/p^m \Z_p)$, are definable $p$-adic integrals and depend only on the residue field for large $p$. See \cite{BDOP} and the survey 
\cite{zeta-surv} for details.

\begin{prob} Formulate the results on global conjugacy class zeta functions adelically.\end{prob}

\subsection{\bf Adelic height zeta functions and rational points}

\

\medskip

The adelic height zeta function is a very useful guiding example for an approach to Problem \ref{prob-zeta}, particularly on its Archimedean factors and meromorphic properties. 

On a general variety $X$ over $\Q$, one can associate a height function $H_L$ to every line bundle $L$ on $X$ via 
Weil's height machine. On $\P^n$, the height $H=H_{\cO_{\P^n(1)}}$ associated to the line bundle 
$\cO_{\P^n(1)}$ of a hyperplane is defined by 
$$H(x)=\sqrt{x_0^2+\dots+x_n^2},$$
where $(x_0,\dots,x_n)$ is a primitive integral vector representing $x\in \P^n(\Q)$. Schanuel proved that as $T\rightarrow \infty$,
$$card(\{x\in \P^n(\Q): H_{\cO_{\P^n(\Q)}}(x)<T\}) \sim cT^{n-1}$$
for an explicit $c\in \R_{>o}$. See \cite{manin-book}.

For a general variety $X$ over $\Q$ and an ample line bundle $L$ on $X$, there is a height function on $X(\Q)$ defined
via an embedding of $X$ into $\P^n$ and pulling back the height function on $\P^n(\Q)$. For a subset $U\subset X$ let
$$N_U(T)=card(\{x\in X(\Q) \cap U: H_L(x)<T\}).$$
Manin has given a conjecture on the asymptotic of the numbers $N_U(T)$ as $T\rightarrow \infty$. See \cite{manin-book},\cite{GO}. In \cite{GO}, Gorodnik and Oh prove new cases of Manin's conjecture for orbits of group actions and for compactifications of affine homogeneous varieties using an ergodic theoretic approach of Duke-Rudnick-Sarnak. 

They consider an algebraic group $G$ over a number field $K$ with a representation $\rho:G \rightarrow GL_{n+1}$. Then $G$ acts on $\P^n$ via the canonical map $GL_{n+1} \rightarrow PGL_{n+1}$. 
Let $U=u_0G$, where $u_0\in \P^n(\Q)$. Let $X$ be the Zariski closure of $U$, and $H$ the height 
function on $X(\Q)$ obtained by the pull back of $H_{\cO_{\P^n(1)}}$. For simplicity we only consider the case $K=\Q$ of their results. 

Manin conjecture type estimate state that
$$N_T:=card(\{x\in U(\Q) H(x)<T\}) \sim cT^a.(\mathrm{log} T)^{b-1}$$
for $c\in \R_{>0}$ and $a,b\in \Z, a>0, b\geq 1$.

In \cite{GO}, Gorodnik-Oh prove this under conditions on $G$ and the stabilizer of $u_0$, and other conditions. 
Their beautiful approach is 
to consider $U(\Q)$ as a discrete subset of the adelic space $U(\A_{\Q})$ (defined as a restricted product as in Section 10). 

The height function $H(x)$ can be extended to a hight function on $U(\A_{\Q})$, denoted by $H_{\A_{\Q}}(x)$, so that
$$B_T:=\{x\in U(\A_{\Q}): H_{\A_{\Q}}(x)<T\}$$
is compact. We have that
$$\{x\in U(\Q): H(x)<T\}=U(\Q)\cap B_T.$$
Then under certain conditions the asymptotic of the numbers $N_T$ follows from an asymptotic for the volumes of $B_T$, and for this they can 
make use of the ergodic-theoretic work of Duke-Rudnick-Sarnak (cf. \cite{GO}).

One can ask the following question, whose positive solution would extend some results 
in \cite{GO}.

\begin{prob}\label{height} Prove an analogue of Theorem \ref{Thm-zeta} on meromorphic continuation beyond abscissa of convergence for the integrals
$$\int_{U(\A_K)} H_{\A_{\Q}}(x)^{-s} \tau,$$
where $\tau$ is a suitable measure on $U(\A_{\Q})$.
\end{prob}
\begin{note} Note that $U(\A_{\Q})$ can be partitioned into finitely many pieces each of which is a definable subset of Type I of 
$\A_{\Q}^m$ for some $m$ (using the usual covering of projective space by affine pieces). The adelic integral factors as an Euler product by standard properties of the 
adelic height. \cite{zeta2} contains work in progress on Problem \ref{height}. 
\end{note}

\section{\bf Finite fields with additive characters and continuous logic}\label{fin-char}

Ax's results in \cite{ax} on decidability of the theory of $\F_p$, for all (and all but finitely many) $p$, and $\F_q$, for all (and all but finitely many) prime powers (both for single $p$ and all but finitely many $p$) were proved as a result of the 
decidability of the theory of pseudo-finite fields. These are defined as perfect pseudo-algebraically closed fields that have exactly one extension of each degree inside their algebraic closure. 

It is easy to see that infinite models of the theory of finite fields are 
pseudo-finite. Work of Ax \cite{ax} implies the converse statement. Kiefe \cite{Kiefe} gave a quantifier elimination for the theory of pseudo-finite fields in an expansion of the ring language by the solvability predicates stated in Section \ref{ssec-venrich}.

In \cite{CDM}, Chatzidakis, van den Dries, and Macintyre revisited the model theory of finite fields, and proved, among other results, generalizations of the 
Lang-Weil estimates for the number of $\F_q$-points of an absolutely irreducible variety defined over $\Q$ to definable sets. They also introduced a pseudo-finite measure. 

In \cite{udi-char}, Hrushovski added additive characters to the language, and studied the continuous logic theory of 
pseudo-finite fields with an additive character. Firstly, he proves that the usual first-order theory is undecidable.

In continuous logic, a structure $\cM$ is a set together with a function $R^{\cM}: \cM^n \rightarrow V_R$
for each $n$-ary relation $R$, where $V_R\subseteq \C$ is a compact set called the set of values of $\phi$. This gives an interpretation $\phi^{\cM}$ of a formula $\phi$ within its set $V_{\phi}$ of values as a function $\phi^{\cM}:\cM^n\rightarrow \C$ with compact image. See \cite{udi-char} for decidability, quantifier elimination, and related notions in continuous logic.

If the image of $\phi^{\cM}$ is the set $\{0,1\}$, then the pullback of $1$ is called a discretely definable set. For example the graphs of addition and multiplication are discretely definable. 

Let $\Psi_p(n+p\Z)=exp^{2\pi i n/p}$ be an additive character on the field with $p$ elements $\F_p$, and $\Psi_q(x)=\Psi_p(tr_{\F_q/\F_p}(x))$, where $tr_{\F_q/\F_p}(x)$ is the trace map from $\F_q$ to $\F_p$, 
an additive character on the finite field $\F_q$. 
Add a unary function symbol to $\cL_{rings}$ to be interpreted as the additive character in the standard models $\F_q$, and let $\F_q^+=(\F_q,+,.,\Psi_q)$ be the finite field with $\Psi_q$ in continuous logic. We note that any other additive character on $\F_q$ has the form $\Psi_q(ax)$ for a unique $a\in \F_q^*$, thus the additive characters are all uniformly definable.

Let $T=Th(\{\F_q^+, q~\text{prime ~power}\})$, the theory  of all finite fields with additive character. In \cite{udi-char} Hrushovski proves that $T$ is decidable, admits quantifier elimination to the level of algebraically bounded quantifiers, and is simple. He also proves that the pseudo-finite measure, 
introduced by Chatzidakis-van den Dries-Macintyre \cite{CDM} is definable in this setting and its Fourier transform is also definable, and the discretely definable sets are exactly the sets definable in Ax's theory. 
He also proves that the asymptotic first-order theory of $\F_q^+$ with an additive character $\Psi_q$ where the characteristic is unbounded, is undecidable. 

These results generalize the results of Ax \cite{ax} and 
Chatzidakis-van den Dries-Macintyre on pseudo-finite fields to $T$. We remark that \cite{udi-char} contains applications to exponential sums over definable sets in finite fields.

\section{\bf $p$-adic fields with additive characters and continuous logic}\label{ad-char}

Given a $p$-adic number $x=\sum_{j\geq -N} c_j p^j$, where $-N=v_p(x)$, 
the $p$-adic fractional part of $x$ is defined by $$\{x\}=\sum_{-N\leq j\leq -1} c_jp^j.$$
The map 
$\psi_p(x)=e^{2\pi i \{x\}}$ is an additive character on $\Q_p$ that is trivial on $\Z_p$. Given a finite extension 
$K$ of $\Q_p$, the map $\psi_p(tr_{K/\Q_p}(x))$ is an additive character on $K$ that is trivial on the ring of integers 
$\cO_K$. 

Enrich the language of rings by a 1-place predicate to be interpreted as the character. Again the additive characters on $K$ are uniformly definable.

Hrushovski \cite{udi-char} proves that the first-order theory of $(\Q_p,\psi_p)$ and the asymptotic first-order theory of $(\Q_p,\psi_p, p~\text{prime})$ are undecidable. 

In analogy to \cite{udi-char}, one can ask.

\begin{prob}[Hrushovski {\cite{udi-char}}] \label{udi-prob} Is the continuous logic theory of $(\Q_p,\psi_p)$ decidable?
Is integration with respect to $p$-adic measure definable, both for single $(\Q_p,\psi_p)$ and asymptotically?
\end{prob}

Problem \ref{udi-prob} is related to defining suitable Fourier transform operators on definable sets in $p$-adic fields and adeles which is related to the problems in Section \ref{ssec-lang}, especially Problems \ref{closed} and \ref{prob-poisson}.

In \cite{udi-char}, Hrushovski gives axioms for the theory of pseudo-finite fields with an additive character. In conversations with him, the first author learned of the following question.
\begin{prob}[Hrushovski] What are axioms for the continuous logic theory of $(\Q_p,\psi_p)$ both for single $p$ and 
asymptotically?
\end{prob}

Set $\psi_{\infty}(x)=e^{-2\pi ix}$ for $x\in \R$. Then the map 
$$\psi(x)=\prod_{p\leq \infty} \psi_v(x(v)),$$
where $x\in \A_{\Q}$, is an additive character on $\A_{\Q}$ that is trivial on $\Q$. Let $tr$ denote the trace map from 
$\A_K$ of a number field $K$ to $\A_{\Q}$, then $\psi_K(x)=\psi(tr(x))$ is an additive character on $\A_K$.

\begin{thm}\label{ad-char} The first-order theory of $(\A_{\Q},\psi)$, where $\psi(x)$ is an additive character, is undecidable.\end{thm}
\begin{proof} Fix a prime $p$. Let $e_p$ denote the supremum of all the minimal idempotents $e$ such that $e\A_{\Q}$ has residue field 
equal to $\F_p$. Then $e\A_{\Q}$ is isomorphic to $\Q_p$ and the $p$-adic additive character on $\Q_p$ is the $p$th component of $\psi(x)$. This 
gives an additive character $\psi_e$ on $e\A_{\Q}$ such that the theory of $(e\A_{\Q},\psi_e)$ is undecidable by Hrushovski's theorem \cite{udi-char} 
on the undecidability of $(\Q_p,\psi_p(x))$. Since $(e\A_{\Q},\psi_e)$ is definable in $(\A_{\Q},\psi)$, we are done.
\end{proof}

\begin{prob}\noindent\begin{enumerate}
\item What can one say about the continuous logic theory of $\A_K$ with an additive character $\psi$?
\item What are axioms for this theory? Is it decidable? 
\item What can one say about definability of integration and Fourier transform on $\A_K$ in continuous logic?
\end{enumerate}
\end{prob}

\section{\bf $p$-adic and adelic multiplicative characters and $L$-functions }\label{ssec-char}

An $L$-function is generally defined as a Dirichlet series $\sum_{n\geq 1} a_n n^{-s}$ that can be written as an Euler product $\prod_{p} P_p(s)$ over primes $p$, where $P_p(t)$ is a rational function. They are initially defined in a right half-plane, but admit meromorphic continuation beyond their half plane of convergence. Examples are the Riemann zeta function $\zeta(s)$, where $a_n=1$ for all $n$, the Dedekind zeta function $\zeta_K(s)$ of a number field $K$, and $L$-functions associated to algebraic varieties 
and Galois representations. See Serre's book \cite{serre-book}.

Suppose $G=\prod_{i\in I} G_i$ is a restricted product of locally compact abelian groups with respect to open subgroups $H_i$. If $\chi$ is a character on $G$, then it follows that 
for any $x$ in $G$, $\chi(x(i))=1$ for all but finitely many $i\in I$, and $\chi(x)=\prod_i \chi(x(i))$ (see \cite{ramak}).

Let $K$ be a number field. A multiplicative character of $K_v^*$ is called unramified if $\chi\vert_{\cO_v^*}=1$. A multiplicative character $\chi$ on the idele class group $\I_K/K^*$ is called an idele class character. By the above, $\chi=\prod_v \chi_v$, where $\chi_v$ is a character of $K_v^*$ that is unramified for all but finitely many $v$. It can be shown that if $\chi$ is character of $K_v^*$, then 
$\chi=\mu ||.||^s$, where $\mu$ is a character of $\cO_v^{*}$ and $Re(s)$ is uniquely determined. See \cite{ramak}. 

Given an idele class character $\chi$, write it as $\mu |.|^s$, where $\mu$ is unitary. For each $v$ we get character of 
$K_v^*$ defined by 
$$\chi_v(t)=\chi(1,\dots,1,t,1,\dots,1),$$
where $t$ is in the $v$th component, hence 
$\chi(x)=\prod_v \chi_v(x)$, a product that makes sense as the restriction of $\chi_v$ to the the units $\cO_v^{*}$ is trivial for all but finitely many $v$. 

The Hecke $L$-function attached to $\chi$ is the Euler product 
$$L(s,\chi)=\prod_{v\in V_K} L(s,\chi_v),$$
where for Archimedean $v$, $L(s,\chi_v)$ is defined using $\Gamma$-functions (cf. \cite{ramak}). If $v$ is non-Archimedean, corresponding to a prime $\frak p$, and $\chi_v$ is trivial on the units $\cO_v^*$ (which holds for all but finitely many $v$), then one defines  
$$L(s,\chi_v)=\frac{1}{1-\chi_v(\pi_v)},$$
where $\pi_v$ is a uniformizing element of $K_v$. The function $L(s,\chi)$ can be analytically continued to all of $\C$ and has a functional equation, as proved by Tate \cite{tate-thesis}. These $L$-functions generalize zeta functions of number fields and $L$-functions of Dirichlet characters \cite{ramak}.

On the other hand, an Artin $L$-function is an $L$-function that is associated to a finite-dimensional representation $\rho$ of the Galois group 
$Gal(F/K)$, where $F/K$ is a finite extension. It is defined by
$$L(s,\rho)=\prod_{v\in V_K} L(s,\rho_v)$$
where $\rho_v$ is the restriction of $\rho$ to the decomposition group (cf. \cite{ramak}).

For the $v$ that are associated to a prime $\frak p$ that 
is unramified in $F$ (which is for all but finitely many $v$), one has the Frobenius conjugacy class $Frob_{\frak p}$ in $Gal(F/K)$, and
$$L(s,\rho_v)=\frac{1}{det(I-\rho(Frob_{\frak p}N\frak p^{-s}))}=\prod_{1\leq i\leq d} \frac{1}{1-\beta_i(\frak p)N\frak p^{-s},}$$
where $\beta_1(\frak p),\dots,\beta_d(\frak p)$ are the eigenvalues of $\rho(Frob_{\frak p})$.

$L(s,\rho)$ has meromorphic continuation to all of $\C$. Artin conjectured that it is an entire function if $\rho$ is irreducible and non-trivial, and proved this for one-dimensional $\rho$ via proving the following 

\

{\it Correspondence between Artin $L$-functions and Hecke $L$-functions}: An $L(s,\rho)$ attached to a one-dimensional $\rho$ has the form $L(s,\chi)$ for a character $\chi=\chi(\rho)$ of $\I_K$ vanishing on $K^*$. This follows from Artin's reciprocity law. See \cite{ramak} and Section \ref{ssec-artin}. Other cases of Artin's conjecture have been proved by Langlands and others (see \cite{lang-prob},\cite{lang-icm}).

\begin{prob} Is there a language (extending an appropriate language for restricted products) where one can express or give a model-theoretic interpretation of the {\it correspondence between Artin $L$-functions and Hecke $L$-functions}? \end{prob}

These questions would have implications for a model-theoretic understanding of {\it automorphic representations} and {\it automorphic forms on adele groups}, which would be related to a model theory for {\it classical modular forms}. See Section \ref{ssec-lang}

\begin{prob}\label{prob-idele-char} Let $\cL$ be a language for the adeles $\A_K$ augmented by a unary predicate for a multiplicative character $\chi$ defined on the ideles and trivial on $K^*$. 
What can be said about the $\cL$-theory of $(\A_K,\chi)$?\end{prob}

\section{\bf Tate's thesis and zeta integrals for $GL_1$}\label{ssec-tate}
Let $K$ be a local field. A $\C$-valued function $f$ on $K$ is called smooth if it is $C^{\infty}$ when $K$ is 
Archimedean, and locally constant when $K$ is non-Archimedean. A smooth function on $K$ is a Schwartz-Bruhat function if it goes to zero rapidly at infinity if $K$ is Archimedean and if it has compact support if $K$ is non-Archimedean.  $S(F)$ denotes the class of Schwartz-Bruhat functions on $K$. See \cite{ramak}.

Tate's thesis \cite{tate-thesis} was a beautiful and fundamental work that influenced a wide range of topics in modern number theory and arithmetic geometry. The main goal was a generalization, via different proofs, of 
Hecke's results on meromorphic continuation and functional equation for Hecke $L$-functions for number fields. However the theory has had far reaching influence and applications. 

Concerning the Hecke $L$-functions, it gave more information on the 
functional equation and so-called epsilon factors, gamma factors, and root numbers, and various quantities acquire an interpretation in terms of volumes of subsets in the adeles or ideles. 
For example the class number formula can be given a new volume-theoretic proof, see \cite{ramak}. 

In a similar vein, an important formula of Siegel on quadratic forms was given an interpretation by Weil in terms of volumes of adelic spaces, leading to Weil's conjecture that the Tamagawa volume of $G(\A_K)/G(K)$ is equal to $1$ for a semi-simple simply connected algebraic group $G$ over a number field $K$. 
This has been a great influence in the conjecture of Birch and Swinnerton Dyer. 

In a work, started by Paul Cohen (unpublished) and pursued by Alain Connes and others (see \cite{connes-selecta}), proving an analogue of Tate's thesis for the space of adele classes $\A_K/K^*$ is considered a key step for a proof of the Riemann hypothesis.

Tate's thesis has also played a central role in the development of the Langlands program and in the Langlands conjectures. It has been a starting point for this theory. See Section \ref{ssec-lang} for more on these and a suggested model-theoretic framework and questions.

Tate's results on $L$-functions are deduced from results on local zeta functions attached to characters. Given a Schwartz-Bruhat function $f\in S(K)$ on a local field $K$ with residue field of cardinality $q$, 
and multiplicative character $\chi\in X(K^*)$, the local zeta function is defined by 
$$Z(f,\chi)=\int_{K^*}f(x)\chi(x)d^*x,$$
where $dx^*$ is the measure $(1-q^{-1})^{-1}dx/|x|$, where $dx$ is an additive Haar measure on $K$. These satisfy functional equations relating $Z(f,\chi)$ with $Z(f,\hat{\chi})$, where $\hat{\chi}$ is the dual character, cf. \cite{tate-thesis}.

Let $K$ be a number field. The class of adelic Schwartz-Bruhat functions on $\A_K$ is defined as the restricted tensor product
$$S(\A_K)=\otimes'_v S(K_v)$$
consisting of functions of the form $f=\otimes f_{v\in V_K}$, where $f_v\in S(K_v)$ and $f_v=1$ for all but finitely many $v$. 

Let $\chi$ be a character on $\I_K$ that is trivial on $K^*$ (i.e. an idele class character). Let $f\in S(\A_K)$ be an adelic Schwartz-Bruhat function. The global zeta function of Tate is defined as
$$Z(f,\chi)=\int_{\I_k} f(x)\chi(x) d^*x.$$
Tate proves meromorphic continuation of $Z(f,\chi)$ and a functional equation relating $Z(f,\chi)$ with
$Z(\hat{f},\check{\chi})$, where $\hat{f}$ is the adelic Fourier transform of $f$ and $\check{\chi}$ the dual character. cf \cite{tate-thesis}. 

There is an Euler product factorization 
$$Z(f.\chi)=\prod_{v\in V_K} \int_{K_v^*} f_v(x)\chi_v(x)d^*x$$
where $\chi=\prod_v \chi_v$. As stated in Section \ref{ssec-char}, for all but finitely many $v$, $\chi$ is unramified, so has the form $|.|^s$, and integration of this function is well-understood over $\Q_p^*$ and $\I_K$. 
For the finitely many $v$ where ramification occurs, the integrals are evaluated via the properties of the characters (see \cite{tate-thesis}).

The local factors of these integrals are special cases of motivic integrals of Denef, Cluckers, Loeser and Hrushovski-Kazhdan (see \cite{Denefrationality}, \cite{DL}, \cite{CL1}, \cite{CL2}, \cite{HK}).

This raises the general question what would be a generalization of Tate's thesis where the local zeta integrals are replaced with the Denef-Loeser type motivic integrals. We shall formulate several question in this regard. 
As a first step one can ask.

\begin{prob}Use model theory to generalize Tate's global zeta integrals to integrals of definable functions over definable subsets of $\A_K^m$ for $m\geq 1$, and prove meromorphic continuation results.\end{prob}
We propose a form for these "definable integrals" in the Section \ref{ssec-lang} for the case of $GL_2$ and $G_n$.
We shall propose a generalization of $p$-adic specialization of motivic integrals and study their Euler products. One can then apply Theorem \ref{Thm-zeta} to study these Euler products.

\section{\bf Automorphic representations and zeta integrals for $GL_n$}\label{ssec-lang}

Tate's thesis naturally lead to various questions beyond $GL_1$. These relate to a wide range of problems and topics including class field theory, non-abelian extensions of $\Q$, non-abelian generalization of 
Artin reciprocity, mysteries of zeta and $L$-functions, Langlands' conjectures, problems in Diophantine geometry of integer and rational points on varieties and homogeneous spaces, and arithmetic aspects of algebraic groups. 

\medskip

\subsection{\bf Langlands program and Jacquet-Langlands theory}\label{ssec-jl}

\

\medskip

A fascinating program and set of conjectures were given by Langlands which turned out to be related to several of the above topics at the interactions of algebra, geometry, analysis, and representation theory. Here one works with reductive algebraic groups and the Langlands functoriality conjecture is one of the strongest of the conjectures (see \cite{lang-prob}). 
For an introduction to the Langlands conjectures and program see \cite{lang-prob}, \cite{lang-icm}, \cite{lang-jer}, \cite{bump}. 

It had been known that an approach to the Langlands program is to start by generalizing Tate's thesis to more general groups. Jacquet and Langlands \cite{jac-lang} did this for $GL_2$. Godement and Jacquet \cite{jaq-good} did it for $GL_n$. 
The $GL_2$-theory captures results on modular and Maass forms originated by Hecke and Maass (see \cite{lang-jer}). The general case beyond $GL_n$ concerns a reductive group and its Langlands dual group, which we will not consider in this paper, and leave for a future work.

In this Subsection I shall propose a model-theoretic framework and pose some questions on the 
Jacquet-Langlands theory. This is a first attempt to develop a model-theoretic study in the Langlands program. At the end I will comment on more general situations.

I formulate the basic definitions in the case of $GL_n$. Let $\chi$ be a unitary character of $\I_K/K^*$. Let $L^2(GL_n(K)\backslash GL_n(\A_K),\chi)$ be 
the space of $\Bbb C$-valued measurable functions $f$ on $GL_n(\A_K)$ such that for all $z\in I_K$
$$f((zI)g)=\chi(z)f(g)$$
and 
$$\int_{Z_A GL_n(K)\backslash GL_n(\A_K)} |f(g)|^2dg<\infty,$$
where $dg$ is a Haar measure on $GL_n(\A_K)$, and $Z_A$ the group of scalar matrices with entries in $\I_K$. 

Let $L^2_0(GL_n(K)\backslash GL_n(\A_K),\chi)$ be the subspace $L^2(GL_n(K)\backslash GL_n(\A_K),\chi)$ consisting of functions that satisfy the condition 
$$\int_{(\A_K/K)^{r(n-r)} }f(\begin{pmatrix}
I_r & x\\
0 & I_{n-r}
\end{pmatrix}g) ~ dx=0$$
for almost all $g\in GL_n(\A_K)$ and all $1\leq r<n$, where $x$ is an $r\times(n-1)$ block matrix. These are called {\it cusp forms}. 

Consider the right regular representation 
$$\rho: GL_n(\A_K) \rightarrow End(L^2(GL_n(K)\backslash GL_n(\A_K),\chi))$$ defined by $$(\rho(g)f)(x)=f(xg).$$
Automorphic representations are irreducible representations of $GL_n(\A_K)$ that occur in $L^2(GL_n(K)\backslash GL_n(\A_K),\chi)$. 

The space of cusp forms is invariant under this representation and has the convenient property that it decomposes as an infinite direct sum of irreducible invariant subspaces where each factor appears at most once. This "multiplicity one theorem" was proved by Jacquet-Langlands for $n=2$, and Piatetski-Shapiro and Shalika independently for $n>2$ (see \cite{bump} and the references there).

If $\pi$ is a representation of $GL_n(\A_K)$ that is isomorphic to the representation on one of these invariant subspaces (for some $\chi$), then $\pi$ is called an automorphic cuspidal representation with central character $\chi$.
When $n=1$, an automorphic cuspidal representation is nothing but an idele class character of $\I_K/K^*$. For $n=2$, they arise from modular forms and Maass forms. See \cite{bump}. 

An instance of the functoriality conjecture of Langlands is the following. 

\begin{Conjecture}[Langlands] Let $E/F$ be a finite extension and $\rho:Gal(E/F)\rightarrow GL_n(\C)$ be a representation. 
Then there is an automorphic representation $\pi$ corresponding to $\rho$ such that $L(\rho,s)=L(\pi,s)$.
\end{Conjecture}

This gives a non-abelian extension of Artin reciprocity. See \cite{lang-prob},\cite{bump},\cite{lang-icm}, and impies Artin's conjecture that $L(\rho,s)$ is entire as one knows the poles of automorphic $L$-functions.

Interestingly, this conjecture (and more general functoriality conjectures of Langlands) follow from properties of 
automorphic representations (involving no Galois group at all!). The strongest result in this direction is due to 
L. Lafforgue (see \cite{laff-nott}, \cite{laff-intro}), where the full Langlands functoriality conjecture is proved to be equivalent to a 
non-abelian generalization of the adelic Poisson summation formula of Tate in \cite{tate-thesis}. These involve only adelic zeta integrals. See Subection \ref{ssec-poisson}.

The $L$-functions are defined for $GL_n$ generalizing Tate's method. Their analytic properties are proved via global zeta integrals defined by Jacquet-Langlands for $GL_2$ in \cite{jac-lang} and Godement-Jacquet in \cite{jaq-good}. For simplicity, we restrict ourselves to the case $K=\Q$. 

For $GL_2(\A_{\Q})$ the notion of an automorphic cuspidal representation is a natural generalization of the notion of a cusp form 
$$f(z)=\sum_{i=1}^{\infty}a_ne^{2\pi inz}$$
on $SL_2(\Z)$, and the work of Jacquet-Langlands leads to a new point of view on the $L$-functions 
$$L(s,f)=(2\pi)^{-s}\Gamma(s)(\sum_{i=1}^{\infty} a_nn^{-s})=\int_{0}^{\infty}f(iy)y^{s-1}dy$$
attached to $f$ by Hecke. See \cite{bump},\cite{GS-book}.

Similarly, the $L$-functions attached to a Dirichlet character $L(s,\chi)$ (defined by Dirichlet in the proof of his celebrated theorem on arithmetic progressions) can be written as a Mellin transform 
$$\int_{0}^{\infty} \varphi_{\chi}(t)t^{s/2} dt/t$$
of a so-called $\theta$-series $\varphi_{\chi}$, and can be generalized adelically. See \cite{bump}.

The study of the $L$-functions proceeds via analytic properties of zeta integrals following Tate's thesis.  
The Jacquet-Langlands global zeta integrals have the form
$$\int_{\I_{\Q}/\Q^*} \varphi(\begin{pmatrix} a & 0\\ 0 & 1\end{pmatrix})|a|^{s-1/2}_{\A_{\Q}} da,$$
where $\varphi$ is any function in the subspace $H_{\pi}$ of $L^2(GL_n(K)\backslash GL_n(\A_K),\chi)$ realizing $\pi$, and $|.|_{\A_{\Q}}^s$ is the adelic absolute value 
$$|x|_{\A_{\Q}}=\prod_{p\in \text{Primes}\cup \{\infty\}} |x(p)|_p,$$
for $x\in \A_K$. One assumes that $\varphi(x)$ is $K$-finite where $K=\prod_{p\in \text{Primes}\cup \{\infty\}} S_p$, and $S_p$ is the maximal compact subgroup of $GL_n(\Q_p)$, in the sense that the right-translates 
of $\varphi$ by elements $k$ in $K$ span a finite-dimensional space of functions.


\medskip

\subsection{\bf Adelic constructible integrals}\label{ssec-cons}

\

\medskip

To capture the essential required model-theoretic properties of these global zeta integrals and propose a framework to study them and pose questions, we define the following integrals. As before 
$K$ is a number field with completions $K_v$.

\begin{Def}\label{con-gen} Let $\cL$ be a language extending $\cL_{rings}$ containing a unary predicate for a multiplicative character $\chi$. 
Let $x$ be an $n$-tuple of variables and 
$$\psi_1(x),\psi_2(x),\dots,\psi_k(x)$$ $\cL$-formulas which give definable functions from $K_v^n$ into $K_v$ for all $v\in V_K^{fin}$. 
Let $\varphi(x)$ be an $\cL$-formula.

An $\cL$-constructible integral of product type on the finite adeles $\A_K^{fin}$ is an Euler product of the form
$$\prod_{v\in V_K^{fin}} \int_{\varphi(K_v)}\Psi(x)_v \Phi_v(x) |\psi_1(x)|_v^s |\psi_2(x)|_v\dots |\psi_k(x)|_v dx$$
where
\begin{enumerate}
\item dx is the normalized Haar measure on $K_v^n$ such that $\int_{\cO_v^n}dx=1$,
\item$\Phi_v$ is a Schwartz-Bruhat function on $\varphi(K_v^n)$ such that $\Phi_v=1$ for all but finitely many $v$. 
\item $\Psi(x_1,\dots,x_n)$ is a function from $\varphi(\A_K)$ into $\C$ that is an Euler product of $\cL$-definable
functions $\Psi_v$ from $\varphi(K_v)$ in $\C$,
\item For every $a\in \varphi(\A_K)$ one has $\Psi_v(a(v))=1$ for all but finitely many $v$,
\item $\chi$ is interpreted as a multiplicative character on $K_v^*$ or in case $v$ is non-Archimedean also as a multiplicative character on the residue field $k_v^*$.
\end{enumerate}
 \end{Def}

{\bf Note.} When $\Psi_v$ is the term $\chi(x)$, we see that Tate's zeta integrals are $\cL$-constructible.

\begin{Def}\label{con-fin} Let $x$ be an $n$-tuple of variables and $\varphi(x)$ an $\cL$-formula. 
Let $\cL$ be a language extending $\cL_{rings}$ containing a unary predicate for a multiplicative character $\chi$. 
Let $\psi_1(x),\psi_2(x),\dots,\psi_k(x)$ be $\cL$-definable functions from $\varphi(\A_K^n)$ into the finite ideles $\I_K^{fin}$. Let $\Phi$ be a 
Schwartz-Bruhat function as in Def \ref{con-gen}.

An $\cL$-constructible integral on the finite adeles $\A_K^{fin}$ is a function of the form
$$\int_{\varphi(\A_K^{fin})}\Psi(x) |\psi_1(x)|_{\A_K^{fin}}^s |\psi_2(x)|_{\A_K^{fin}}\dots |\psi_k(x)|_{\A_K^{fin}} dx,$$
where
\begin{enumerate}
\item $|.|_{\A_K^{fin}}=\prod_{v\in V_K^{fin}} |.|_v$,
\item $dx$ is a normalized Haar measure on $(\A_K^{fin})^n$,
\item $\Psi(x_1,\dots,x_n)$ is a function from $\varphi(\A_K^{fin})$ into $\C$ that is an Euler product of $\cL$-definable
functions $\Psi_v$ from $\varphi(K_v)$ into $\C$,
\item For every $a\in \varphi(\A_K^{fin})$, for all but finitely many $v$, $\Psi_v(a(v))=1$.
\end{enumerate}
If $\varphi(\A_K)$ is a definable set of Type I in the sense of Remark \ref{typeI}, then we add Schwartz-Bruhat functions to the integrals, to get 
$$\int_{\varphi(\A_K^{fin})}\Phi(x) \Psi(x) |\psi_1(x)|_{\A_K^{fin}}^s |\psi_2(x)|_{\A_K^{fin}}\dots |\psi_k(x)|_{\A_K^{fin}} dx,$$
where $\Phi=\otimes_{v\in V_K^{fin}} \Phi_v$ where $\Phi_v$ is a Schwartz-Bruhat function on $\varphi(K_v)$ and $\Phi_v=1$ for all but finitely many $v\in V_K^{fin}$.
 \end{Def}
Recall that definable sets of Type I in the adeles are sets of the form 
$$X=\{a\in \A_K^n: [[\psi(a)]]=1\}$$ for an $L$-formula $\psi$, where $L$ is a language for all the $K_v$. 

A definable set $X$ of Type I can be written as a restricted product of definable sets in $K_v$ with respect to $\varphi(\cO_v)$. In this case an integral (resp. Schwartz-Bruhat function) on $\A_K^m$  
decomposes as an Euler product of integrals (resp. Schwartz-Bruhat functions) over $K_v^m$. 

This is the reason that to add Schwartz-Bruhat functions in Definition \ref{con-fin} we assume $\varphi(\A_K)$ is of Type I 
since for non-Archimedean $K_v$ a construction of Schwartz-Bruhat functions for definable sets can be given using quantifier elimination (a construction is given by Cluckers-Loeser in \cite{CL2}).
\begin{prob}\label{real-sb} Define Schwartz-Bruhat functions for definable sets in $\R^n$.\end{prob} 
By quantifier elimination this reduces to real semi-algebraic sets. One can use real cell decomposition to reduce it to cells.

\begin{note} Solving Problem \ref{real-sb} would extend Definition \ref{con-fin} to the adeles $\A_K$.\end{note}

\begin{prob} Define Schwartz-Bruhat functions for definable sets in $\A_K^m$ for $m\geq$ which are not of Type I.
\end{prob}
\begin{remark} The problem here is that when the definable set 
is defined by conditions of the from $Fin([[...]])$ or $C_j([[...]])$, we do not know that it is a restricted product and we do not know if the Schwartz-Bruhat functions are Euler products of local functions.\end{remark}

\begin{thm}[Berman-Derakhshan-Onn-Paajanen {\cite[Theorem A]{BDOP}}]\label{change-of-meas} Suppose $\varphi$ defines a Chevalley group $G$. Suppose $\mu$ is a Haar measure on $G(F)$ for a non-Archimedean local field $F$. Then there is a finite partition of $G(F)$ into 
$\cL_{rings}$-definable sets such that on a given piece an integral with respect to 
$\mu$ can be written an integral with respect to $|\phi(x)|dx$, where $\phi$ is an $\cL_{rings}$-definable function and $dx$ an 
additive Haar measure on $F^{dim(G)}$. 
If $F$ vary over the non-Archimedean local fields $K_v$, then this "change of measure" is uniform in $v$, i.e. the definable partition of the domain and the definable functions $\phi(x)$ can be chosen independently of $v$.
\end{thm}

We would need to add Archimedean factors to the constructible integrals so that they become related to $L$-functions and have good analytic properties (meromorphic property, Fourier transform). 

\begin{prob}\label{con-fin-comp} Complete the constructible integrals in Definitions \ref{con-gen} and \ref{con-fin} by adding factors for the Archimedean places $v$.
\end{prob}

We propose such a definition for the case $K=\Q$. We let $\Q_{\infty}=\R$. 

\begin{Def}\label{def-sp-ad-con} Let $\cL=(\cL_{real},\cL_{na},\psi(x))$ be a language were $\psi(x)$ is a unary predicate interpreted as a multiplicative character on $\Q_p^*$ for $p\leq \infty$, and 
$\cL_{real}$ and $\cL_{na}$ extend $\cL_{rings}$. Let $x$ be an $n$-tuple of variables. Let $\varphi(x)$ be an $\cL$-formula. 
Let $\phi_1^i(x),\phi_2^i(x),\dots,\phi_k^i(x)$, for $i\in \{0,1\}$, be $\cL$-formulas which when $i=0$ are $\cL_{real}$-formulas and give $\cL_{real}$-definable functions from $\varphi(\R)$ into $\R$, and 
when $i=1$ are $\cL_{na}$-formulas and give $\cL_{na}$-definable functions from $\varphi(\Q_p)$ into $\Q_p$ for $p<\infty$.

A special adelic $\cL$-constructible integral is an Euler product of the form
$$\prod_{p\in \{ \mathrm{Primes}\}\cup \{\infty\}} \lambda_p^{-1}\int_{\varphi(\Q_p)} \Psi_p(x)|\phi_1(x)|^s_p |\phi_2(x)|_p\dots |\phi_k(x)|_p dx$$
where
\begin{enumerate}
\item $dx$ is an additive Haar measure on $\Q_p^n$ normalized such that $\int_{\Z_p^n}dx=1$ for $p<\infty$, 
\item $\lambda_p\in \C$ have the form of a product of convergence factors (to make the Euler product converge) and normalizing factors (to give it a special intended value),
\item For all $p\leq \infty$ the following hold,

(3.1). $\Psi_p$ are functions from $\varphi(\Q_p)$ into $\C$,

(3.2). There are $\cL_{rings}$-formulas $\theta_i^t(x)$, for $i\in \{1,\dots,n\}$ and $t\in \{1,2,3\}$
that give definable functions from $\varphi(\Q_p)$ into $\Q_p$ such that $\theta_i^3(a)\neq 0$ for all $a\in \varphi(\Q_p)$ and all $i$,

(3.3). There are multiplicative characters $\chi_1,\dots,\chi_N$ on $\Q_p^*$ and Schwartz-Bruhat functions $\Phi_1,\dots,\Phi_N$ on $\Q_p$ such that 

(i) For $p<\infty$ and $a\in \varphi(\Q_p)$, 
$$\Psi_p(a)=\sum_{1\leq i\leq N} \ulcorner \Phi_i(\theta_i^1(a)) v(\theta_i^2(a))\chi_i(\theta_i^3(a))\urcorner,$$ 
where the notation $\ulcorner \dots \urcorner$ means that in the summand, one of more of the terms $\Phi_i(\theta_i^1(a))$, $v(\theta_i^2(a))$, or $\chi_i(\theta_i^3(a))$ may not be present.

(ii) For $p=\infty$ and $a\in \varphi(\R)$, there is some element $\gamma \in \R$ (possibly $\gamma=0$) such that as $\theta(a) \rightarrow \gamma$, one has
$$\Psi_p(a)\sim \sum_{1\leq i\leq N} \ulcorner \Phi_i(\theta_i^1(a)) log |\theta_i^2(a)| \chi_i(\theta_i^3(a)).\urcorner$$
The notation $\ulcorner \dots \urcorner$ has the same meaning as in part $(i)$.

\end{enumerate}
\end{Def}
\begin{note} Since $|.|^s$ is a character of $\R^*$, the Archimedean factor of a special adelic constructible function can include terms of the form $|\theta(x)|^s$ where $\theta$ is a non-vanishing 
definable function on a definable set. This is a real analogue of the integrals of Denef and Loeser (see Section \ref{ssec-euler})\end{note}

\begin{Def}\label{def-Whitt} A special adelic $\cL$-constructible function is of Whittaker type if $\Psi_v(a)=0$ when $|\theta(a)|_{\infty}$ is large, and one can choose $\gamma=0$ in $(ii)$.
\end{Def}
\begin{remark} Definitions \ref{con-gen} and \ref{def-sp-ad-con} give a family of definitions, one for each choice for the language $\cL$. An important choice is when $\cL_{real}$ is the language 
of restricted analytic functions with exponentiation defined by van den Dries-Macintyre-Marker \cite{VMM} and $\cL_{na}$ is any of the languages of Belair \cite{Belair}, Basarab \cite{basarab}, or Denef-Pas \cite{pas} from Subsection \ref{ssec-venrich}.
\end{remark}
\begin{remark}\label{rem-real} 
One could also formulate the integrals in Definitions \ref{con-fin} in terms of a volume form $\omega$ on $\varphi(\A_K)$. 
For this one starts with a volume form on $\varphi(K_v)$ which can be constructed as in \cite{CL1} and 
\cite{CL2}. This gives a volume form on $\varphi(\A_K)$ as in Section \ref{ssec-tam}.

If $\varphi(x)$ defines an algebraic group $G(K_v)$ for each $v$, then one can use Theorem \ref{change-of-meas} to reduce integration with respect to a Haar measure on $G(K_v)$ to integration with respect to a suitable 
measure for integrating definable functions. This enables us to reduce integrals of definable functions on $G(\A_K)$ with respect to a Haar measure on $G(\A_K)$ (e.g. Tamagawa measure) to an Euler product of local components which are integrals of function of the form 
$|\psi(x)|^s$, where $\psi$ is a definable function from $K_v^m$ into $K_v$. 

For non-Archimedean $K_v$, these local factors can be evaluated using methods of Denef (see Section \ref{ssec-euler}) and the Euler product an be understood by applying 
Theorem \ref{Thm-zeta}. 

Note that $G$ is a definable set of Type I and functions that are $1$ at almost all places factorize into local functions.
\end{remark}

\begin{note} In Definitions \ref{con-gen}, \ref{con-fin}, and \ref{def-sp-ad-con}, we can require that $\Psi$ is an $\cL_{Boolean}^{fin,res}$-definable function. This is more general than being an Euler product of "definable factors".\end{note}

\begin{note} We can generalize the constructible integrals by replacing $\cL_{rings}$ by the restricted product language $\cL_{Boolean}^{fin,res}(L)$ where $L$ is any of the languages for the factors. This works over definable sets of Type I which are restricted products.\end{note}

\begin{ex}\label{ex-gamma} The characters of the group $\R^{*}$ are of the form $|.|^s$ or $sgn |.|^s$, where $s\in \C^*$ and 
$sgn$ is the sign character $x\rightarrow x/|x|$. Let $\chi=|.|^s$. Let $f(x)=e^{-\pi x^2}$. $f$ is a Schwartz-Bruhat function on $\R$. The function $f(x)\chi(x)$ is part of a 
special adelic constructible integral. Integrating this function gives the $\Gamma$-function.\end{ex}
\begin{proof} The local zeta function in this case becomes 
$$Z(f,\chi)=\int_{\R^*} e^{-\pi x^2} |x|^s d^*x=2\int_{0}^{\infty} e^{-\pi x^2}x^{s-1} dx$$
$$=\pi^{-s/2}\int_{0}^{\infty} e^{-u}u^{s/2-1}du=\pi^{-s/2}\Gamma(s/2),$$
where in the second line we have put $u=\pi x^2$. This calculation is in Tate's thesis \cite{tate-thesis}.
\end{proof}

\begin{thm}\label{ex-jl} Let $\pi$ be an automorphic cuspidal representation with central character $\chi$ Let $\varphi$ be any function in the subspace $H_{\pi}$ of $L^2(GL_n(K)\backslash GL_n(\A_K),\chi)$ realizing $\pi$. 
The Jacquet-Langlands integrals 
$$Z_{JL}(\varphi,s)=\int_{\I_{\Q}/\Q^*} \varphi(\begin{pmatrix} x & 0\\ 0 & 1\end{pmatrix})|x|^{s-1/2}_{\A_{\Q}} d^*x$$
are adelic $\cL_{rings}^{\Psi}$-constructible of Whittaker type.
\end{thm}
\begin{proof} This follows from formulas and calculations of Jacquet and Langlands \cite{jac-lang}. 
see \cite{jac-lang} and \cite{godem-book}. We sketch it below to give a slight model-theoretic interpretation.

Write the Fourier expansion of $\varphi$:
$$\varphi(g)=\sum_{\xi \in \Q^*} W_{\varphi}^{\psi}(\begin{pmatrix} \xi & 0\\ 0 & 1\end{pmatrix}g).$$
Here $\psi$ is a fixed additive character on the adeles $\A_{\Q}$ that is trivial on the global field $\Q$ and 
$W_{\varphi}^{\psi}$ is the $\psi$th-Fourier coefficient of $\varphi$ given by 
$$W_{\varphi}^{\psi}(g)=\int_{\A_{\Q}/\Q}\varphi(\begin{pmatrix} 1 & x\\ 0 & 1\end{pmatrix}g)\overline{\psi(x)}dx.$$
It follows that
$$Z_{JL}(\varphi,s)=\int_{\I_{\Q}}W_{\varphi}^{\psi}(\begin{pmatrix} 1 & x\\ 0 & 1\end{pmatrix})|x|^{s-1/2}_{\A_{\Q}} d^*x.$$
Now 
$$X=\{(x,0,0,1): x\in \I_{\Q}\}$$
is a definable subset of $\A_{\Q}^4$ isomorphic to $\I_{\Q}$. There is a volume form on it which is the pull-back of the volume form $dx^*=\prod_v d^*x_v$ on $\I_{\Q}$. One has the following properties. For
proofs of these see \cite{jac-lang}, \cite{godem-book}, and \cite[pp.11]{GS-book}. 

$(i).$ The function $W_{\varphi}^{\psi}$ (a so-called Whittaker function) is an Euler product
$$W_{\varphi}^{\psi}(g)=\prod_{p\leq \infty} W_p(g(p))$$
where $W_p(g(p))$ are local Whittaker functions and $W_p(x)=1$ for $x\in GL_2(\Z_p)$. 

$(ii).$ The local integrals 
$$\int_{\Q_p^*}W_p(\begin{pmatrix} 1 & x\\ 0 & 1\end{pmatrix})|x|^{s-1/2}_{p} d^* x$$
are absolutely convergent for $\mathrm{Re}(s)$ large, and one has 
$$Z(\varphi,s)=\prod_{p\leq \infty} Z(W_p,s).$$

As stated in \cite[page 11]{GS-book},
 
$(iii).$ For a given $p$, and any $x\in X$, if $|\mathrm{det}(x)|$ is large, then $W_p(x)=0$,

$(iv).$ There are finitely many Schwartz-Bruhat functions $\Phi_1,\dots,\Phi_N$ and functions $c_1,\dots,c_N$ on $\Q_p^*$ 
such that for any $a\in X$
$$W_p(x)=\sum_{1\leq i\leq N} c_i(\mathrm{det}(x))\Phi_i(\mathrm{det}(x)),$$
where $c_i$ are defined to be $\Q$-linear combinations of products of characters $\chi(x)$ with functions of the form 
$v(x)$, where $m\in \Z$, when $p<\infty$, and $\mathrm{log}|x|$ when $p=\infty$.

Since $\mathrm{det}$ is a definable function from $X$ into $\I_{\Q}$, the proof is complete.
\end{proof}

\begin{prob} The proof of Theorem \ref{ex-jl} suggests constructibility of Fourier coefficients of automorphic cusp form. 
Explore this phenomenon.\end{prob}

\begin{prob} Generalize special adelic constructible functions to $\A_K$ for a general number field $K$.\end{prob}

\begin{prob}\label{zeta-ex} Give examples of zeta integrals for $GL_n(\A_K)$ from the work of Godement-Jacquet \cite{jaq-good} that are $\cL$-constructible or special adelic constructible (resp. of Whittaker type) for some $\cL$. What if the group $GL_n$ is replaced by a reductive algebraic group? Can we characterize such $\pi$?\end{prob}

\begin{prob}\label{def-aut} Formulate a notion of automorphic representation for a definable subgroup of $\A_K^m$, $m\geq 1$, generalizing the case of algebraic groups. Define corresponding zeta integrals in analogy with the case of $GL_n$ or 
a reductive group, and explore whether they are $\cL$-constructible for some $\cL$. Is there any notion similar to that of an automorphic representation for a definable set in $\A_K^m$?
\end{prob}
We would like the zeta integrals of an automorphic representation for a definable group $G=\varphi(\A_K)$ to have the form
$$Z(\Phi,s,\psi,\phi_1,\dots,\phi_k,\pi)=
\int_{\varphi(\A_K)} \Phi(x) \pi(x) |\psi(x)|_{\A_K}^s|\phi_1(x)|_{\A_K}\dots |\phi_k(x)|_{\A_K} dx$$
where $\varphi$ is an $\cL$-formula, $\psi$ is an $\cL$-definable function from $\A_K^m$ into 
$\I_K$, $\Phi(x)$ is a Schwartz-Bruhat function, and $dx$ is a Haar measure on $\A_K^m$. 

If $G$ is Type I definable, then $G(\A_K)$ is a restricted product of definable groups over $K_v$ and 
$Z(\Phi,s,\psi,\phi_1,\dots,\phi_k,\pi)$ has an Euler product factorization into local "definable integrals" 
$$Z(\Phi_v,\psi,\phi_1,\dots,\phi_k, \pi_v)=\prod_v \int_{\varphi(K_v)} \Phi_v(x)\pi(x)|\psi(x)|_v^s|\phi_1(x)|_v\dots |\phi_k(x)|_v dx$$ 
where $\Phi_v(x)$ is a local Schwartz-Bruhat function on $\varphi(K_v)$ 
and $\pi_v(x)$ is a local factor of $\pi$ (that should also be defined together with $\pi$). 

\begin{prob}\label{aut-eval} Suppose $\pi$ is defined for a definable set of Type I. Use methods of model theory including cell decomposition, p-adic integration, and resolution of singularities, following Denef-Cluckers-Loeser \cite{DL},\cite{CL1},\cite{CL2} to 
evaluate the local zeta integrals $Z(\Phi_v,\psi,\phi_1,\dots,\phi_k, \pi_v)$. This would give meromorphic continuation beyond abscissa of convergence 
by Theorem \ref{Thm-zeta}.\end{prob}
Note that we assume the set is of Type I to get an Euler product factorization of the zeta integral.
It is an open problem if zeta integrals of general reductive groups have meromorphic continuation (cf. \cite{lang-prob}). Problem \ref{aut-eval} would give a partial solution using Theorem \ref{Thm-zeta}.

\begin{prob}\label{zeta-arch}\noindent\begin{itemize} 
\item Define $L$-functions corresponding to the "definable zeta integrals" above. Define local factors at the real places using $\Gamma$-functions. We would like the Archimedean local factors to be definable in some tame extension of $\R$ possibly related to issues in Hodge theory and the theory of O-minimality related to the work of Bakker-Klingler-Tsimerman \cite{BKT}. 
\item Find a product formula connecting the Archimedean and non-Archimedean local factors.
\end{itemize}\end{prob}

\begin{prob}\label{forms}\noindent\begin{itemize}
\item Use constructible integrals to develop a model theory for modular forms, Maass forms and automorphic forms via adele groups (or more general definable sets) over $\A_K$. 
\item In the case $K=\Q$, we would like to have a $2$-sorted language which has the language of 
van den Dries-Macintyre-Marker \cite{VMM} for restricted analytic functions with exponentiation for the real sort, and a suitable extension on the language of rings or one of the languages in Subsection \ref{ssec-venrich} for the non-Archimedean sort. 
\item Explore connections to $O$-minimality and Hodge theory, via \cite{BKT} by considering adelic versions of the real manifolds and homogeneous space in \cite{BKT} using results on definability of fundamental domains. 
\item If $K$ is a general number field, explore the factors which are $\C$. 
\end{itemize}
\end{prob}


We can also ask a basic question.
\begin{prob}\label{basic} Let $\pi$ be an automorphic cuspidal representation of $GL_n(\A_K)$ and $\Phi(x)$ a Schwartz-Bruhat function for $GL_n(\A_K)$. 
Is there an extension $\cL$ of $\cL_{rings}$ and $\cL$-formulas $\varphi(x)$ and $\psi(x),\phi_1(x),\dots,\phi_k(x)$ which give definable function from $\varphi(\A_K)$ into $\I_K$ 
such that 
$$\int_{\varphi(\A_K)} \pi(x) \Phi(x) |\psi(x)|_{\A_K}^s|\phi_1(x)|_{\A_K}\dots |\phi_k(x)|_{\A_K} dx$$
is $\cL$-adelic constructible?
\end{prob}
\begin{ex} The main example for which Problem \ref{basic} has a positive solution is $GL_2$. The proof of Theorem \ref{ex-jl} gives the required definable sets.
\end{ex}
We define adelic constructible functions.
\begin{Def}\label{con-fn} A function of the form $\pi(x) \Phi(x) |\psi(x)|_{\A_K}^s|\phi_1(x)|_{\A_K}\dots |\phi_k(x)|_{\A_K}$ in Problem \ref{basic} is called adelic $\cL$-constructible.\end{Def}

One the main results in the theory of motivic integration is that the class of motivic constructible functions is closed under integration and Fourier transform, see \cite{CL2}. Our class of adelic constructible functions are closely related. At least over 
Type I definable sets, they are Euler products whose Euler factors are an extension of the $p$-adic specializations of motivic constructible functions. We can thus ask.

\begin{prob}\label{closed} Is the class of adelic constructible functions (resp. special/of Whittaker type) closed under adelic integration and adelic Fourier transform?
\end{prob}

\begin{note} Note that in Definition \ref{con-fn} and Problem \ref{closed}, a special case is when 
we replace $\A_K$ by $\A_K^{fin}$. This is seen by taking $\pi,\Phi,\psi,\phi_j$ to be trivial at the 
Archimedean factors.\end{note}

Once one has defined constructible functions for every number field, from a model-theoretic prospective, it is natural to ask.

\begin{prob} Given an automorphic representation, to what extent its adelic constructible integral representation given by positive solution of Problem \ref{basic} depends on the number field? (e.g. on its degree?). 
If one varies the number field how do they (i.e. their definable functions and formulas) vary?\end{prob}

\section{\bf On identities between adelic integrals}

\medskip

\subsection{\bf Adelic Poisson summation}\label{ssec-poisson}

\

\medskip

In this section as before $K$ is a number field. 
Recall that the space of adelic Schwartz-Bruhat functions $S(\A_K)$ is defined as the restricted product of the Schwartz-Bruhat spaces $S(K_v)$ over all $v\in V_K$. Given $f\in S(\A_K)$ one writes $f(x)=\prod_{v\in V_K} f_v(x(v))$ for $x\in \A_K$. 

Fix a non-trivial additive character $\Psi$ on $\A_K$ such that $\Psi|_{K}=1$. See \cite[Section 7]{ramak} for existence of this. Then the adelic Fourier transform of a function $f\in S(\A_K)$ is defined by 
$$\hat{f}(y)=\int_{\A_K} f(x)\Psi(xy)dx.$$
(usually one has a certain normalization of $dx$, see \cite{ramak}). 

For any $\phi \in S(\A_K)$, to get a function that is invariant under translation by elements of $K$, 
define 
$$\tilde{\phi}(x)=\sum_{\gamma\in K}\phi(\gamma+x).$$ 
If this is convergent for all $x$, then $\tilde{\phi}(x)=\tilde{\phi}(x+\delta)$ for all $\delta \in K$. 

A $\C$-valued function on $\A_K$ is called admissible if $\tilde{f}$ and $\tilde{\hat{f}}$ are both normally convergent on compact sets. Every adelic Schwartz-Bruhat function is admissible (see \cite[Lemma 7-6]{ramak}). 

The adelic Poisson summation formula of 
Tate (see \cite{tate-thesis}, \cite[Theorem 7-7]{ramak}) states that for any $f\in S(\A_K)$, one has $\tilde{f}=\tilde{\hat{f}}$, i.e.
$$\sum_{\gamma \in K} f(\gamma+x)=\sum_{\gamma\in K} \hat{f}(\gamma+x)$$ 
for every $x\in \A_K$.

\begin{prob} Give examples of definable admissible functions in a natural language. Define the notion of admissible for a definable function from $\A_K^m$ into $\C$).\end{prob}
From the Poisson summation formula, Tate derives a Riemann-Roch theorem which states that for an idele $x$ and $f\in S(\A_K)$
$$\sum_{\gamma \in K} f(\gamma x)=\frac{1}{|x|} \sum_{\gamma \in K} \hat{f}(\gamma x^{-1}).$$
See \cite{tate-thesis} and \cite[Theorem 7-10]{ramak}. Tate proves both the Poisson summation formula and the Riemann-Roch theorem for global fields of positive characteristic as well. In this case the adelic Riemann-Roch theorem  implies the usual Riemann-Roch theorem of algebraic geometry. See \cite[Section 7]{ramak}.

To get a model-theoretic hold on the adelic Poisson summation, one would need to know the following which seems plausible.
\begin{prob}\label{prob-poisson} Can the adelic Poisson summation formula or the adelic Riemann-Roch theorem of Tate be interpreted as an identity between adelic $\cL$-constructible functions for some $\cL$? If so, what would be a model-theoretic generalization?
\end{prob}

Laurent Lafforgue has formulated a conjectural non-abelian generalization of Tate's Poisson summation formula on $G(\A_K)$, where $G$ is a reductive algebraic group, and has proved that it is equivalent to the 
Langlands functoriality conjecture (which implies that above conjecture of Langlands). See \cite{laff-nott} and \cite{laff-intro}. 
It would be interesting to study connections to this work.

\medskip

\subsection{\bf Adelic transfer principles}\label{ssec-transfer}

\

\medskip

Model-theoretic transfer principles go back to Tarski who proved that a sentence of the language of rings holds in an algebraically closed field of characteristic zero if and only if it holds in every algebraically closed field of characteristic $p$, for large $p$. See \cite{cherlin}, \cite{KK}. Ax and Kochen gave axiomatization and completeness for theories of $p$-adic fields and Henselian valued fields with characteristic zero residue field \cite{AK1},\cite{AK2},\cite{AK3},\cite{cherlin}. 
From this they deduced that a sentence holds in $\Q_p$ if and only if it holds in $\F_p((t))$, for large $p$. There are also such results for finite extensions of these fields. 

After Denef's work on rationality of "definable" $p$-adic integrals in \cite{Denefrationality}, uniform (in $p$) rationality results were proved by Pas \cite{pas} and Macintyre \cite{Macintyre2}. In \cite{Macintyre2} Macintyre, as a result of his uniform rationality theorem, proved that an identity of "definable integrals" 
$$\int_{\varphi(K)} |f_1(x)|^s|g_1(x)| dx=\int_{\varphi(K)}|f_2(x)|^s|g_2(x)| dx$$
holds for $K=\Q_p$ when $f_i,g_i$ are interpreted in $\Q_p$ if and only if
it holds for $K=\F_p((t))$ when $f_i,g_i$ are interpreted in $\F_p((t))$, for large $p$ (larger than some function of $\varphi,f_i,g_i$).

This was generalized by Denef-Loeser \cite{DL} and Cluckers-Loeser \cite{CL2} to motivic integrals. In \cite{CL2} the motivic integrals are extended to have Schwartz-Bruhat functions and additive characters, including integrals of the form 
$$\int_{\Q_p^n} f(x) \Psi(g(x)) dx,$$
where $f(x)$ is a $p$-adic constructible function in the sense of \cite{CL1},\cite{CL2}, $g(x)$ is a $\Q_p$-valued definable function on $\Q_p^n$, and $\Psi$ is a non-trivial additive character on $\Q_p$, and it is proved that an identity of two such integrals  has a transfer principle between $\Q_p$ and $\F_p((t))$ for large $p$. We note that a $p$-adic constructible function is constructed from functions of the form $v(h(x)),|s(x)|^s$ (cf. \cite{CL2} for details).

The question arises as to whether such a phenomenon is true for the adeles. One way to formulate a question in this connection is the following.

\begin{prob} How does the truth of statements or identities between adelic constructible integrals transfer between different models of the theory of $\A_K$? or between different adele rings $\A_K$ and their ultraproducts 
(what we call pseudo-adelic rings)? Or between adeles of number fields and adeles of function fields?\end{prob}

\medskip

\subsection{\bf A completeness problem}

\

\medskip

The following problem relates to the completeness of the axioms for adeles in \cite{DM-axioms} and in Section \ref{sec-axioms}. It would be interesting to investigate if it follows from completeness of a theory related to that of $\A_K$. 

From another perspective, it can be regarded as an adelic version of a question of Kontsevich and Zagier in \cite{KZ} on periods. A period is a complex number whose real and imaginary parts are values of absolutely convergent integrals of 
rational functions with rational coefficients over semi-algebraic subsets of $\R^n$. The Kontsevich-Zagier question asks assuming two periods are equal, whether the identity of the periods follows from the rules of additivity, change of variables, and 
Newton-Leibniz formula of integrals.
  
The problem we pose is of a model-theoretic nature. I believe that it has implications for some number-theoretic problems formulated in this paper.

\begin{prob} If a statement or an identity of constructible integrals is true in the adeles $\A_K$, is there a proof of it using the Axioms in Section \ref{sec-axioms}?
\end{prob}

\bibliographystyle{acm}
\bibliography{bibadeles}

\begin{thebibliography}{10}

\bibitem{ax}
{\sc Ax, J.}
\newblock The elementary theory of finite fields.
\newblock {\em Ann. of Math. (2) 88\/} (1968), 239--271.

\bibitem{AK1}
{\sc Ax, J., and Kochen, S.}
\newblock Diophantine problems over local fields. {I}.
\newblock {\em Amer. J. Math. 87\/} (1965), 605--630.

\bibitem{AK2}
{\sc Ax, J., and Kochen, S.}
\newblock Diophantine problems over local fields. {II}. {A} complete set of
  axioms for {$p$}-adic number theory.
\newblock {\em Amer. J. Math. 87\/} (1965), 631--648.

\bibitem{AK3}
{\sc Ax, J., and Kochen, S.}
\newblock Diophantine problems over local fields. {III}. {D}ecidable fields.
\newblock {\em Ann. of Math. (2) 83\/} (1966), 437--456.

\bibitem{BKT}
{\sc Bakker, Benjamin, K.~B., and Tsimerman, J.}
\newblock Tame topology of arithmetic quotients and algebraicity of hodge loci.
\newblock {\em arXiv:1810.04801\/} (2018).

\bibitem{basarab}
{\sc Basarab, {\c{S}}.~A.}
\newblock Relative elimination of quantifiers for {H}enselian valued fields.
\newblock {\em Ann. Pure Appl. Logic 53}, 1 (1991), 51--74.

\bibitem{Belair}
{\sc B{\'e}lair, L.}
\newblock Substructures and uniform elimination for {$p$}-adic fields.
\newblock {\em Ann. Pure Appl. Logic 39}, 1 (1988), 1--17.

\bibitem{BDOP}
{\sc Berman, M., Derakhshan, J., Onn, U., and Paajanen, P.}
\newblock Uniform cell decomposition and applications to {C}hevalley groups.
\newblock {\em Journal of London Mathematical Society, 87}, 2 (2013), 586--606.

\bibitem{bump}
{\sc Bump, D.}
\newblock {\em Automorphic forms and representations}, vol.~55 of {\em
  Cambridge Studies in Advanced Mathematics}.
\newblock Cambridge University Press, Cambridge, 1997.

\bibitem{lang-jer}
{\sc Bump, D., Cogdell, J.~W., de~Shalit, E., Gaitsgory, D., Kowalski, E., and
  Kudla, S.~S.}
\newblock {\em An introduction to the {L}anglands program}.
\newblock Birkh\"{a}user Boston, Inc., Boston, MA, 2003.
\newblock Lectures presented at the Hebrew University of Jerusalem, Jerusalem,
  March 12--16, 2001, Edited by Joseph Bernstein and Stephen Gelbart.

\bibitem{Cassels}
{\sc Cassels, J. W.~S.}
\newblock Global fields.
\newblock In {\em Algebraic {N}umber {T}heory ({P}roc. {I}nstructional {C}onf.,
  {B}righton, 1965)}. Thompson, Washington, D.C., 1967, pp.~42--84.

\bibitem{cassels-local}
{\sc Cassels, J. W.~S.}
\newblock {\em Local fields}, vol.~3 of {\em London Mathematical Society
  Student Texts}.
\newblock Cambridge University Press, Cambridge, 1986.

\bibitem{CF}
{\sc Cassels, J. W.~S., and Fr{\"o}hlich, A.}, Eds.
\newblock {\em Algebraic number theory\/} (London, 1986), Academic Press Inc.
  [Harcourt Brace Jovanovich Publishers].
\newblock Reprint of the 1967 original.

\bibitem{CDM}
{\sc Chatzidakis, Z., van~den Dries, L., and Macintyre, A.}
\newblock Definable sets over finite fields.
\newblock {\em J. Reine Angew. Math. 427\/} (1992), 107--135.

\bibitem{cherlin}
{\sc Cherlin, G.}
\newblock {\em Model-Theoretic Algebra}, vol.~521 of {\em Lecture Notes in
  Mathematics}.
\newblock Springer-Verlag, 1976.

\bibitem{cherlin-book}
{\sc Cherlin, G.}
\newblock {\em Model theoretic algebra---selected topics}.
\newblock Lecture Notes in Mathematics, Vol. 521. Springer-Verlag, Berlin-New
  York, 1976.

\bibitem{CH}
{\sc Chernikov, A., and Hils, M.}
\newblock Valued difference fields and ${N}{T}{P}_{2}$.
\newblock arXiv:12081341.

\bibitem{CDLM}
{\sc Cluckers, R., Derakhshan, J., Leenknegt, E., and Macintyre, A.}
\newblock Uniformly defining valuation rings in {H}enselian valued fields with
  finite or pseudo-finite residue fields.
\newblock {\em Ann. Pure Appl. Logic 164}, 12 (2013), 1236--1246.

\bibitem{CL1}
{\sc Cluckers, R., and Loeser, F.}
\newblock Constructible motivic functions and motivic integration.
\newblock {\em Invent. Math. 173\/} (2008), 23--121.

\bibitem{CL2}
{\sc Cluckers, R., and Loeser, F.}
\newblock Constructible exponential functions, motivic fourier transform, and
  transfer princple.
\newblock {\em Ann. Math 171}, 2 (2010), 1011--1065.

\bibitem{connes-selecta}
{\sc Connes, A.}
\newblock Trace formula in noncommutative geometry and the zeros of the
  {R}iemann zeta function.
\newblock {\em Selecta Math. (N.S.) 5}, 1 (1999), 29--106.

\bibitem{CC}
{\sc Connes, A., and Consani, C.}
\newblock The hyperring of ad\`ele classes.
\newblock {\em J. Number Theory 131}, 2 (2011), 159--194.

\bibitem{connes-c-site}
{\sc Connes, A., and Consani, C.}
\newblock The arithmetic site.
\newblock {\em C. R. Math. Acad. Sci. Paris 352}, 12 (2014), 971--975.

\bibitem{CC2}
{\sc Connes, A., and Consani, C.}
\newblock Geometry of the arithmetic site.
\newblock {\em Adv. Math. 291\/} (2016), 274--329.

\bibitem{elem-prod}
{\sc D'Aquino, P., and Macintyre, A.}
\newblock Commutative unital rings elementarily equivalent to prescribed
  product products.
\newblock {\em Preprint\/} (2019).

\bibitem{PDAJM}
{\sc D'Aquino, P., and Macintyre, A.}
\newblock Zilber problem on residue rings of models of arithemtic, part 1: the
  prime power case.
\newblock {\em Preprint\/} (2019).

\bibitem{dpm}
{\sc D'Aquino~Paola, Derakhshan, J., and Macintyre, A.}
\newblock Truncations of ordered abelian groups.
\newblock {\em Preprint\/} (arXiv: 1910/14517).

\bibitem{Denefrationality}
{\sc Denef, J.}
\newblock The rationality of the {P}oincar\'e series associated to the $p$-adic
  points on a variety.
\newblock {\em Invent. Math. 77\/} (1984), 1 -- 23.

\bibitem{DL}
{\sc Denef, J., and Loeser, F.}
\newblock Definable sets, motives, and $p$-adic integrals.
\newblock {\em J. Amer. Math. Soc. 14}, 2 (2001), 429--469.

\bibitem{zeta1}
{\sc Derakhshan, J.}
\newblock Euler products of $p$-adic integrals and definable sets {I}:
  meromorphic continuation and zeta functions of groups.
\newblock {\em In Preparation\/}.

\bibitem{zeta2}
{\sc Derakhshan, J.}
\newblock Euler products of $p$-adic integrals and definable sets {II}: height
  zeta functions and equidistribution of rational points.
\newblock {\em In Preparation\/}.

\bibitem{zeta-surv}
{\sc Derakhshan, J.}
\newblock $p$-adic model theory, $p$-adic integration, {E}uler products, and
  zeta functions of groups.
\newblock {\em this volume\/}.

\bibitem{DM-supp}
{\sc Derakhshan, J., and Macintyre, A.}
\newblock Some supplements to {F}eferman-{V}aught related to the model theory
  of adeles.
\newblock {\em Ann. Pure Appl. Logic 165}, 11 (2014), 1639--1679.

\bibitem{DM-bool}
{\sc Derakhshan, J., and Macintyre, A.}
\newblock Enrichments of {B}oolean algebras by {P}resburger predicates.
\newblock {\em Fund. Math. 239}, 1 (2017), 1--17.

\bibitem{DM-axioms}
{\sc Derakhshan, J., and Macintyre, A.}
\newblock Axioms for {C}ommutative {U}nital {R}ings {E}lementarily {E}quivalent
  to {R}estricted {P}roducts of {C}onnected {R}ings.
\newblock {\em Preprint\/} (2020).

\bibitem{DM-MC}
{\sc Derakhshan, J., and Macintyre, A.}
\newblock Model completeness for {H}enselian valued fields with finite
  ramification valued in a {Z}-group.
\newblock {\em Preprint\/} (arXiv: 1603.08598).

\bibitem{DM-ad}
{\sc Derakhshan, J., and Macintyre, A.}
\newblock Model theory of adeles {I}.

\bibitem{DM-ad2}
{\sc Derakhshan, J., and Macintyre, A.}
\newblock Decidability problems for adele rings and related restricted
  products.
\newblock {\em Preprint\/} (arXiv: 1910.14471).

\bibitem{DM}
{\sc Dougherty, R., and Miller, C.}
\newblock Definable {B}oolean combinations of open sets are {B}oolean
  combinations of open definable sets.
\newblock {\em Illinois J. Math. 45}, 4 (2001), 1347--1350.

\bibitem{Duret}
{\sc Duret, J.-L.}
\newblock Les corps faiblement alg\'ebriquement clos non s\'eparablement clos
  ont la propri\'et\'e d'ind\'ependence.
\newblock In {\em Model theory of algebra and arithmetic ({P}roc. {C}onf.,
  {K}arpacz, 1979)}, vol.~834 of {\em Lecture Notes in Math.} Springer, Berlin,
  1980, pp.~136--162.

\bibitem{enderton-book}
{\sc Enderton, H.~B.}
\newblock {\em A mathematical introduction to logic}.
\newblock Academic Press, New York-London, 1972.

\bibitem{enderton}
{\sc Enderton, H.~B.}
\newblock {\em A mathematical introduction to logic}, second~ed.
\newblock Harcourt/Academic Press, Burlington, MA, 2001.

\bibitem{FV}
{\sc Feferman, S., and Vaught, R.~L.}
\newblock The first order properties of products of algebraic systems.
\newblock {\em Fund. Math. 47\/} (1959), 57--103.

\bibitem{FJ}
{\sc Fried, M., and Jarden, M.}
\newblock {\em Field arithmetic}, vol.~3 of {\em Ergebnisse der Mathematik und
  ihrer Grenzgebiete}.
\newblock Springer, 1986.

\bibitem{fs}
{\sc Fried, M., and Sacerdote, G.}
\newblock Solving {D}iophantine problems over all residue class fields of a
  number field and all finite fields.
\newblock {\em Ann. of Math. (2) 104}, 2 (1976), 203--233.

\bibitem{GS-book}
{\sc Gelbart, S., and Shahidi, F.}
\newblock {\em Analytic properties of automorphic {$L$}-functions}, vol.~6 of
  {\em Perspectives in Mathematics}.
\newblock Academic Press, Inc., Boston, MA, 1988.

\bibitem{godem-book}
{\sc Godement, R.}
\newblock {\em Notes on {J}acquet-{L}anglands' theory}, vol.~8 of {\em CTM.
  Classical Topics in Mathematics}.
\newblock Higher Education Press, Beijing, 2018.
\newblock With commentaries by Robert Langlands and Herve Jacquet.

\bibitem{jaq-good}
{\sc Godement, R., and Jacquet, H.}
\newblock {\em Zeta functions of simple algebras}.
\newblock Lecture Notes in Mathematics, Vol. 260. Springer-Verlag, Berlin-New
  York, 1972.

\bibitem{GO}
{\sc Gorodnik, A., and Oh, H.}
\newblock Rational points on homogeneous varieties and equidistribution of
  adelic periods.
\newblock {\em Geom. Funct. Anal. 21}, 2 (2011), 319--392.
\newblock With an appendix by Mikhail Borovoi.

\bibitem{HW}
{\sc Hardy, G.~H., and Wright, E.~M.}
\newblock {\em An introduction to the theory of numbers}, sixth~ed.
\newblock Oxford University Press, Oxford, 2008.
\newblock Revised by D. R. Heath-Brown and J. H. Silverman, With a foreword by
  Andrew Wiles.

\bibitem{HHM}
{\sc Haskell, D., Hrushovski, E., and Macpherson, D.}
\newblock {\em Stable domination and independence in algebraically closed
  valued fields}, vol.~30 of {\em Lecture Notes in Logic}.
\newblock Association for Symbolic Logic, Chicago, IL, 2008.

\bibitem{udi-char}
{\sc Hrushovski, E.}
\newblock Ax's theorem with an additive character.
\newblock {\em Preprint\/} (arXiv:1911.01096).

\bibitem{HK}
{\sc Hrushovski, E., and Kazhdan, D.}
\newblock Integration in valued fields.
\newblock In {\em Algebraic geometry and number theory}, vol.~253 of {\em
  Progr. Math.} Birkh\"auser Boston, Boston, MA, 2006, pp.~261--405.

\bibitem{HMR}
{\sc Hrushovski, E., Martin, B., and Rideau, S.}
\newblock Definable equivalence relations and zeta functions of groups.
\newblock {\em J. Eur. Math. Soc. (JEMS) 20}, 10 (2018), 2467--2537.
\newblock With an appendix by Raf Cluckers.

\bibitem{udi-pillay-groups}
{\sc Hrushovski, E., and Pillay, A.}
\newblock Groups definable in local fields and pseudo-finite fields.
\newblock {\em Israel J. Math. 85}, 1-3 (1994), 203--262.

\bibitem{iwasawa}
{\sc Iwasawa, K.}
\newblock On the rings of valuation vectors.
\newblock {\em Ann. of Math. (2) 57\/} (1953), 331--356.

\bibitem{jac-lang}
{\sc Jacquet, H., and Langlands, R.~P.}
\newblock {\em Automorphic forms on {${\rm GL}(2)$}}.
\newblock Lecture Notes in Mathematics, Vol. 114. Springer-Verlag, Berlin-New
  York, 1970.

\bibitem{kiefe}
{\sc Kiefe, C.}
\newblock Sets definable over finite fields: their zeta-functions.
\newblock {\em Trans. Amer. Math. Soc. 223\/} (1976), 45--59.

\bibitem{kochen-local}
{\sc Kochen, S.}
\newblock The model theory of local fields.
\newblock In {\em {$1xsy) $} {ISILC} {L}ogic {C}onference ({P}roc. {I}nternat.
  {S}ummer {I}nst. and {L}ogic {C}olloq., {K}iel, 1974)}. Springer, Berlin,
  1975, pp.~384--425. Lecture Notes in Math., Vol. 499.

\bibitem{Koenigsmann}
{\sc Koenigsmann, J.}
\newblock Defining {$\Bbb Z$} in {$\Bbb Q$}.
\newblock {\em Ann. of Math. (2) 183}, 1 (2016), 73--93.

\bibitem{KZ}
{\sc Kontsevich, M., and Zagier, D.}
\newblock Periods.
\newblock In {\em Mathematics unlimited---2001 and beyond}. Springer, Berlin,
  2001, pp.~771--808.

\bibitem{KK}
{\sc Kreisel, G., and Krivine, J.-L.}
\newblock {\em Elements of mathematical logic. {M}odel theory}.
\newblock Studies in Logic and the Foundations of Mathematics. North-Holland
  Publishing Co., Amsterdam, 1967.

\bibitem{kuhlmann}
{\sc Kuhlmann, F.-V.}
\newblock Quantifier elimination for {H}enselian fields relative to additive
  and multiplicative congruences.
\newblock {\em Israel J. Math. 85}, 1-3 (1994), 277--306.

\bibitem{laff-nott}
{\sc Lafforgue, L.}
\newblock Le principe de fonctorialite de langlands comme un probleme de
  generalisation de la loi d addition.

\bibitem{laff-intro}
{\sc Lafforgue, L.}
\newblock Formules de {P}oisson non lin\'{e}aires et principe de
  fonctorialit\'{e} de {L}anglands.
\newblock In {\em Introduction to modern mathematics}, vol.~33 of {\em Adv.
  Lect. Math. (ALM)}. Int. Press, Somerville, MA, 2015, pp.~323--347.

\bibitem{lang-prob}
{\sc Langlands, R.~P.}
\newblock Problems in the theory of automorphic forms.
\newblock In {\em Lectures in modern analysis and applications, {III}}. 1970,
  pp.~18--61. Lecture Notes in Math., Vol. 170.

\bibitem{lang-icm}
{\sc Langlands, R.~P.}
\newblock {$L$}-functions and automorphic representations.
\newblock In {\em Proceedings of the {I}nternational {C}ongress of
  {M}athematicians ({H}elsinki, 1978)\/} (1980), Acad. Sci. Fennica, Helsinki,
  pp.~165--175.

\bibitem{Macintyre1}
{\sc Macintyre, A.}
\newblock On definable subsets of {$p$}-adic fields.
\newblock {\em J. Symbolic Logic 41}, 3 (1976), 605--610.

\bibitem{Macintyre2}
{\sc Macintyre, A.}
\newblock Rationality of {$p$}-adic {P}oincar\'e series: uniformity in {$p$}.
\newblock {\em Ann. Pure Appl. Logic 49}, 1 (1990), 31--74.

\bibitem{manin-book}
{\sc Manin, Y.~I., and Panchishkin, A.~A.}
\newblock {\em Introduction to modern number theory}, second~ed., vol.~49 of
  {\em Encyclopaedia of Mathematical Sciences}.
\newblock Springer-Verlag, Berlin, 2005.
\newblock Fundamental problems, ideas and theories, Translated from the
  Russian.

\bibitem{nkrch}
{\sc Neukirch, J.}
\newblock {\em Algebraic number theory}, vol.~322 of {\em Grundlehren der
  Mathematischen Wissenschaften [Fundamental Principles of Mathematical
  Sciences]}.
\newblock Springer-Verlag, Berlin, 1999.
\newblock Translated from the 1992 German original and with a note by Norbert
  Schappacher, With a foreword by G. Harder.

\bibitem{ono-int}
{\sc Ono, T.}
\newblock An integral attached to a hypersurface.
\newblock {\em Amer. J. Math. 90\/} (1968), 1224--1236.

\bibitem{pas}
{\sc Pas, J.}
\newblock Uniform {$p$}-adic cell decomposition and local zeta functions.
\newblock {\em J. Reine Angew. Math. 399\/} (1989), 137--172.

\bibitem{pas2}
{\sc Pas, J.}
\newblock Cell decomposition and local zeta functions in a tower of unramified
  extensions of a {$p$}-adic field.
\newblock {\em Proc. London Math. Soc. (3) 60}, 1 (1990), 37--67.

\bibitem{perlis}
{\sc Perlis, R.}
\newblock On the equation {$\zeta _{K}(s)=\zeta _{K'}(s)$}.
\newblock {\em J. Number Theory 9}, 3 (1977), 342--360.

\bibitem{Platonov-R-book}
{\sc Platonov, V., and Rapinchuk, A.}
\newblock {\em Algebraic groups and number theory}, vol.~139 of {\em Pure and
  Applied Mathematics}.
\newblock Academic Press, Inc., Boston, MA, 1994.
\newblock Translated from the 1991 Russian original by Rachel Rowen.

\bibitem{PR-book}
{\sc Prestel, A., and Roquette, P.}
\newblock {\em Formally {$p$}-adic fields}, vol.~1050 of {\em Lecture Notes in
  Mathematics}.
\newblock Springer-Verlag, Berlin, 1984.

\bibitem{ramak}
{\sc Ramakrishnan, D., and Valenza, R.~J.}
\newblock {\em Fourier analysis on number fields}, vol.~186 of {\em Graduate
  Texts in Mathematics}.
\newblock Springer-Verlag, New York, 1999.

\bibitem{roquette-book}
{\sc Roquette, P.}
\newblock {\em The {R}iemann hypothesis in characteristic {$p$} in historical
  perspective}, vol.~2222 of {\em Lecture Notes in Mathematics}.
\newblock Springer, Cham, 2018.
\newblock History of Mathematics Subseries.

\bibitem{peter-nicolas}
{\sc Sarnak, P., Shin, S.-W., and Templier, N.}
\newblock Families of l-functions and their symmetry.
\newblock {\em Preprint\/} (aXiv:1401.5507).

\bibitem{serre-book}
{\sc Serre, J.-P.}
\newblock {\em Lectures on {$N_X (p)$}}, vol.~11 of {\em Chapman \& Hall/CRC
  Research Notes in Mathematics}.
\newblock CRC Press, Boca Raton, FL, 2012.

\bibitem{tate-thesis}
{\sc Tate, J.~T.}
\newblock Fourier analysis in number fields, and {H}ecke's zeta-functions.
\newblock In {\em Algebraic {N}umber {T}heory ({P}roc. {I}nstructional {C}onf.,
  {B}righton, 1965)}. Thompson, Washington, D.C., 1967, pp.~305--347.

\bibitem{uchida}
{\sc Uchida, K.}
\newblock Isomorphisms of {G}alois groups.
\newblock {\em J. Math. Soc. Japan 28}, 4 (1976), 617--620.

\bibitem{VMM}
{\sc van~den Dries, L., Macintyre, A., and Marker, D.}
\newblock The elementary theory of restricted analytic fields with
  exponentiation.
\newblock {\em Ann. of Math. (2) 140}, 1 (1994), 183--205.

\bibitem{weil-adeles-gps}
{\sc Weil, A.}
\newblock {\em Adeles and algebraic groups}, vol.~23 of {\em Progress in
  Mathematics}.
\newblock Birkh\"{a}user, Boston, Mass., 1982.
\newblock With appendices by M. Demazure and Takashi Ono.

\bibitem{weisp-hab}
{\sc Weispfenning, V.}
\newblock Model theory of lattice products.
\newblock {\em Habilitation, Universitat Heidelberg\/} (1978).

\bibitem{Weisp2}
{\sc Weispfenning, V.}
\newblock Quantifier elimination and decision procedures for valued fields.
\newblock In {\em Models and sets ({A}achen, 1983)}, vol.~1103 of {\em Lecture
  Notes in Math.} Springer, Berlin, 1984, pp.~419--472.

\end{thebibliography}

\end{document}